\documentclass[
  a4paper, 
  reqno, 
  oneside, 
  12pt
]{amsart}

\usepackage[utf8]{inputenc}

\usepackage[dvipsnames]{xcolor} 
\colorlet{cite}{red}
\usepackage{tikz}  

\usetikzlibrary{arrows,positioning}
\tikzset{ 
  baseline=-2.3pt,
  text height=1.5ex, text depth=0.25ex,
  >=stealth,
  node distance=2cm,
  mid/.style={fill=white,inner sep=2.5pt},
}

\usepackage{lmodern}
\usepackage{tikz-cd}
\usepackage{amsthm, amssymb, amsfonts}

\usepackage{amsthm, amssymb, amsfonts}	
\usepackage[all]{xy}
\usepackage{graphicx,caption,subcaption}
\usepackage{braket}  
\usepackage{microtype}
\usepackage{comment}
\usepackage{DotArrow}
\usepackage[makeroom]{cancel}
\usepackage[%
  bookmarks=true,			
  unicode=true,			
  pdftitle={Morse theory on Lie groupoids I},		%
  pdfauthor={Ortiz}{Valencia},	%
  pdfkeywords={Morse theory, groupoids},	
  colorlinks=true,		
  linkcolor=black,			
  citecolor=black,		
  filecolor=magenta,		
  urlcolor=RoyalBlue			
]{hyperref}				
\usepackage{cleveref}

\newtheoremstyle{mydef}
  {}		
  {}		
  {}		
  {}		
  {\scshape}	
  {. }		
  { }		
  {\thmname{#1}\thmnumber{ #2}\thmnote{ #3}}	

\newtheorem{theorem}{Theorem}[section]
\newtheorem*{theorem*}{Theorem}
\newtheorem{proposition}[theorem]{Proposition}
\newtheorem*{proposition*}{Proposition}
\newtheorem{lemma}[theorem]{Lemma}
\newtheorem*{lemma*}{Lemma}
\newtheorem{corollary}[theorem]{Corollary}
\newtheorem*{corollary*}{Corollary}
\theoremstyle{definition}
\newtheorem{definition}[theorem]{Definition}
\newtheorem{example}[theorem]{Example}
\newtheorem{assumption}[theorem]{Assumption}

\theoremstyle{remark}
\newtheorem{remark}[theorem]{Remark}

\newtheorem*{conjecture*}{Conjecture}
\usepackage{tikz}

\newcommand{\rr}{\rightrightarrows}

\author{Cristian Ortiz and Fabricio Valencia}
\address{}
\date{\today}

\address{C. Ortiz and F. Valencia - Instituto de Matem\'atica e Estat\'istica, Universidade de S\~ao Paulo, Rua do Mat\~ao 1010, Cidade Universit\'aria, 05508-090 S\~ao Paulo - Brasil. \newline  
      \phantom{xx}
  cortiz@ime.usp.br, fabricio.valencia@ime.usp.br}

\title{Morse theory on Lie groupoids }

\begin{document}
\maketitle

\begin{abstract}

In this paper we introduce Morse Lie groupoid morphisms and study their main properties. We show that this notion is Morita invariant which gives rise to a well defined notion of Morse function on differentiable stacks. We show a groupoid version of the Morse lemma which is used to describe the topological behavior of the critical subgroupoid levels of a Morse Lie groupoid morphism around its nondegenerate critical orbits. We also prove Morse type inequalities for certain separated differentiable stacks and construct a Morse double complex whose total cohomology is isomorphic to the Bott-Shulman-Stasheff cohomology of the underlying Lie groupoid. We provide several examples and applications.

\end{abstract}

\tableofcontents
\section{Introduction}

This is a paper devoted to the study of Morse theory on Lie groupoids and their differentiable stacks. Classical Morse theory is a powerful tool which allows us to extract both geometric and topological information of a manifold equipped with a Morse function, that is, a smooth real valued function on the manifold all of whose critical points are non-degenerate.  Morse functions admit a local model around any critical point, this is the content of the Morse Lemma, which establishes that on a neighborhood of any critical point $x\in M$ a Morse function $f:M\to \mathbb{R}$ looks like a quadratic form $Q_f(x)=-(x^2_1+\cdots+x^2_{\lambda})+x^2_{\lambda+1}+\cdots+x^2_n$. The integer $\lambda$ is called the \textbf{index} of the critical point $x$ of $f:M\to \mathbb{R}$. As a consequence of the Morse Lemma one concludes that non-degenerate critical points are isolated. Some of the main results of Morse theory are: the fundamental theorem of Morse theory which describes how the topology of critical sub-levels changes after crossing a critical level; the Morse inequalities give a relation between the Betti numbers of the manifold and the alternate sum of the number of critical points grouped by their index; the Morse complex which computes the homology of the manifold. Morse theory has led to several interesting geometric results, including: the existence of closed geodesics on a Riemannian manifold as well as an infinite dimensional version of the Morse complex which led to Floer homology.

When dealing with manifolds equipped with a Lie group action, classical Morse theory is no longer the right setting to extract topological information out of an invariant Morse function. This is due to the fact that in equivariant Morse theory, critical points come in orbits hence they are no longer non-degenerate. The right framework to study an invariant Morse function on a manifold is given by Morse-Bott theory. Here, critical points are arranged in families of submanifolds which are non-degenerate in the sense that the normal Hessian is non-degenerate. Among the several applications of Morse-Bott theory one finds: the computation of the cohomology of complex Grassmannians, Bott periodicity theorem or the Atiyah-Guillemin-Sternberg theorem on the convexity of the image of moment maps for torus actions on symplectic manifolds.

In many situations a manifold comes with a Morse-Bott function whose critical submanifolds are not necessarily given by the orbits of a Lie group action but they still come from certain symmetries of the manifold.  In this work we are interested in the case of Morse-Bott functions on a manifold whose critical submanifolds are given by the orbits of a \emph{Lie groupoid}. 

Morse theory on singular spaces given by the orbit space of a Lie groupoid has been studied by several authors. In this regard, in \cite{LT} Lerman and Tolman  study torus actions on symplectic orbifolds, for which they proved some results on Morse-Bott theory on orbifolds. Similarly, in \cite{H} Hepworth studies Morse theory on differentiable Deligne-Mumford stacks, showing for instance the Morse inequalities for orbifolds. Also, Cho and Hong introduce the Morse-Smale-Witten complex for orbifolds \cite{ChoHong}. Another approach known in the literature suitable for studying Morse theory on certain singular spaces is provided by the stratified Morse theory of Goresky and MacPherson in \cite{GM}.

Recently, Lie groupoids equipped with geometric structures suitably compatible with Morita equivalence have been object of intense research. This is due to the fact that such structures descend to the quotient stack of a Lie groupoid. For instance, in \cite{dho} the notion of Morita equivalence of VB-groupoids plays a role to define vector bundles over differentiable stacks; Riemannian metrics on Lie groupoids were introduced in \cite{dHF} showing that they behave well with respect to Morita equivalence, hence inducing a Riemannian metric on the associated quotient stack \cite{dHF2}; Lie algebroids over stacks are modeled by LA-groupoids \cite{waldron}, the space of multiplicative sections of an LA-groupoid  \cite{OW} has the structure of a Lie 2-algebra which is Morita invariant and, in particular, the space of vector fields on a differentiable stack has a natural structure of Lie 2-algebra \cite{bel,OW} as conjectured in \cite{H2}. Also, invariants of differentiable stacks can be defined in terms of their groupoid counterparts. As an example, the equivariant cohomology of a Lie group acting on a differentiable stack was introduced in \cite{barbosaneumann} by looking at the Bott-Shulmann cohomology of a certain action groupoid encoded by an equivariant atlas of a $G$-stack. Another invariant of a differentiable stack is introduced in \cite{dhos} by means of the cohomology of a Lie groupoid with coefficients in a representation up to homotopy of Lie groupoids.

This paper is concerned with Morse theory on differentiable stacks by looking at its groupoid counterpart. For that, we introduce Morse theory for Lie groupoid morphisms $F:G\to \mathbb{R}$ with values in the unital Lie groupoid $\mathbb{R}\rr \mathbb{R}$. Such morphisms are completely determined by basic functions $f\in C^{\infty}(M)^G$, that is, functions satisfying $s^*f=t^*f$ and hence constant along the groupoid orbits. Note that $C^{\infty}(M)^G$ is clearly Morita invariant since it coincides with the $0^{th}$-degree groupoid cohomology $H^0_{diff}(G)$. Hence one thinks of either a Lie groupoid morphism $G\to \mathbb{R}$ or its corresponding basic function, as a function on the quotient stack $[M/G]$.

The paper is organized as follows. In Section \ref{S:2} we review the basics on Morse-Bott theory, Lie groupoids and Riemannian metrics on Lie groupoids as introduced in \cite{dHF}. In Section \ref{S:3} we introduce the main objects of study of the paper, namely Morse Lie groupoid morphisms. We study their main properties and we show Proposition \ref{MoritaInvariance} which establishes that the property of being a Morse Lie groupoid morphism is Morita invariant. In Section \ref{S:4} we give several examples of Morse Lie groupoid morphisms and we show an existence result for proper Lie groupoids whose canonical projection onto their orbit space is proper, this is the content of Theorem \ref{ExistenceProper}. Then we recall the notion of Hamiltonian actions of Lie 2-groups on $0$-symplectic groupoids \cite{hsz} and we show that in the case of a foliation Lie 2-group whose base is a torus the components of the moment map are Morse--Bott Lie groupoid morphisms. In Section \ref{S:5} by using Morita invariance and the equivariant Morse lemma we show Theorem \ref{LieGroupoidMorseLemma} which is our version of the Morse Lemma in the Lie groupoid setting. It is worth mentioning that such a result can also be obtained by means of tubular neighborhoods given by Euler-like vector fields \cite{BLM}. In particular, we introduce a notion of index of a non-degenerate critical subgroupoid as well as their positive and negative normal bundles. We verify that the Hessian of a basic function along an orbit is invariant by the normal representation, thus obtaining canonical group actions of the isotropies on the negative normal bundles as well as on their corresponding disk and sphere bundles when a $2$-metric in the sense of \cite{dHF} comes into the picture, see Proposition \ref{NormalActionFeatures}. In Section \ref{S:6} we study gradient vector fields of real valued Lie groupoid morphisms with respect to a 2-metric. We show Proposition \ref{Multiplicative-Gradient} which says the such gradient vector fields are multiplicative. We also prove Proposition \ref{AttachingDiskProposition} which gives an attaching construction of Lie groupoids whose outcome is a topological groupoid. As a consequence, we show Proposition \ref{NonCriticalLevel} and Theorem  \ref{CriticalLevel} which describe critical sub-levels of Morse Lie groupoid morphisms in terms of attaching groupoids. In Section \ref{S:7} we extend all the results mentioned above to the differentiable stack context. It is worth mentioning that the fact that our notion of Morse Lie groupoid morphism is Morita invariant makes the passage clearer. As an interesting consequence, we get Morse-like inequalities for certain separated differentiable stacks, see Theorem \ref{MorseInequalities}. We also establish a bridge between our approach to study Morse theory over the orbit space of a proper Lie groupoid and the approach provided by stratified Morse theory, thus exposing why our focus becomes more suitable, natural and cleaner for our purposes. This is the content of Subsection \ref{S:7.2}. In Section \ref{S:8} we extend the classical results of Morse-Smale dynamics to the framework of Lie groupoids. We show Theorem \ref{Stable/unstableGroupoidTheorem} which can be thought of as the groupoid version of the so-called stable/unstable manifold theorem and Proposition \ref{ModuliGroupoid} which says that the moduli space of gradient flow lines has a natural structure of Lie groupoid. In Section \ref{S:9} we construct a double complex which is the Lie groupoid analogue of the Morse-Bott complex defined by Austin and Braam in \cite{AB}. The main result of this section is Theorem \ref{MorseCohomology} which shows that the total cohomology of the groupoid double Morse complex is isomorphic to the Bott-Shulman-Stasheff cohomology of the underlying Lie groupoid. In particular, the cohomology of the double groupoid Morse complex is Morita invariant. We finish this section by introducing an equivariant groupoid Morse double complex with respect to an action of a Lie 2-group. We prove Proposition \ref{EquivariantMorseCohomology} which establishes that the equivariant cohomology of a Lie 2-group action as defined in \cite{OBT} can be computed by means of the cohomology of the equivariant groupoid Morse double complex. We use this result to describe the equivariant cohomology of toric symplectic stacks by means of groupoid Morse theoretical tools. 

\vspace{.2cm}
{\bf Acknowledgments:} The authors have benefited from several conversations with Matias del Hoyo, Mateus de Melo, Jo\~ao Nuno Mestre and Luca Vitagliano, we are very thankful for all their comments, corrections and suggestions that improved the first version of this work. We are indebted to M. del Hoyo for pointing out a mistake and some inaccuracies in the first version of Theorem \ref{ExistenceProper} and for suggesting a different proof of Theorem \ref{LieGroupoidMorseLemma}. We also would like to thank Rui Fernandes and Maarten Mol for useful comments. Ortiz was partially supported by National Council of Research and Development CNPq-Brazil, Bolsa de Produtividade em Pesquisa Grant 315502/2020-7 and by Grant 2022/09234-3 Sao Paulo Research Foundation - FAPESP. Valencia was supported by Grants 2020/07704-7 and 2022/11994-6 Sao Paulo Research Foundation - FAPESP.

\section{Preliminaries}\label{S:2}

In this section we briefly introduce the basic notions and results on both Morse--Bott theory and Lie groupoids which will be used throughout this paper. Much of the classical results about Morse theory may be found for instance in \cite{AB,Bo,F,Mi,Ni}. For the general notions regarding Lie groupoids, Morita equivalences and Riemannian groupoids we follow \cite{dH,dHF,MM} closely.

\subsection{Morse--Bott functions}

Let $f:M\to \mathbb{R}$ be a smooth function such that $\textnormal{Crit}(f):=\lbrace x\in M:\ df(x)=0\rbrace$ contains a submanifold $C$ of positive dimension. The choice of a Riemannian metric on $M$ yields a decomposition

$$TM|_C=TC\oplus \nu(C),$$
where $TC$ and $\nu(C)$ are the tangent bundle and the normal bundle of $C$, respectively. Let $\mathcal{H}_x(f)$ be the Hessian of $f$ at $x\in C$, then $T_xC\subseteq \ker(\mathcal{H}_x(f))$. Indeed, if $v,w\in T_xC$ and $\tilde{w}\in\mathfrak{X}(M)$ is any extension of $w$, then

$$\mathcal{H}_x(f)(v,w)=v(\widetilde{w}\cdot f)=0,$$
since $df(\widetilde{w})|_C=0$ because $C\subseteq \textnormal{Crit}(f)$. Therefore, the Hessian $\mathcal{H}_x(f)$ induces a well defined symmetric bilinear form on $\nu_x(C)$ referred to as the \textbf{normal Hessian} of $f$ at $x\in C$ and which we shall denote again as $\mathcal{H}_x(f)$ only if there is no risk of confusion. A critical submanifold $C\subseteq M$ of $f$ is called \textbf{non-degenerate} if the normal Hessian at every $x\in C$ is non-degenerate. This is equivalent to asking $\ker(\mathcal{H}_x(f))=T_xC$ for every $x\in C$.

\begin{definition}
A smooth function $f:M\to \mathbb{R}$ is said to be \textbf{Morse--Bott} if $\textnormal{Crit}(f)$ is a disjoint union of connected submanifolds which are non-degenerate.
\end{definition}

Examples of Morse--Bott functions include: usual Morse functions \cite{Mi,Ni}, invariant smooth functions by the action of a compact Lie group which have non-degenerate critical orbits \cite{W}, component functions of moment maps associated to Hamiltonian torus actions on (pre)symplectic manifolds \cite[s. 3.5]{Ni}, \cite{LS}, among others.

If $C$ is a connected non-degenerate critical submanifold for $f$ then we may define a function $Q_f:\nu(C)\to \mathbb{R}$ which is quadratic along the fibers. Namely,
\begin{equation}\label{eq:quadratic}
Q_f(v)=\dfrac{1}{2}\mathcal{H}_{\pi(v)}(f)(v,v),\qquad v\in \nu(C),
\end{equation}
where $\pi:\nu(C)\to C$ is the bundle projection. The Morse-Bott lemma establishes that there exist neighborhoods $U\subseteq M$ of $C$ and $V\subseteq \nu(C)$ of the zero section $C$, together with a diffeomorphism $\phi:V\to U$ with 

\begin{equation}\label{eq:MBlemma}
\phi|_C=\textnormal{id}\quad \text{ and }\quad  \phi^*f=Q_f.
\end{equation}

There is a splitting  $\nu(C)=\nu_+(C)\oplus \nu_{-}(C)$ where $\nu_+(C)$ and $\nu_{-}(C)$ are spanned by the eigenvectors of $H_f$ with positive and negative eigenvalues, respectively. The \textbf{index} of the critical submanifold $C\subseteq M$ is defined as 

\begin{equation}\label{eq:indexMB}
\lambda(C,f):=\mathrm{rk}(\nu_{-}(C)),
\end{equation}
the rank of $\nu_{-}(C)$.

\begin{remark}
One of the advantages of working with Morse--Bott functions is that they allow us to get similar results to those obtained with usual Morse functions but without assuming that their critical point set is formed by isolated points. For instance: we have for them a similar local linear representation provided by the so-called Morse--Bott lemma \cite{BH1}, they describe very well the topological behavior of a manifold around a non-degenerate critical submanifold \cite{Bo}, and using Morse--Bott--Smale dynamics it is possible to recover the de Rham cohomology of a compact oriented manifold by means of a cochain complex which is defined in terms of the de Rham complex of the critical point sets and gradient flow line spaces \cite{AB}. The latter fact yields a way to obtain the Morse--Bott inequalities.
\end{remark}
\begin{remark}\label{rmk:Hessianpullback} For future purposes it is worth mentioning that if $\pi:M\to N$ is a surjective submersion and $f:N\to \mathbb{R}$ is a Morse--Bott function then $f\circ \pi:M\to \mathbb{R}$ is also a Morse--Bott function since the formula
	\begin{equation}\label{eq:Hessianpullback}
		\mathcal{H}_{x}(f\circ \pi)=d\pi(x)^T\cdot \mathcal{H}_{\pi(x)}(f)\cdot d\pi(x), 
	\end{equation}
	is satisfied at every critical point $x$ of $f\circ \pi$.
\end{remark}

\subsection{Lie groupoids}

 A \textbf{Lie groupoid} $G\rightrightarrows M$ consists of a manifold $M$ of objects and a manifold $G$ of arrows, two surjective submersions $s,t:G\to M$ respectively indicating the source and the target of the arrows, and a smooth associative composition $m:G^{(2)}\to G$ over the set of composable arrows $G^{(2)}=G\times_M G$, admitting unit $u:M\to G$ and inverse $i:G\to G$, subject
to the usual groupoid axioms. The collection of maps mentioned above are called \textbf{structural maps} of the Lie groupoid.

Special instances of Lie groupoids are given by manifolds, Lie groups, Lie group actions, surjective submersions, foliations, pseudogroups, principal bundles, vector bundles, among others. For specific details the reader is recommend to visit the references \cite{dH, mackbook,MM}.

Let us now describe some features concerning the structure of a Lie groupoid. Let $G\rightrightarrows M$ be a Lie groupoid. For each $x\in M$, its \textbf{isotropy group} $G_x:=s^{-1}(x)\cap t^{-1}(x)$ is a Lie group and an embedded submanifold in $G$. There is an equivalence relation on $M$ defined by $x\sim y$ if there exists $g\in G$ with $s(g)=x$ and $t(g)=y$. The corresponding equivalence class of $x\in M$ is denoted by $\mathcal{O}_x\subseteq M$ and called the \textbf{orbit} of $x$. The previous equivalence relation defines a quotient space $M/G$ called the \textbf{orbit space} of $G\rightrightarrows M$.  This space equipped with the quotient topology is in general a \emph{singular space}, that is, it does not carry a differentiable structure making the quotient projection $M\to M/G$ a surjective submersion.

Given a Lie groupoid $G\rr M$, its \textbf{tangent groupoid} is the Lie groupoid $TG\rightrightarrows TM$ obtained by applying the tangent functor to each of its structural maps. If $S \subset M$ is a saturated submanifold, i.e. it is given by the union of orbits, then we can restrict the groupoid structure to $G_{S}=s^{-1}(S)=t^{-1}(S)$, thus obtaining a Lie subgroupoid $G_S\rightrightarrows S$ of $G\rightrightarrows M$. Furthermore, the Lie groupoid structure of $TG\rightrightarrows TM$ induces a Lie groupoid $\nu(G_S)\rightrightarrows \nu(S)$ on the normal bundles, having the property that all of its structural maps are fiberwise isomorphisms. In particular, if $S=\mathcal{O}$ is any orbit, the source map of the normal Lie groupoid $\nu(G_{\mathcal{O}})\rr \nu(\mathcal{O})$ yields a vector bundle isomorphism $\overline{ds}:\nu(G_{\mathcal{O}})\to s^*\nu(\mathcal{O})$. In particular, 

\begin{equation}\label{eq:normalrep}
\overline{dt}\circ \overline{ds}^{-1}:s^*\nu(\mathcal{O})\to t^*\nu(\mathcal{O}),
\end{equation}

\noindent defines a representation of $G_{\mathcal{O}}\rr \mathcal{O}$ on the normal bundle $\nu(\mathcal{O})$. As a consequence, for every $x \in M$ the isotropy group $G_x$ has a canonical representation on the normal fiber $\nu_x:=\nu_x(\mathcal{O}_x)$, called the \textbf{normal representation} of $G_x$ on the normal direction.

\begin{definition}
A \textbf{Lie groupoid morphism} between $G\rightrightarrows M$ and $G'\rightrightarrows M'$ is a pair $\phi:=(\phi^1,\phi^0)$ where $\phi^1:G\to G'$ and $\phi^0:M\to M'$ are smooth functions commuting with both source and target maps and preserving the composition maps. 
\end{definition}

A \textbf{Morita map} is a groupoid morphism $\phi:(G\rightrightarrows M) \to (G'\rightrightarrows M')$ which is \textbf{fully faithful} and \textbf{essentially surjective}, in the sense that the source/target maps define a fibred product of manifolds $G \cong (M \times M) \times_{ (M'\times M')} G'$ and that the map $G' \times _{M'}M \to M$ sending $(\phi^0(x)\to y)\mapsto y$ is a surjective submersion, see \cite{dH, MM}. An important fact shown in \cite{dH} is that a Lie groupoid morphism is a Morita map if and only if it yields an isomorphism between \emph{transversal data}. That is, the morphism must induce: a homeomorphism between the orbit spaces, a Lie group isomorphism $G_x\cong G'_{\phi^0(x)}$ between the isotropies and isomorphisms between the normal representations $G_x\curvearrowright \nu_x\to G'_{\phi^0(x)}\curvearrowright \nu'_{\phi^0(x)}$.

\begin{definition}
Let $G\rightrightarrows M$ and $G'\rightrightarrows M'$ be Lie groupoids. We say that $G$ and $G'$ are \textbf{Morita equivalent} if there exists a third Lie groupoid $H\rightrightarrows N$ with Morita maps $H\to G$ and $H\to G'$.
\end{definition}

It is well known that a Morita equivalence can be always realized by \textbf{Morita fibrations}, that is, Morita maps covering a surjective submersion on objects. For more details see \cite{dH, MM}.

\subsection{Riemannian groupoids}\label{sec:riemanniangroupoids}

Here we briefly recall the notion of Riemannian metric on a Lie groupoid introduced in \cite{dHF}. Such a notion of Riemannian metric is compatible with the groupoid composition so that it plays an important role in our work. We start by recalling that a submersion $\pi:(E,\eta^E)\to B$ with $(E,\eta^E)$ a Riemannian manifold is said to be \textbf{Riemannian} if the fibers of it are equidistant (transverse condition). In this case the base $B$ inherits a metric $\eta^B$ for which the linear map $d\pi(e):(\textnormal{ker}(d\pi(e)))^\perp\to T_{\pi(e)}B$ is an isometry for all $e\in E$. If $(\eta^{E})^\ast$ denotes the dual metric associated to $\eta^{E}$ then the condition for a Riemannian submersion can be rephrased as follows. For all $e\in E$ the map  $d\pi(e)^\ast: T_{\pi(e)}^\ast B \to \textnormal{ker}(d\pi(e))^\circ$ is an isometry, where $\textnormal{ker}(d\pi(e))^\circ$ denotes the annihilator of the vectors tangent to the fiber. If $\pi:E\to B$ is a surjective submersion then a Riemannian metric $\eta^E$ on $E$ is said to be \textbf{transverse} to $\pi$ if for all $x\in B$ and all $e_1,e_2\in \pi^{-1}(x)$ we have that the map
$$d\pi(e_1)^\ast\circ (d\pi(e_2)^\ast)^{-1}:\textnormal{ker}(d\pi(e_2))^\circ\to T_x^\ast B\to  \textnormal{ker}(d\pi(e_1))^\circ,$$
is a linear isometry. In this case, there exists a unique metric $\eta^B$ on $B$ such that $\pi$ becomes a Riemannian submersion. Such a metric is defined by the expression $\eta^B(d\pi(v), d\pi(w)) := \eta^E(v, w)$ for $v, w \in \textnormal{ker}(d\pi)^\perp$ and is called the \textbf{push-forward metric}. The notation we shall be using for the previous Riemannian metric is $\eta^B:=\pi_\ast \eta^E$.

It is well known that given a Lie groupoid $G\rightrightarrows M$ every pair of composable arrows in $G^{(2)}$ may be identified with an element in the space of commutative triangles so that it admits an action of the symmetric group $S_3$ determined by permuting the vertices of such triangles. In these terms, a \textbf{Riemannian groupoid} is a pair $(G\rightrightarrows M,\eta)$ where $G\rightrightarrows M$ is a Lie groupoid and $\eta=\eta^{(2)}$ is a Riemannian metric on $G^{(2)}$ that is invariant by the $S_3$-action and transverse to the composition map $m:G^{(2)}\to G$. The metric $\eta$ induces metrics $\eta^{(1)}=(\pi_2)_\ast \eta^{(2)}=m_\ast \eta^{(2)}=(\pi_1)_\ast \eta^{(2)}$ on $G$ and $\eta^{(0)}=s_\ast \eta^{(1)}=t_\ast \eta^{(1)}$ on $M$ such that $\pi_2, m, \pi_1:G^{(2)}\to G$ and $s,t:G\to M$ are Riemannian submersions and $i:G\to G$ is an isometry. This is because the $S_3$-action permutes these face maps.

The metric $\eta^{(j)}$, for $j=2,1,0$, is called a $j$-\textbf{metric}. It is important to mention that every proper groupoid can be endowed with a $2$-metric (more generally, an $n$-metric as defined below) and if a Lie groupoid admits a $2$-metric then it is weakly linearizable. For more details visit \cite{dHF,CS}.

We finish this section by giving a quick observation that will be very useful when working with the nerve of a Lie groupoid in Section \ref{S:9}.
\begin{remark}
The notion of \textbf{$n$-metric} on Lie groupoids for $n\geq 3$ was introduced in  \cite{dHF}. This is just a Riemannian metric on the set of $n$-composable arrows $G^{(n)}$ that is invariant by the canonical $S_{n+1}$-action on $G^{(n)}$ and transverse to one (hence to all) face map $G^{(n)}\to G^{(n-1)}$. We can push this $n$-metric forward with the different face maps $G^{(n)}\to G^{(n-1)}$ to define an $(n-1)$-metric on $G^{(n-1)}$ in such a way these face maps become Riemannian submersions. One can use this process to obtain $r$-metrics $\eta^{(r)}$ on $G^{(r)}$ for all $0\leq r\leq n-1$ so that we get Riemannian submersions $(G^{(r)},\eta^{(r)})\to (G^{(r-1)},\eta^{(r-1)})$.
\end{remark}

\section{Morse Lie groupoid morphisms}\label{S:3}
Let $G\rightrightarrows M$ be a Lie groupoid. The space of \textbf{basic functions} on $M$ is defined by

$$C^{\infty}(M)^G:=\{f\in C^{\infty}(M): s^*f=t^*f\}.$$
In other words, a basic function $f:M\to \mathbb{R}$ is just a function which is constant along the groupoid orbits.  On the one hand, the space of basic functions is Morita invariant since

$$H^0_{\mathrm{diff}}(G)=C^{\infty}(M)^G,$$
where $H^0_{\mathrm{diff}}(G)$ denotes $0^{th}$-degree groupoid cohomology, which is well-known to be Morita invariant \cite{crainic}. Hence, one can think of $H^0_{\mathrm{diff}}(G)=C^{\infty}(M)^G$ as the space of smooth functions on the quotient stack $[M/G]$. On the other hand, any basic function $f\in C^{\infty}(M)^G$ induces a Lie groupoid morphism $F:(G\rr M)\to (\mathbb{R}\rr \mathbb{R})$ given either by $F=s^*f$ or $F=t^*f$. It is clear that every Lie groupoid morphism $F:G\to \mathbb{R}$ has this form. Hence one can identify the space of real valued Lie groupoids morphisms with that of basic functions. 

Let us see some elementary examples of real valued Lie groupoid morphisms.

\begin{example} If $M$ is a smooth manifold and $M\rightrightarrows M$ is its underlying unit Lie groupoid then $C^{\infty}(M)^M=C^{\infty}(M)$.
	
\end{example}

\begin{example} If $G$ is a Lie group acting on a smooth manifold $M$ and $G\ltimes M \rightrightarrows M$ is the corresponding action groupoid then $C^{\infty}(M)^{G\times M}$ is given by the set of $G$-invariant smooth function on $M$. 
\end{example}

\begin{example} Suppose that $\pi:M\to N$ is a surjective submersion with corresponding submersion groupoid $M\times_N M\rightrightarrows M$. In this case, $C^{\infty}(M)^{M\times_N M}$ equals the set of smooth functions on $M$ that are constant on the fibers of $\pi$.
\end{example}

\begin{example}\label{KeyExample} Let $G\rightrightarrows M$ be a proper Lie groupoid with proper Haar measure system $\lbrace \mu^x\rbrace$. For any smooth function $f:M\to \mathbb{R}$ it follows that the averaging 
		$$f^\mu(x):=\int_{g\in s^{-1}(x)}(f\circ t)(g)\mu^{x}(g),\qquad x\in M,$$
defines a basic function on $M$. Indeed, the properness of the Haar system $\lbrace \mu^x\rbrace$ ensures that the integral defining $f^\mu$ is finite, the smoothness tells us that $f^\mu$ is also smooth and, moreover, the right-invariance and the identity $t\circ m=t\circ \pi_1$ imply that for all $h\in G$
		\begin{eqnarray*}
			f^\mu(t(h)) & =& \int_{g\in s^{-1}(t(h))}(f\circ t)(g)\mu^{t(h)}(g)= \int_{g\in s^{-1}(s(h))}(f\circ t)(gk)\mu^{s(h)}(g)\\
			&=& \int_{g\in s^{-1}(s(h))}(f\circ t)(g)\mu^{s(h)}(g)=f^\mu(s(h)).
		\end{eqnarray*}
\end{example}

\begin{remark}
A smooth function $F:G\to \mathbb{R}$ is said to be \textbf{multiplicative} if and only if it satisfies $F(gh)=F(g)+F(h)$ for all $(g,h)\in G^{(2)}$. These kinds of functions are in one-to-one correspondence with Lie groupoid morphisms $F:(G\rr M)\to (\mathbb{R}\rr \lbrace\ast\rbrace)$. Observe that multiplicative functions don not allow us to establish a well behaved notion of smooth function on the stacky quotient $[M/G]$ since, for instance, in the most elementary case of the unit groupoid $M\rr M$ we get that those functions turn out to be trivial. That is, we can not even recover the usual smooth functions on $M$ in a natural way. Same inconvenient appears when analyzing other elementary examples.
\end{remark}

%
 
Let $F:G\to \mathbb{R}$ be a Lie groupoid morphism covering $f:M\to \mathbb{R}$. Note that if $x$ is a critical point of $f$ then its orbit $\mathcal{O}_x$ is a critical submanifold of $f$. Hence the critical point set $\textnormal{Crit}(f)\subset M$ is saturated.

\begin{lemma}\label{StructureLemma}
There exists a natural topological groupoid structure $\textnormal{Crit}(F)\rightrightarrows \textnormal{Crit}(f)$.
\end{lemma}

\begin{proof}
Let us suppose that $g\in G$ is a critical arrow of $F$. It is simple to see that both $s(g)$ and $t(g)$ are critical points of $f$ since both $s$ and $t$ are surjective submersions. This in turn implies that $g^{-1}$ is critical arrow of $F$ as well. Also, if $x\in \mathrm{Crit}(f)$ one easily sees that $1_x\in\mathrm{Crit}(F)$. Finally, using the identities $s\circ m=s\circ \pi_2$ and $t\circ m=t\circ \pi_1$, one concludes that if $(g,h)\in G^{(2)}$ with either $g$ or $h$ a critical arrow of $F$, then the composition $gh$ is also a critical arrow of $F$.
\end{proof}

It follows from the previous lemma that $s^{-1}\textnormal{Crit}(f)=t^{-1}\textnormal{Crit}(f)=\textnormal{Crit}(F)$. In particular, if $\mathcal{O}\subseteq M$ is a critical orbit of $f$ then $G_{\mathcal{O}}$ is a critical submanifold for $F$ so that the restricted Lie groupoid $G_{\mathcal{O}}\rr \mathcal{O}$ is a critical Lie subgroupoid of $G\rr M$. Furthermore, we have that $\mathcal{O}\subseteq M$ is a non-degenerate critical orbit for $f$ if and only if $G_{\mathcal{O}}\subseteq G$ is a non-degenerate critical submanifold for $F$. This follows from the fact that $\overline{ds}:\nu(G_{\mathcal{O}})\to \nu(\mathcal{O})$ is a fiberwise isomorphism and Formula \eqref{eq:Hessianpullback} holds true.

\begin{definition}

Let $F:G\to \mathbb{R}$ be a Lie groupoid morphism covering $f:M\to \mathbb{R}$.  We say that $F$ is a \textbf{Morse Lie groupoid morphism} if every critical orbit $\mathcal{O}\subseteq M$ of $f$ is non-degenerate. 
\end{definition}

It is worth noticing that $F:G\to \mathbb{R}$ is a Morse Lie groupoid morphism if and only if every critical subgroupoid $G_{\mathcal{O}}\rightrightarrows \mathcal{O}$ is non-degenerate in the sense that both $\mathcal{O}\subset M$  and $G_{\mathcal{O}}\subseteq G$ are nondegenerate critical submanifolds.

We show now that the notion of Morse Lie groupoid morphism is Morita invariant. Recall that a Morita equivalence between $G\rr M$ and $G'\rr M'$ yields an isomorphism between zero degree groupoid cohomology, that is, an isomorphism between the corresponding spaces of basic functions.  More precisely, if $K\rr N$ is a Lie groupoid together with Morita fibrations $\phi:K\to G$ and $\psi:K\to G'$, then $\psi^*f'\in C^{\infty}(N)^K$ for every $f'\in C^{\infty}(M')^{G'}$. Also, there exists a unique $f\in C^{\infty}(M)^G$ with $\phi^*f=\psi^*f'$. This defines an isomorphism

\begin{equation}\label{eq:moritabasic}
C^{\infty}(M')^{G'}\to C^{\infty}(M)^G; f'\mapsto f.
\end{equation} 

In particular, the isomorphism \eqref{eq:moritabasic} yields an isomorphism

\begin{equation}\label{eq:moritagroupoidmorphism}
\mathrm{Hom}_{Gpds}(G',\mathbb{R})\to \mathrm{Hom}_{Gpds}(G,\mathbb{R}); F'=s^*{f'}\mapsto F=s^*f,
\end{equation}
where $f'\in C^{\infty}(M')^{G'}$ and $f\in C^{\infty}(M)^G$ are related by \eqref{eq:moritabasic}. Our main goal now is to show that \eqref{eq:moritabasic} gives rise to an isomorphism between basic Morse functions, hence between Morse Lie groupoid morphisms. 

\begin{proposition}\label{MoritaInvariance}

Let $G\leftarrow K \rightarrow G'$ be a Morita equivalence covering surjective submersions at the level of objects. The isomorphism \eqref{eq:moritagroupoidmorphism} preserves Morse Lie groupoid morphisms.

\end{proposition}


\begin{proof}
Suppose that $f\in C^{\infty}(M)^G$ allows us to define a Morse Lie groupoid morphism. To prove that $f'\in C^{\infty}(M')^{G'}$ induces another Morse Lie groupoid morphism it suffices to show that $\phi^\ast F$ is a Morse Lie groupoid morphism since $\phi^*f=\psi^*f'$ and both $\phi^0$ and $\psi^0$ are surjective submersions. Indeed, the Lie groupoid morphism $\phi^\ast F:(K\rightrightarrows N)\to(\mathbb{R}\rightrightarrows \mathbb{R})$ is given by the pair $(F\circ \phi^1,f\circ \phi^0)$. Note that if $x\in M$ is a critical point of $f\circ \phi^0$, then $\phi^0(x)$ is a critical point of $f$ since $\phi^0$ is a surjective submersion. Thus, from Identity \eqref{eq:Hessianpullback} we get at $x$ that
$$\mathcal{H}_x(f\circ \phi^0)=d\phi^0(x)^T\cdot \mathcal{H}_{\phi^0(x)}(f)\cdot d\phi^0(x).$$
From \cite{dH} we know that if $\phi$ is a Morita map then $\overline{d\phi^0}:\nu(\mathcal{O}_x)\to \nu(\mathcal{O}'_{\phi^0(x)})$ is a fiberwise isomorphism. Thus, as $\mathcal{H}_{\phi^0(x)}(f)$ is non-degenerate when restricted to $\nu_{\phi^0(x)}(\mathcal{O}'_{\phi^0(x)})$ we conclude that $\mathcal{H}_x(f\circ \phi^0)$ is non-degenerate when restricted to $\nu_x(\mathcal{O}_x)$, as desired.
\end{proof}

\section{Examples and existence of Morse Lie groupoid morphisms}\label{S:4}

In this short section we mention some examples of Morse Lie groupoid morphisms. Elementary examples are the following:

\begin{example} If $M$ is a smooth manifold then every Morse function $f:M\to \mathbb{R}$ induces a Lie groupoid Morphism on the unit groupoid $M\rightrightarrows M$.
\end{example}

\begin{example}  A \textbf{Lie group bundle} is a Lie groupoid $G\rightrightarrows M$ such that $s=t$. Therefore, any Morse function $f:M\to \mathbb{R}$ induces a Morse Lie groupoid morphism on $G\rightrightarrows M$ since its orbits are points. In particular, on a Lie group $G\rr \{*\}$ every Lie groupoid morphism $G\to \mathbb{R}$ is necessarily constant. Hence, there are no interesting examples of Morse Lie groupoid morphisms on Lie groups. Combining this with Proposition \ref{MoritaInvariance} we conclude the same for any transitive Lie groupoid $G\rr M$ since transitive Lie groupoids are always Morita equivalent to Lie groups.

\end{example}

\begin{example} Let $G$ be a compact Lie group acting on a smooth manifold $M$ and consider the action groupoid $G \ltimes M \rightrightarrows M$. Wasserman showed in \cite{W} that the set of $G$-invariant Morse functions on $M$ is dense in the set of $G$-invariant functions. Hence, there always exist Morse Lie groupoid morphisms on $G \ltimes M \rightrightarrows M$. 
\end{example}

\begin{example} Let $(M,\mathcal{F})$ be a complete transverse parallel foliated connected manifold (see for instance \cite[s. 4.5]{Mo} or \cite[s. 4.1.2]{MM}). Let $p:\widetilde{M}\to M$ denote the universal covering map and consider the induced foliation $\widetilde{\mathcal{F}}$ on $\widetilde{M}$. This foliation is simple so that we have that $X=\widetilde{M}/\widetilde{\mathcal{F}}$ is a Hausdorff manifold and the canonical projection $\pi_{\textnormal{bas}}:\widetilde{M}\to X$ is a surjective submersion. It turns out that with this data it is possible to obtain a natural structure of principal groupoid bi-bundle $(\widetilde{M},p,\pi_{\textnormal{bas}}):\textnormal{Hol}(M,\mathcal{F})\dotarrow{} \pi_1(M)\ltimes X$ between the holonomy groupoid $\textnormal{Hol}(M,\mathcal{F})\rightrightarrows M$ and the action groupoid $\pi_1(M)\ltimes X \rightrightarrows X$, that is, a Morita equivalence. Therefore, if $M$ has finite fundamental group then as a consequence of Wasserman's result and Proposition \ref{MoritaInvariance} we get that there exist Morse Lie groupoid morphisms on $\textnormal{Hol}(M,\mathcal{F})\rightrightarrows M$.
\end{example}

\begin{example} Suppose that $G\rightrightarrows M$ is a Lie groupoid for which the orbit space $M/G$ admits a structure of smooth manifold such that the canonical projection $\pi:M\to M/G$ is a surjective submersion. Any Morse function $\tilde{f}:M/G\to \mathbb{R}$ induces a basic function $f:=\pi^*\tilde{f}:M\to \mathbb{R}$. The fact that every orbit $\mathcal{O}\subseteq M$ is non-degenerate follows by applying \eqref{eq:Hessianpullback} to the surjective submersion $\pi:M\to M/G$. In particular, $G$ has Morse groupoid morphisms.
\end{example}

\begin{example}
Consider the action of $\mathbb{Z}$ on $S^1$ given by $n\cdot e^{i\theta}= e^{i(\theta+2\pi n\alpha)}$ where $\alpha \in [0,1]$ is some irrational number. This action is free but not proper. Moreover, every orbit of such an action is dense in $S^1$. Therefore, if there is a $\mathbb{Z}$-invariant function on $S^1$ then it must be constant. As a consequence of Morita invariance, the foliation groupoid on the 2-torus $\mathbb{T}^2$ associated to the Kronecker foliation admits no Morse Lie groupoid morphisms different from the constant functions. 
\end{example}


\subsection {The proper groupoid case.}

It is well known that the set of Morse functions over a compact manifold $M$ form a open and dense subset of $C^\infty(M)$ with respect to the strong topology. As we mentioned before, Wasserman showed in \cite{W} a similar result for the set of $G$-invariant Morse functions placed inside the set of all $G$-invariant functions when the Lie group $G$ acting on $M$ is compact.  For the case of orbifolds, Hepworth introduced in \cite{H} a modified strong topology on the set of functions over an orbifold and proved that the subset of those that are Morse is also open and dense with respect such a topology \cite[s. 6]{H}. 

\begin{remark}\label{NonCompactDiscrete}
If $G$ is a non-compact Lie group acting properly on a smooth manifold $M$ then the strong topology induced on the set of all $G$-invariant smooth functions on $M$ becomes discrete, see \cite[Prop. 4.7]{IK}. As consequence, the set of $G$-invariant Morse functions on $M$ can not be dense with respect to the strong topology induced on $C^\infty(M)^G$.
\end{remark}

A simple example in which the previous phenomena occurs is the following.

\begin{example}\label{ExampleNonProperOrbitMap}
Consider the usual free and proper action of $\mathbb{Z}$ on $\mathbb{R}$ given by translations. In this case the quotient map (exponential) $\pi:\mathbb{R}\to S^1$ is clearly not proper. It follows that the set of $\mathbb{Z}$-invariant Morse functions on $\mathbb{R}$ can not be dense in $C^\infty(\mathbb{R})^{\mathbb{Z}}$ with respect to the strong topology.
\end{example}

It is well known that if $G$ is a compact Lie group acting on a smooth manifold $M$ then the canonical orbit projection $M\to M/G$ is a proper map. Let us consider a proper Lie groupoid $G\rightrightarrows M$. Motivated by Remark \ref{NonCompactDiscrete} and Example \ref{ExampleNonProperOrbitMap} it seems reasonable to further assume that the canonical projection $\pi:M\to M/G$ is a proper map in order to show that Morse Lie groupoid morphisms are dense in $C^\infty(M)^G$ with respect to the induced strong topology. Our aim now is to prove that this is in fact the case. Moreover, if $M/G$ is compact then we show that they actually form an open subset. Such a result clearly recovers both the classical and the equivariant cases and partially recovers the case of orbifolds. This is because the modified strong topology defined by Hepworth only agrees with the classical strong topology on $C^\infty(M)^G$ when $\pi:M\to M/G$ is proper \cite[Prop. 6.5]{H}. Note that if $G$ is a proper Lie groupoid over a compact manifold $M$ then the previous requirements are clearly fulfilled. In particular, an interesting example to have in mind is given by the holonomy groupoid $\textnormal{Hol}(M,\mathcal{F})\rightrightarrows M$ induced by a regular Riemannian foliation $\mathcal{F}$ over a compact manifold $M$.


\begin{lemma}
The set $C^\infty(M)^G$ is a Baire space in $C^\infty(M)$ with respect to the strong topology.
\end{lemma}
\begin{proof}
Let us consider a proper Haar measure system $\lbrace \mu^x\rbrace$ for $G\rightrightarrows M$. By taking average with respect to $\lbrace \mu^x\rbrace$ we can define a surjective linear continuous operator $P^\mu:C^{\infty}(M)\to C^{\infty}(M)^G$ by sending $f$ to $f^\mu$ as defined in Example \ref{KeyExample}. This is consequence of having that $C^{\infty}(M)$ is a Fréchet space and the average operator $P^\mu$ is a projection, i.e. $(P^\mu)^2=P^\mu$. In particular, by the Banach-Schauder Theorem it follows that $P^\mu$ is open. Mather proved in \cite[Prop. 3.1]{Ma} that $C^{\infty}(M)$ is a Baire space so that the previous facts imply that the space of basic functions $C^{\infty}(M)^G$ is also a Baire space in $C^\infty(M)$, as claimed.
\end{proof}

Although the proof of next result uses some terminology to be introduced later on we state it here because of our purposes in this subsection.

\begin{lemma}\label{OpenNoDegenerate}
Let $G_U\rr U$ be an open subgroupoid of $G\rr M$ and $K\subset U/G_U \subset M/G$ be a compact subset. Then the set of basic functions $f:M\to \mathbb{R}$ such that $f|U$ has no degenerate critical orbits inside the compact $\pi^{-1}(K)$ is an open subset in $C^{\infty}(M)^G$.
\end{lemma}
\begin{proof}
The proof of this result is a straightforward adaptation of the proof of \cite[Lem. 5.32]{BH3} by considering instead the coarse differential $dF_{[x]}:T_{[x]}[M/G]\to \mathbb{R}$ and the coarse Hessian $H_{[x]}(F):T_{[x]}[M/G]\times T_{[x]}[M/G]\to \mathbb{R}$ as defined in Section \ref{S:7}.
\end{proof}

\begin{theorem}\label{ExistenceProper}
Suppose that $G\rightrightarrows M$ is a proper Lie groupoid such that the canonical projection $\pi:M\to M/G$ is a proper map. Then Morse Lie groupoid morphisms on $G$ are dense in the space of all Lie groupoid morphisms $G\to \mathbb{R}$. Moreover, if $M/G$ is compact then they form an open subset.
\end{theorem}

\begin{proof}

It is well known that since our groupoid is proper any $x\in M$ has an open neighborhood $U_\alpha$ in $M$ such that $G_{U_{\alpha}}\rightrightarrows U_\alpha$ is Morita equivalent to the action groupoid $G_x \ltimes \nu_x(\mathcal{O}_x)\rightrightarrows \nu_x(\mathcal{O}_x)$, see for instance \cite{CS}. Therefore, by Proposition \ref{MoritaInvariance} combined with the Wasserman's density result we have that there exist Morse Lie groupoid morphisms on $G_{U_\alpha}\rightrightarrows U_\alpha$ and they are actually dense. Recall that the canonical projection $\pi:M\to M/G$ is an open map. Thus, we may assume that the open cover $\lbrace U_\alpha/G_\alpha \rbrace$ of $M/G$ we obtain from above is countable and it satisfies that for every $\alpha$ there exists a compact subset $K_\alpha  \subset U_\alpha/G_\alpha$ such that $\lbrace K_\alpha \rbrace$ also covers $M/G$ since this is Hausdorff and paracompact.

Let $\mathcal{M}_\alpha(G)$ denote the subset formed by basic functions $f\in C^{\infty}(M)^G$ for which $f|_{U_\alpha}$ has no degenerate critical orbits inside the compact set $\pi^{-1}(K_\alpha)$. Note that the set $\mathcal{M}(G)$ of all basic functions defining Morse Lie groupoid morphisms on $G\rightrightarrows M$ agrees with $\bigcap_\alpha \mathcal{M}_\alpha(G)$ since $\lbrace \pi^{-1}(K_\alpha) \rbrace$ covers $M$. If we show that $\mathcal{M}_\alpha(G)$ is open and dense for all $\alpha$ then $\mathcal{M}(G)$ will be dense since $C^\infty(M)^G$ is a Baire space. On the one hand, the fact that $\mathcal{M}_\alpha(G)$ is open follows from Lemma \ref{OpenNoDegenerate}. On the other hand, pick $f\in C^\infty(M)^G$ and let $\mathcal{N}$ be an open neighborhood of $f$. If we prove that $\mathcal{N}\cap \mathcal{M}_\alpha(G)$ is nonempty then we will have that $\mathcal{M}_\alpha(G)$ is dense. Consider the restriction $f|_{U_\alpha}$ of $f$ on $U_\alpha$. By taking average with respect to $\lbrace\mu^x \rbrace$ it follows that given the closed subset $\pi^{-1}(K_\alpha)$ and some open neighborhood $W_\alpha\subseteq U_\alpha$ of $\pi^{-1}(K_\alpha)$ there exists a basic function $\phi_\alpha: U_\alpha\to \mathbb{R}$ such that $\phi_\alpha\equiv 1$ in a small neighborhood of $\pi^{-1}(K_\alpha)$ and such that its support is compact and contained in $W_\alpha$, use \cite[Prop. 9]{CM} with \cite[Lem. 3.12]{H}. Also, since $\phi_\alpha$ has compact support it follows that by using again an average process we can consider the map $C^\infty(U_\alpha)^{G_{U_\alpha}}\to C^{\infty}(M)^G$ given by $g\mapsto \widetilde{\phi_\alpha g}$, where the symbol $\ \widetilde{}\ $ denotes smooth $G$-invariant extension by zero, compare with \cite[Lem 3.11 and Lem. 6.12]{H}. This map is continuous so that there exists a small enough open neighborhood $\mathcal{N}'$ in $C^\infty(U_\alpha)^{G_{U_\alpha}}$ such that $f(1-\widetilde{\phi_\alpha})+\widetilde{\phi_\alpha g}\in \mathcal{N}$ for all $g\in \mathcal{N}'$. We already know that Morse Lie groupoid morphisms on $G_{U_\alpha}\rightrightarrows U_\alpha$ are dense so that we may assume that there exists $g\in \mathcal{N}'$ defining a Morse Lie groupoid morphism such that $f(1-\widetilde{\phi_\alpha})+\widetilde{\phi_\alpha g}\in \mathcal{N}$. Recall that by construction $\phi_\alpha\equiv 1$ in a neighborhood of $\pi^{-1}(K_\alpha)\subset U_\alpha$ for which  $f(1-\widetilde{\phi_\alpha})+\widetilde{\phi_\alpha g}$ restricts to $g$ over such a neighborhood, meaning that $f(1-\widetilde{\phi_\alpha})+\widetilde{\phi_\alpha g}\in \mathcal{N}\cap \mathcal{M}_\alpha(G)$. That is, $\mathcal{N}\cap \mathcal{M}_\alpha(G)$ is nonempty as claimed.

Finally, if $M/G$ is compact then the intersection $\mathcal{M}(G)=\bigcap_\alpha \mathcal{M}_\alpha(G)$ becomes finite. This completes the proof.
\end{proof}

\begin{remark}
Instead of considering the Morita equivalence between $G_{U_{\alpha}}\rightrightarrows U_\alpha$ and $G_x \ltimes \nu_x(\mathcal{O}_x)\rightrightarrows \nu_x(\mathcal{O}_x)$ as we did above we could also have considered the following identification of $G_{U_\alpha}$ which appears in \cite[Cor. 3.11]{PPT}. For a proper Lie groupoid $G\rr M$ there is an open neighborhood $U_\alpha$ of $x$ in $M$ diffeomorphic to $O\times V_x$ where $O$ is an open ball in the orbit $\mathcal{O}_x$ centered at $x$ and $V_x$ is a $G_x$-invariant open ball in $\nu_x(\mathcal{O}_x)$ centered at the origin. Under this diffeomorphism $G_{U_\alpha}\rr U_\alpha$ is isomorphic to the product of the pair groupoid $O\times O\rr O$ and the action groupoid  $G_x \ltimes V_x\rightrightarrows V_x$. The pair groupoid admits Morse Lie groupoid morphisms since this is Morita equivalent to a manifold and the action groupoid admits Morse Lie groupoid morphisms as consequence of Wasserman's result.
\end{remark}


\subsection {Moment maps on $0$-symplectic groupoids.}\label{S:4.2}

 Our goal now is to show that moment maps for Hamiltonian $2$-actions on $0$-symplectic groupoids in the sense of \cite{hsz} induce Morse--Bott Lie groupoid morphisms. This will be consequence of the results proved in \cite{LS} for Hamiltonian actions on presymplectic manifolds. We start by briefly introducing some necessary terminology which can be found in \cite{hsz}.

A \textbf{foliation groupoid} is a Lie groupoid $G \rightrightarrows M$ whose space of objects $M$ is Hausdorff and whose isotropy groups $G_x$ are discrete for all $x\in M$. For instance, every \'etale  Lie groupoid with Hausdorff objects manifold is a foliation groupoid. The converse is not true, however every foliation groupoid is Morita equivalent to an \'etale groupoid. As shown in \cite{crainic,CraMoer}, being a foliation groupoid is equivalent to the associated Lie algebroid anchor map $\rho:A\to TM$ being injective. As a consequence, the manifold $M$ comes with a regular foliation $\mathcal{F}$ tangent to the leaves of $\mathrm{im}(\rho)\subseteq TM$. Note that if $G \rightrightarrows M$ is source-connected the leaves of $\mathrm{im}(\rho)\subseteq TM$ coincide with the groupoid orbits. 

A \textbf{basic} $2$-form on a foliation groupoid $G\rr M$ is given by a pair of 2-forms $\omega=(\omega_1,\omega_0)$ with $\omega_1\in \Omega^2(G)$, $\omega_0\in\Omega^2(M)$ satisfying $s^*\omega_0=\omega_1=t^*\omega_0$. We say that $\omega$ is \textbf{non-degenerate} if $\ker(\omega_0)=\mathrm{im}(\rho)\subseteq TM$. A basic 2-form $\omega=(\omega_1,\omega_0)$ is \textbf{closed} if $\omega_0$ is closed.


\begin{definition}\label{0-symplectic}\cite{hsz}
A $0$-\textbf{symplectic groupoid} is a foliation groupoid equipped with a closed and non-degenerate basic $2$-form.
\end{definition}
It is important to point out that this notion of symplectic groupoid differs from that of Weinstein introduced in \cite{We}, since $\omega_1\in\Omega^2(G)$ is not necessarily non-degenerate nor multiplicative. 

It follows immediately from Definition \ref{0-symplectic} that $(M,\omega_0)$ is a pre-symplectic manifold with $\textnormal{ker}(\omega_0)=T\mathcal{F}$ in the sense of \cite{LS,RatiZung}. Additionally, there is a left action of the product groupoid $G\times G \rightrightarrows M\times M$ on $G$ along $(s,s)$ given by $(g,h)f=gfh^{-1}$. The components of the orbits of this action define a regular foliation $\mathcal{F}_1$ of $G$ satisfying $T\mathcal{F}_1=\textnormal{ker}(ds)+\textnormal{ker}(dt)$. In particular, $\omega$ in Definition \ref{0-symplectic} is non-degenerate if and only if $\textnormal{ker}(\omega_1)=\textnormal{ker}(ds)+\textnormal{ker}(dt)$. As a consequence, $(G,\omega_1)$ is also a pre-symplectic manifold with $\textnormal{ker}(\omega_1)=T\mathcal{F}_1$.

To define Hamiltonian $2$-actions in this context we need to introduce the notion of Lie $2$-group. A \textbf{Lie 2-group} is a Lie groupoid $K^{(1)} \rightrightarrows K^{(0)}$ where both $K^{(1)}$ and $K^{(0)}$ are Lie groups and all the structural maps are Lie group morphisms. A \textbf{$2$-action} of $K^{(1)} \rightrightarrows K^{(0)}$ on $G \rightrightarrows M$ is Lie groupoid morphism from the product groupoid $K^{(1)}\times G \rightrightarrows K^{(0)}\times M$ to $G \rightrightarrows M$ whose component maps are Lie group actions in the usual sense. If we apply the Lie functor to $K^{(1)} \rightrightarrows K^{(0)}$ then we obtain a \textbf{Lie $2$-algebra} $\mathfrak{k}^{(1)} \rightrightarrows \mathfrak{k}^{(0)}$, i.e. a Lie groupoid where $\mathfrak{k}^{(1)}$ and $\mathfrak{k}^{(0)}$ are Lie algebras and all the structural maps are Lie algebra morphisms. 

It is well known that there exists a bijective correspondence between Lie 2-groups and \textbf{crossed modules of Lie groups}; see \cite{BS}. By a crossed module of Lie groups we mean a quadruple $(K,H,\partial,\alpha)$ where $K$ and $H$ are Lie groups, $\partial:H\to K$ is a Lie group morphism, and $\alpha:K\to  \textnormal{Aut}(H)$ is an action of $K$ on $H$ subject to the requirements $\partial(\alpha(g)(h)) = g\partial(h)g^{-1}$ and $\alpha(\partial(h))(h') = hh'h^{-1}$ for all $g\in K$
and $h, h'\in H$. In particular, if $K^{(1)} \rightrightarrows K^{(0)}$ is Lie $2$-group then its associated crossed module is determined by the data $K=K^{(0)}$, $H=\textnormal{ker}(s)$, $\partial=t|_H$, and $\alpha$ is the conjugation action of $K^{(1)}$ on $H$ composed with the identity bisection map $u:K^{(0)}\to K^{(1)}$.

Suppose that $K^{(1)} \rightrightarrows K^{(0)}$ is a foliation Lie $2$-group with associated crossed module of Lie groups $(K,H,\partial,\alpha)$ and Lie $2$-algebra $\mathfrak{k}^{(1)} \rightrightarrows \mathfrak{k}^{(0)}$. In this case we have that $\textnormal{Lie}(\partial):\mathfrak{h}\to \mathfrak{k}=\mathfrak{k}^{(0)}$ is injective. As the Lie algebra $\textnormal{Lie}(\partial)(\mathfrak{h})\cong \mathfrak{h}$ is an ideal in $\mathfrak{k}$ then we may consider the quotient Lie algebra $\mathfrak{k}/\mathfrak{h}$. Let us denote by $\pi:\mathfrak{k}\to \mathfrak{k}/\mathfrak{h}$ the quotient map.

\begin{lemma}\cite{hsz}
If $K^{(1)} \rightrightarrows K^{(0)}$ is a foliation Lie $2$-group then the pair $(\pi\circ \textnormal{Lie}(t),\pi):(\mathfrak{k}^{(1)} \rightrightarrows \mathfrak{k}^{(0)})\to (\mathfrak{k}/\mathfrak{h}\rightrightarrows \mathfrak{k}/\mathfrak{h})$ is a Morita morphism of Lie $2$-algebras. 
\end{lemma}

As a consequence of the previous result, the Lie groupoid morphism $\textnormal{Ad}:(K^{(1)}\times \mathfrak{k}^{(1)} \rightrightarrows K^{(0)}\times \mathfrak{k}^{(0)})\to (\mathfrak{k}^{(1)} \rightrightarrows \mathfrak{k}^{(0)})$, which is formed by the adjoint actions $\textnormal{Ad}_j$ of $K^{(j)}$ on $\mathfrak{k}^{(j)}$ (for $j=0,1$), descends to a well defined $2$-action of $K^{(1)} \rightrightarrows K^{(0)}$ on $\mathfrak{k}/\mathfrak{h}\rightrightarrows \mathfrak{k}/\mathfrak{h}$. By abuse of language, we will call this induced $2$-action as the \textbf{adjoint action} and denote it by $\textnormal{Ad}$ as well. Accordingly, this notion of adjoint action allows us to speak about the \textbf{coadjoint action} $\textnormal{Ad}^\ast:(K^{(1)}\times (\mathfrak{k}/\mathfrak{h})^\ast\rightrightarrows K^{(0)}\times (\mathfrak{k}/\mathfrak{h})^\ast)\to ((\mathfrak{k}/\mathfrak{h})^\ast \rightrightarrows (\mathfrak{k}/\mathfrak{h})^\ast)$ which is nothing but the $2$-action whose component maps are the coadjoint actions $\textnormal{Ad}^\ast_j$ of $K^{(j)}$ on $(\mathfrak{k}/\mathfrak{h})^\ast$ induced by the identification we mentioned above. 

The final ingredient necessary to define Hamiltonian 2-actions is given by the notion of fundamental vector field associated to a $2$-action. Namely, consider a $2$-action of  $K^{(1)} \rightrightarrows K^{(0)}$ on $G \rightrightarrows M$. It is simple to check that for every $\xi\in \mathfrak{k}=\mathfrak{k}^{(0)}$ the pair $(\textnormal{Lie}(u)(\xi)_{G},\xi_{M})$, formed by the fundamental vector fields of the respective Lie group actions, determines a multiplicative vector field on $G \rightrightarrows M$. Therefore, if $K^{(1)} \rightrightarrows K^{(0)}$ is a foliation Lie $2$-group and $G \rightrightarrows M$ is a foliation groupoid then the \textbf{fundamental vector field} associated to the $2$-action above is by definition the basic vector field on $G \rightrightarrows M$ determined by the pair $(\textnormal{Lie}(u)(\xi)_{G},\xi_{M})$. By applying this procedure it is possible to show that there exists a Lie algebra anti-morphism from $\mathfrak{k}/\mathfrak{h}$ to the Lie algebra of basic vector fields on $G \rightrightarrows M$; see \cite[Pro. 6.9.2]{hsz} for further details.

\begin{definition}\cite{hsz}
Let $(G \rightrightarrows M,\omega)$ be a $0$-symplectic groupoid and let $K^{(1)} \rightrightarrows K^{(0)}$ be a foliation Lie $2$-group with associated crossed module $(K,H,\partial,\alpha)$. A $2$-action of $K^{(1)} \rightrightarrows K^{(0)}$ on $(G \rightrightarrows M,\omega)$ is said to be \textbf{Hamiltonian} if the following conditions hold:
\begin{enumerate}
\item the action of $K^{(0)}$ on $M$ is presymplectic, and
\item there is a morphism of Lie groupoids called \textbf{moment map}
$$\mu=(\mu_1,\mu_0):(G \rightrightarrows M)\to ((\mathfrak{k}/\mathfrak{h})^\ast \rightrightarrows (\mathfrak{k}/\mathfrak{h})^\ast),$$
verifying 
\begin{itemize}
\item[i)] for all $\xi\in \mathfrak{k}/\mathfrak{h}$ it satisfies $d\mu^\xi_0=\iota_{\xi_{M}}\omega_0$, and 
\item[ii)] $\mu$ is equivariant with respect to the 2-action of $K^{(1)} \rightrightarrows K^{(0)}$ on $(G \rightrightarrows M,\omega)$ and the coadjoint action of $K^{(1)} \rightrightarrows K^{(0)}$ on $(\mathfrak{k}/\mathfrak{h})^\ast \rightrightarrows (\mathfrak{k}/\mathfrak{h})^\ast$.
\end{itemize}
\end{enumerate}
If all these conditions are satisfied then we say that $(G \rightrightarrows M,\omega)$ is a \textbf{Hamiltonian $(K^{(1)} \rightrightarrows K^{(0)})$-groupoid} with moment map $\mu$.
\end{definition}
Some observations about the previous definition come in order. First, as we are working with a $2$-action and $\omega$ is basic then we immediately get that the action of $K^{(1)}$ on $G$ is also presymplectic. Additionally, $d\mu^{\xi}_1=\iota_{\textnormal{Lie}(u)(\xi)_{G}}\omega_1$. This follows from the fact that $(\textnormal{Lie}(u)(\xi)_{G},\xi_{M})$ is a multiplicative vector field, $\omega$ is basic, and either $\mu_0\circ s=\mu_1$ or $\mu_0\circ t=\mu_1$. One also observes that
$$s^\ast(\mu^\xi_0)(x)=\mu_0(s(x))(\xi)=\mu_1(x)(\xi)=\mu_0(t(x))(\xi)=t^\ast(\mu^\xi_0)(x).$$
Therefore, for each $\xi\in \mathfrak{k}/\mathfrak{h}$ we have a well defined Lie groupoid morphism $\mu^\xi:(G \rightrightarrows M)\to (\mathbb{R} \rightrightarrows \mathbb{R})$ given either by $s^\ast(\mu^\xi_0)$ or $t^\ast(\mu^\xi_0)$. The required condition which will allow us to ensure that $\mu^\xi$ is a Morse Lie groupoid morphism is determined in terms of the notion of ``cleanness'' introduced in \cite{LS}.

Consider a left action of a connected Lie group $K$ on a presymplectic manifold $(M,\omega)$ with foliation $\mathcal{F}$ and set $\mathfrak{n}(\mathcal{F})=\lbrace \xi\in\mathfrak{k}:\ (\xi_M)(x)\in T_x\mathcal{F}\ \textnormal{for all}\ x\in M\rbrace$. This space is an ideal in $\mathfrak{k}$. Let $N(\mathcal{F})$ be the connected immersed Lie subgroup in $K$ with Lie algebra $\mathfrak{n}(\mathcal{F})$. 

\begin{definition}\cite{LS}
The action of $K$ on $M$ is \textbf{clean} if 
$$T_x(\mathcal{O}_{N(\mathcal{F})}(x))=T_x(\mathcal{O}_{K}(x))\cap T_x\mathcal{F},$$
for all $x\in M$.
\end{definition}

\begin{remark}\label{EquiRemark}
Suppose that $F:G\to \mathbb{R}$ is a Lie groupoid morphism covering a basic $K^{(0)}$-invariant function $f:M\to \mathbb{R}$. Note that the action of $K^{(0)}$ on $M$ imposes additional symmetries in $M$ different from the ones we already had associated to the Lie groupoid structure. This in particular can make the dimension of the connected components of $\textnormal{Crit}(f)$ to increase since they may content more than one groupoid orbit. If this is the case then we say that $F:G\to \mathbb{R}$ is a \textbf{Morse--Bott Lie groupoid morphism} if $f$ is a Morse--Bott function in the usual sense.
\end{remark}

Summing up, we are in conditions to state:

\begin{proposition}\label{MomentMapsMorse}
Let $(G \rightrightarrows M,\omega)$ be a Hamiltonian $(K^{(1)} \rightrightarrows K^{(0)})$-groupoid with moment map $\mu:G\to (\mathfrak{k}/\mathfrak{h})^*$. Suppose that $K^{(0)}$ is a torus and the action of $K^{(0)}$ on $M$ is clean. Then, for every $\xi\in \mathfrak{k}/\mathfrak{h}$ the map $\mu^\xi:(G \rightrightarrows M)\to (\mathbb{R} \rightrightarrows \mathbb{R})$ is a Morse--Bott Lie groupoid morphism with even index at every non-degenerate critical submanifold.
\end{proposition}

\begin{proof}
Let $\xi\in \mathfrak{k}/\mathfrak{h}$ be fixed. The fact that $\mu^\xi$ is a Lie groupoid morphism implies that the critical point set of $\mu_0^\xi$ is saturated in $M$. Thus, when applying \cite[Thm. 3.4.5]{LS} to our situation we get that every critical component in $\textnormal{Crit}(\mu_0^\xi)$ is a non-degenerate saturated submanifold of even index since the action of $K^{(0)}$ on $M$ is clean. Therefore, $\mu^\xi$ is a Morse--Bott Lie groupoid morphism with the required property.
\end{proof}

In particular, this result recovers the case of toric actions on symplectic orbifolds studied for instance in \cite{HolmMatsu}.


\section{The Morse lemma}\label{S:5}

The goal of this section is to state a version of the Morse lemma in the Lie groupoid setting and to describe some features of the negative normal groupoid over a nondegenerate critical orbit. Let us start by recalling that given a Morse-Bott function $f:M\to \mathbb{R}$ and a non-degenerate critical submanifold $C\subseteq M$, the Morse-Bott Lemma gives a local normal form for $f$ around $C$. Namely, on a suitable neighborhood of $C$ the function $f$ looks like the quadratic form $Q_f:\nu(C)\to \mathbb{R}$ defined in \eqref{eq:quadratic} up to a constant, see for instance \cite{BH1}.

We start by considering a Lie groupoid morphism $F:(G\rr M)\to (\mathbb{R}\rr \mathbb{R})$ associated to a basic function $f\in C^{\infty}(M)^G$. Suppose that $\mathcal{O}\subseteq M$ is a non-degenerate critical orbit of $f$ and consider the restricted Lie groupoid $G_{\mathcal{O}}\rr \mathcal{O}$. As $G_{\mathcal{O}}\subseteq G$ is a non-degenerate critical submanifold for $F$ the construction of \eqref{eq:quadratic} applies to both $f$ and $F$ yielding quadratic forms $Q_F:\nu(G_{\mathcal{O}})\to \mathbb{R}$ and $Q_f:\nu(\mathcal{O})\to \mathbb{R}$.

\begin{lemma}\label{QuadraticLemma}
	The pair $(Q_F,Q_f):(\nu(G_{\mathcal{O}})\rightrightarrows \nu(\mathcal{O})) \to (\mathbb{R}\rightrightarrows \mathbb{R})$ is a Lie groupoid morphism.
\end{lemma}
\begin{proof}

Let $g\in G_{\mathcal{O}}$ be a critical arrow. Then, from identity \eqref{eq:Hessianpullback} we get that 
	$$\mathcal{H}_{g}(f\circ s)=\overline{ds}(g)^T\cdot \mathcal{H}_{s(g)}(f)\cdot \overline{ds}(g).$$
	
\noindent Given that $f:M\to \mathbb{R}$ is basic we actually have $\overline{ds}(g)^T\cdot \mathcal{H}_{s(g)}(f)\cdot \overline{ds}(g)=\overline{dt}(g)^T\cdot \mathcal{H}_{x}(f)\cdot \overline{dt}(g)$. Thus, it follows from the definition of $Q_F$ and $Q_f$ that the previous two identities immediately imply that $Q_{f\circ s}=Q_f\circ \overline{ds}$ and $Q_{f\circ t}=Q_f\circ \overline{dt}$ as required.	
\end{proof}

\subsection{Normal form around an orbit}

Let $G\rr M$ be a proper Lie groupoid and $\mathcal{O}\subseteq M$ be an orbit. Denote by $G_{\mathcal{O}}\rr \mathcal{O}$ the restricted Lie groupoid. A \textbf{groupoid neighborhood} of $G_{\mathcal{O}}\rr \mathcal{O}$ is given by an open Lie subgroupoid $(\tilde{U}\rr U)\subseteq (G\rr M)$ with $G_{\mathcal{O}}\subset \tilde{U}$ and $\mathcal{O}\subset U$. As our groupoid is proper we can assume that the groupoid neighborhoods are \textbf{full} in the sense that $\tilde{U}=G_U=s^{-1}(U)\cap t^{-1}(U)$. Let $\nu(G_{\mathcal{O}})\rr \nu(\mathcal{O})$ be the normal groupoid associated to an orbit. We say that $G\rr M$ is \textbf{linearizable} around $\mathcal{O}$ if there are full groupoid neighborhoods $(G_U\rr U)\subseteq (G\rr M)$ of $G_{\mathcal{O}}\rr \mathcal{O}$ and $(\nu(G_{\mathcal{O}})_V\rr V)\subseteq (\nu(G_{\mathcal{O}})\rr \nu(\mathcal{O}))$ of $G_{\mathcal{O}}\rr \mathcal{O}$ seen as the zero section, and a Lie groupoid isomorphism

$$\phi:(\nu(G_{\mathcal{O}})_V\rr V)\to (G_U\rr U),$$
which is the identity on $G_{\mathcal{O}}\rr \mathcal{O}$. We refer to $G_U\rr U$ as a \textbf{full groupoid tubular neighborhood} of $G_{\mathcal{O}}\rr \mathcal{O}$. The \emph{linearization theorem} asserts that every proper Lie groupoid is linearizable around any of its orbits. It is worth mentioning that this result generalizes other classical and important results in differential geometry such as the Ehresmann’s theorem for submersions, the local Reeb stability for foliations, and the tube theorem for proper Lie group actions. The linearization problem was first addressed by Weinstein in \cite{We2} for the case of regular proper groupoids by reducing it to the fixed point case. A first complete proof of this result was provided by Zung in \cite{Zung} with the extra assumption of source locally triviality. The latter hypothesis and variants of it were treated later by Crainic and Struchiner in \cite{CS}. Other novel approaches that lead to much more geometric proofs of the linearization theorem are given in  \cite{dHF,dHdM2,Me2}.

Suppose that $G_{\mathcal{O}}\rr \mathcal{O}$ is a non-degenerate critical subgroupoid of a Lie groupoid morphism $F:(G\rr M)\to (\mathbb{R}\rr \mathbb{R})$ and consider a full groupoid tubular neighborhood $G_U\rr U$ as above. The \textbf{local model} of $F$ around $G_\mathcal{O}$ is defined as the Lie groupoid morphism 

\begin{equation}\label{eq:localF}
\tilde{F}:=\phi^*F:(\nu(G_{\mathcal{O}})_V\rr V)\to (\mathbb{R}\rr \mathbb{R}).
\end{equation}
Note that the zero section is a non-degenerate critical subgroupoid of $\tilde{F}$. 

\begin{theorem}[Morse lemma]\label{LieGroupoidMorseLemma}
Let $F:G\to \mathbb{R}$ be a Morse Lie groupoid morphism covering $f:M\to \mathbb{R}$. If $G$ is proper, then around a non-degenerate critical subgroupoid $G_{\mathcal{O}}\rr \mathcal{O}$ there is a full groupoid tubular neighborhood $\phi: (\nu(G_{\mathcal{O}})_V\rr V)\xrightarrow[]{\cong} (G_U\rightrightarrows U)$ such that
$$\tilde{F}=c+Q_F.$$
\end{theorem}

\begin{proof}
		
Let $x\in \mathcal{O}$. By the Slice Theorem for proper Lie groupoids we know that there is a transversal $T$ to the orbit $\mathcal{O}$ such that the
restricted subgroupoid $G_T\rr T$ is isomorphic to the action groupoid $G_x \ltimes B\rr B$ for some
open set $0 \in B \subset \nu_x(\mathcal{O})$, consult for instance \cite{CS,dH,PPT,Zung}. Also, it is known that $G_T\rr T$ is Morita equivalent to $G_U\rr U$ where $U$ is the  open saturation of $T$, see \cite{CS}. Let us consider the restriction $F_U:G_U\to \mathbb{R}$ of the Morse Lie groupoid morphism $F$ to $G_U\rr U$ and transfer it to a Morse Lie groupoid morphism $F_U'$ on $G_x \ltimes W\rr W$ by using Proposition \ref{MoritaInvariance}. Since $G$ is proper, $G_x$ is compact so $G_x \ltimes W\rr W$ is also a proper groupoid. Therefore, by the equivariant version of the Morse lemma \cite{LT,Me2,W}, it follows that there is a full groupoid tubular neighborhood around the corresponding nondegenerate critical orbit $\mathcal{O'}$ on which $\tilde{F_U'}$ agrees up to constant with $Q_{F_U'}$. Hence, as consequence of Proposition \ref{MoritaInvariance} and the fact that linearization is Morita invariant \cite[Prop. 3.7 and Col. 3.9]{CS}, we get that there is a full groupoid tubular neighborhood around $\mathcal{O}$ on which $\tilde{F_U}$ agrees up to constant with $Q_{F_U}$. This completes the proof.
\end{proof}

\begin{remark}\label{MorseLemmaOtherWay}
	
For the sake of completeness we mention that there are other two possible approaches to prove the previous result. The first one relies on the proof of existence of groupoid tubular neighborhoods around an orbit due to Meinrenken in \cite{Me2}, combined with his ideas used to give a proof of the Morse-Bott Lemma. The second one involves the weakly linearization around an orbit provided by the exponential map of a $2$-metric in the sense of del Hoyo and Fernandes in \cite{dHF}, combined with the classical ideas for the proof of the Morse-Bott lemma that can be found for instance in \cite[App.B]{F}. Although the latter approach is less simple that the first one and the one we used above, it has the advantage of not requiring our Lie groupoid to be proper.
\end{remark}

\subsection{Negative normal groupoid}

Let us suppose that our Lie groupoid $G\rightrightarrows M$ can be equipped with a Riemannian $2$-metric $\eta^{(2)}$ on $G^{(2)}$ and consider the induced  $1$-metric $\eta^{(1)}$ on $G$ and $0$-metric $\eta^{(0)}$ on $M$. If $\mathcal{O}$ is a groupoid orbit then we can restrict the $1$-metric $\eta^{(1)}$ to $\nu(G_{\mathcal{O}})$ and the $0$-metric $\eta^{(0)}$ to $\nu(\mathcal{O})$ and use them to identify $\nu(G_{\mathcal{O}})\cong TG_{\mathcal{O}}^\perp$ and $\nu(\mathcal{O})\cong T\mathcal{O}^\perp$. In particular, if
$\mathcal{O}$ is a nondegenerate critical orbit of a Morse Lie groupoid morphism $F:G\to \mathbb{R}$ covering $f:M\to\mathbb{R}$ then $\mathcal{H}_{g}(f\circ s)=\overline{ds}(g)^{-1}\cdot \mathcal{H}_{x}(f)\cdot \overline{ds}(g)$ for all $g\in G_{\mathcal{O}}$, since $s:G\to M$ is a Riemannian submersion. Thus, the indexes $\lambda(G_{\mathcal{O}},F)$ and $\lambda(\mathcal{O},f)$ agree.

\begin{lemma}\label{NormalInvariance}
The Hessian $H(f):\nu({\mathcal{O}})\oplus \nu({\mathcal{O}})\to \mathbb{R}$ and the fiberwise quadratic form $Q_f:\nu({\mathcal{O}})\to \mathbb{R}$ are invariant under the normal representation \eqref{eq:normalrep}.
\end{lemma}

\begin{proof}
First of all, similar arguments as those used in Lemma \ref{QuadraticLemma} show that the Hessian $H(F):\nu(G_{\mathcal{O}})\oplus \nu(G_{\mathcal{O}})\to \mathbb{R}$ is a Lie groupoid morphism covering $H(f):\nu({\mathcal{O}})\oplus \nu({\mathcal{O}})\to \mathbb{R}$. In other words,
\begin{equation}\label{eq:HessianMorphism}
H(F)=(\overline{ds}\oplus \overline{ds})^*H(f)=(\overline{dt}\oplus \overline{dt})^*H(f).
\end{equation}
Pick $x\in \mathcal{O}$ and	let $g,h\in G_x$, $w_1\in T_g G$ and $w_2\in T_h G$ be such that $ds(g)(w_1)=v_1$ and $ds(h)(w_2)=v_2$. Therefore, by using the identities $t\circ \pi_1=t\circ m$ and $s\circ \pi_2=s\circ m$ we obtain
\begin{eqnarray*}
	\mathcal{H}_{x}(f)(g\cdot [v_1],h\cdot [v_2]) & = & \mathcal{H}_{x}(f)([dt(g)(w_1)],[dt(h)(w_2)])\\
	& = & \mathcal{H}_{x}(f)([dt(g)(w_1)],[d(t\circ \pi_1)_{(h,h^{-1}g)}(w_2,(ds(h^{-1}g))^{-1}(ds(h)(w_2))])\\
	& = & \mathcal{H}_{x}(f)([dt(g)(w_1)],[dt(g)(dm_{(h,h^{-1}g)}(w_2,(ds(h^{-1}g))^{-1}(ds(h)(w_2))])\\
\textnormal{By Identity}\ \eqref{eq:HessianMorphism}	& = & \mathcal{H}_{x}(f)([ds(g)(w_1)],[ds(g)(dm_{(h,h^{-1}g)}(w_2,(ds(h^{-1}g))^{-1}(ds(h)(w_2))])\\
	& = & \mathcal{H}_{x}(f)([ds(g)(w_1)],[ds(h^{-1}g)((ds(h^{-1}g))^{-1}(ds(h)(w_2))])\\
	& = & \mathcal{H}_{x}(f)([ds(g)(w_1)],[ds(h)(w_2)]) =  \mathcal{H}_{x}(f)([v_1],[v_2]).
\end{eqnarray*}
\end{proof}

It follows immediately from the this lemma that the index of a non-degenerate critical orbit $\mathcal{O}\subseteq M$ is well-defined even if $\mathcal{O}$ is not connected.

\begin{definition}\label{def:nondegarrowindex}
	Let $F:G\to \mathbb{R}$ be a Morse Lie groupoid morphism covering $f:M\to\mathbb{R}$. If $\mathcal{O}\subseteq M$ is a nondegenerate critical orbit then any arrow $g\in G_{\mathcal{O}}$ will be called a \textbf{non-degenerate critical arrow} of $F$ and its \textbf{index} is defined as $\lambda(g,F):=\lambda(\mathcal{O},f)$.
\end{definition}

Let us fix a nondegenerate critical orbit $\mathcal{O}$ of a  Morse Lie groupoid morphism $F:G\to \mathbb{R}$ covering $f:M\to\mathbb{R}$. We can use the Riemannian metrics $\eta^{(1)}$ and $\eta^{(0)}$ respectively to split $\nu(G_{\mathcal{O}})=\nu_+(G_{\mathcal{O}})\oplus \nu_{-}(G_{\mathcal{O}})$ and $\nu(\mathcal{O})=\nu_+(\mathcal{O})\oplus \nu_{-}(\mathcal{O})$ into subbundles which are fiberwise defined by the eigenvectors corresponding to the positive/negative eigenvalues of $\mathcal{H}(F)$ and $\mathcal{H}(f)$. In these terms we have:

\begin{lemma}\label{NegativeNormalLG}
	The Lie groupoid structure of $\nu(G_{\mathcal{O}})\rightrightarrows \nu(\mathcal{O})$ can be restricted to define two new Lie subgroupoids $\nu_-(G_{\mathcal{O}})\rightrightarrows \nu_-(\mathcal{O})$ and $\nu_+(G_{\mathcal{O}})\rightrightarrows \nu_+(\mathcal{O})$.
\end{lemma}
\begin{proof}
We will show why it is possible to restrict the groupoid structure in the first case since the second one is completely analogous. We already know that $\mathcal{H}_{g}(f\circ s)=\overline{ds}(g)^{-1}\cdot \mathcal{H}_{x}(f)\cdot \overline{ds}(g)$. If $v$ is an eigenvector of $\mathcal{H}_{g}(f\circ s)$ with negative eigenvalue $c$, then
$$\mathcal{H}_{x}(f)(\overline{ds}(g)(v))=\overline{ds}(g)(\mathcal{H}_{g}(f\circ s)(v))=\overline{ds}(g)(c\cdot v)=c\cdot \overline{ds}(g)(v).$$
That is, $\overline{ds}(g)(v)$ is an eigenvector of $\mathcal{H}_{x}(f)$ with eigenvalue $c$. Same conclusion may be obtained by arguing with the Riemannian submersion $t$. Let $v$ and $u$ be eigenvectors of $\mathcal{H}_{g}(f\circ s)$ and $\mathcal{H}_{h}(f\circ s)$ with respective negative eigenvalues $c_1$ and $c_2$ such that $\overline{dm}_{(g,h)}(v,u)$ is well defined. Let us say $z\xleftarrow[]{\it g}x\xleftarrow[]{\it h}y$. Thus, by using the formula $s\circ m=s\circ \pi_2$ we get that 
\begin{eqnarray*}
	\mathcal{H}_{gh}(f\circ s)(\overline{dm}_{(g,h)}(v,u)) & = & \overline{ds}(gh)^{-1}(\mathcal{H}_{y}(f)(\overline{ds}(gh)(\overline{dm}_{(g,h)}(v,u))))\\
	& = & \overline{ds}(gh)^{-1}(\mathcal{H}_{y}(f)(\overline{d(s\circ m)}_{(g,h)}(v,u))) =  \overline{ds}(gh)^{-1}(\mathcal{H}_{y}(f)(\overline{ds}(h)(u)))\\
	& = & \overline{ds}(gh)^{-1}(\overline{ds}(h)(\mathcal{H}_{h}(f\circ s)(u))) =  c_2\cdot \overline{ds}(gh)^{-1}(\overline{ds}(h)(u))\\
	& = & c_2\cdot \overline{ds}(gh)^{-1}(\overline{ds}_{\pi_2(g,h)}((\overline{d\pi_2})_{(g,h)}(v,u)))\\
	& = & c_2\cdot \overline{ds}(gh)^{-1}(\overline{d(s\circ m)}_{(g,h)}(v,u)) =  c_2\cdot \overline{dm}_{(g,h)}(v,u).
\end{eqnarray*}
If we assume that $F=t^\ast f$ then by using the identity $t\circ m=s\circ \pi_1$ we conclude that $\mathcal{H}_{gh}(f\circ t)(\overline{dm}_{(g,h)}(v,u))=c_1\cdot \overline{dm}_{(g,h)}(v,u)$. This computation implies that the composition $\overline{dm}:(\nu_{-}(G_{\mathcal{O}}))^{(2)}\to \nu_{-}(G_{\mathcal{O}})$ is well defined when considering $\overline{ds}$ and $\overline{dt}$ restricted to $\nu_{-}(G_{\mathcal{O}})$. The restriction of the unit map as $\overline{du}:\nu_{-}(\mathcal{O})\to \nu_{-}(G_{\mathcal{O}})$ is also well defined since $s\circ u=id$ holds true. Indeed, if $v$ is an eigenvector of $\mathcal{H}_x(f)$ with negative eigenvalue $c$ we have that
\begin{eqnarray*}
	\mathcal{H}_{1_x}(f\circ s)(\overline{du}(x)(v)) & = & \overline{ds}(1_x)^{-1}(\mathcal{H}_{x}(f)(\overline{ds}(1_x)(\overline{du}(x)(v)))) =  \overline{ds}(1_x)^{-1}(\mathcal{H}_{x}(f)(\overline{d(s \circ u)}(x)(v)))\\
	& = & c\cdot \overline{ds}(1_x)^{-1}(v) =  c\cdot \overline{ds}(1_x)^{-1}(\overline{d(s \circ u)}(x)(v) =  c\cdot \overline{du}(x)(v).
\end{eqnarray*}
So, $\overline{du}(x)(v)$ is an eigenvector of $\mathcal{H}_{1_x}(f\circ s)$ with eigenvalue $c$.
 Finally, with similar computations, using the identities $t=s\circ i$, $s=t\circ i$, and $\mathcal{H}(f\circ s)=\mathcal{H}(f\circ t)$, we obtain that if $v$ is an eigenvector of $\mathcal{H}_g(f\circ s)$ with negative eigenvalue $c$ then $\mathcal{H}_{g^{-1}}(f\circ s)(\overline{di}(g)(v))=c\cdot \overline{di}(g)(v)$. Thus, the restriction of the inverse map as $\overline{di}:\nu_{-}(G_{\mathcal{O}})\to \nu_{-}(G_{\mathcal{O}})$ is also well defined. The properties required to be satisfied by the composition, the inverse, and the unit map follow from those of $\nu(G_{\mathcal{O}})\rightrightarrows \nu(\mathcal{O})$.
\end{proof}

Consider the \textbf{negative unit disk bundle} $D_{-}(G_{\mathcal{O}})$ defined by 

$$D_{-}(G_{\mathcal{O}})=\lbrace v\in \nu_-(G_{\mathcal{O}}): \| v \|_1\leq 1\rbrace,$$
where $\| \cdot \|_1$ is the norm on $\nu_-(G_{\mathcal{O}})$ induced by $\eta^{(1)}$. The positive unit disk bundle $D_{+}(G_{\mathcal{O}})$ is defined accordingly. Also, one has unit disk bundles at the level of objects $D_{-}(\mathcal{O})$ and $D_{+}(\mathcal{O})$ defined by the induced metric $\eta^{(0)}$ on $M$. As both the ranks of $\nu_-(G_{\mathcal{O}})$ and $\nu_-(\mathcal{O})$ agree and they are actually the index $\lambda$ of the non-degenerate critical submanifolds, the fibers of the negative unit disk bundles are $\lambda$-dimensional disks. Moreover, the unit disk bundles define topological groupoids of the (negative) normal groupoid. Indeed:

\begin{lemma}\label{NegativeUnitDiskLG}
	The Lie groupoid structure of $\nu_-(G_{\mathcal{O}})\rightrightarrows \nu_-(\mathcal{O})$ restricts to the unit disk bundle yielding a topological subgroupoid $D_-(G_{\mathcal{O}})\rightrightarrows D_-(\mathcal{O})$.
\end{lemma}
\begin{proof}
	Recall that $s,t:G\to M$ as well as $\pi_1,m,\pi_2: G^{(2)}\to G$ are Riemannian submersions and that the inversion map $i:G\to G$ is an isometry. Therefore, if we consider the norms $\| \cdot \|_1$ and  $\| \cdot \|_0$ with respect to the metrics $\eta^{(1)}$ and $\eta^{(0)}$ restricted to the normal bundles $\nu(G_{\mathcal{O}})$ and $\nu(\mathcal{O})$, respectively, then we get the identities
	$$\| v \|_1=\| \overline{ds}(v) \|_0, \qquad \| v \|_1=\| \overline{dt}(v) \|_0,\qquad \| (v,w) \|_2=\| \overline{dm}(v,w) \|_1=\| \overline{ds}(w) \|_0=\| w \|_1$$
	$$\| (v,w) \|_2=\| \overline{dm}(v,w) \|_1=\| \overline{dt}(v) \|_0=\| v \|_1,\qquad \| \overline{di}(v) \|_1=\| v \|_1,\qquad \| \overline{du}(v) \|_1=\| v \|_0.$$
	To deduce these formulas it is important to have in mind the identities $s\circ m=s\circ \pi_2$, $t\circ m=t\circ \pi_1$, $t=s\circ i$, $s=t\circ i$, and $s\circ u=id$. Hence, by mimicking the steps followed in Lemma \ref{NegativeNormalLG} it is simple to see that $D_-(G_{\mathcal{O}})\rightrightarrows D_-(\mathcal{O})$ is a topological subgroupoid of $\nu_-(G_{\mathcal{O}})\rightrightarrows \nu_-(\mathcal{O})$.
\end{proof}

The groupoid introduced in Lemma \ref{NegativeUnitDiskLG} will be called the \textbf{unit disk groupoid} of $\nu_-(G_{\mathcal{O}})\rightrightarrows \nu_-(\mathcal{O})$ with respect the $2$-metric $\eta^{(2)}$. Note that this can be thought of as a Lie groupoid with ``boundary''. Namely, the boundary of $D_{-}(G_{\mathcal{O}})$ is the unit sphere bundle
$$\partial D_{-}(G_{\mathcal{O}})=\lbrace v\in \nu(G_{\mathcal{O}}): \| v \|_1=1\rbrace,$$
with the natural projection onto $G_{\mathcal{O}}$. The unit sphere bundle $\partial D_{-}(\mathcal{O})$ is similarly defined by using instead the norm $\| \cdot \|_0$ induced by $\eta^{(0)}$. Observe that the fibers of these sphere bundles are indeed $(\lambda-1)$-dimensional spheres. It is simple to check that there is a well defined Lie groupoid $\partial D_-(G_{\mathcal{O}})\rightrightarrows \partial D_-(\mathcal{O})$ whose structural maps are the induced ones. This Lie groupoid will be called the \textbf{unit sphere groupoid} of $\nu_-(G_{\mathcal{O}})\rightrightarrows \nu_-(\mathcal{O})$ with respect the $2$-metric $\eta^{(2)}$ and we shall usually refer to it as the boundary of $D_-(G_{\mathcal{O}})\rightrightarrows  D_-(\mathcal{O})$.

Recall that the action groupoid of the normal representation $G_x\curvearrowright \nu_x(\mathcal{O})$ canonically sits inside the local model $\nu(G_{\mathcal{O}})\rightrightarrows \nu(\mathcal{O})$ and the inclusion $(G_x\ltimes \nu_x(\mathcal{O}) \rightrightarrows \nu_x(\mathcal{O})) \hookrightarrow(\nu(G_{\mathcal{O}})\rightrightarrows \nu(\mathcal{O}))$ is a Morita map. Hence, as consequence of Lemma \ref{NormalInvariance} and the fact that the normal representation acts by isometries when we are equipped with a $2$-metric \cite{dHF}, the following result becomes clear.

\begin{proposition}\label{NormalActionFeatures}
The normal representation \eqref{eq:normalrep} induces a Lie groupoid representation of $G_{\mathcal{O}}\rr \mathcal{O}$ along $\nu_-(\mathcal{O})\to \mathcal{O}$ and a groupoid (resp. Lie groupoid) action of $G_{\mathcal{O}}\rr \mathcal{O}$ along $D_-(\mathcal{O})\to \mathcal{O}$ (resp. $\partial D_-(\mathcal{O})\to \mathcal{O}$). Moreover, the action groupoids of these normal actions
$$G_x\ltimes \nu_-(\mathcal{O})_x\rr \nu_-(\mathcal{O})_x,\qquad G_x\ltimes D_-(\mathcal{O})_x\rr D_-(\mathcal{O})_x,\qquad G_x\ltimes \partial D_-(\mathcal{O})_x\rr \partial D_-(\mathcal{O})_x, $$
canonically sit inside their respective local models and the inclusions are Morita. In particular, there are homeomorphisms between the orbit spaces $D_-(\mathcal{O})_x/G_x\cong D_-(\mathcal{O})/D_-(G_{\mathcal{O}})$ and $\partial D_-(\mathcal{O})_x/G_x\cong \partial D_-(\mathcal{O})/\partial D_-(G_{\mathcal{O}})$.
\end{proposition}

It is simple to see that same conclusions can be obtained about the groupoids defined in terms of positive eigenvalues. 

\section{Gradient vector field and level subgroupoids}\label{S:6}

After having described some of the local features of a Morse Lie groupoid morphism around a non-degenerate critical orbit, in this section we deal with one of the most important results of Morse theory that, in turn, addresses the topological behavior of a Lie groupoid around a non-degenerate critical Lie subgroupoid. We start by introducing a notion of attaching groupoid and then we study multiplicative gradient vector fields.
 
\subsection{Attaching groupoid} 

In this subsection we quickly explain how to construct topological groupoids by an attaching procedure between Lie groupoids.

\begin{definition}\label{def:attachingdata}
 
Let $G\rightrightarrows M$ be a Lie groupoid. A \textbf{groupoid attaching data} on $G$ consists of: 

 \begin{itemize}
 \item[i)] a Lie groupoid $G'\rr M'$,
 \item[ii)] closed submanifolds $\partial G'\subset G'$ and $\partial M'\subset M'$ such that $\partial G'\rightrightarrows \partial M'$ is a Lie subgroupoid of $G'\rr M'$, and
 \item[iii)] a Lie groupoid morphism $(B,b):(\partial G'\rightrightarrows \partial M')\to (G\rightrightarrows M)$.
 \end{itemize}

 \end{definition}

A groupoid attaching data defines two topological spaces $G\sqcup_B G'$ and $M\sqcup_b M'$ defined by the usual attaching construction. Namely,  the quotient space $G\sqcup_B G'$ is defined by taking the disjoint union $G\sqcup G'$ and then identifying $g'\sim B(g')$ for all $g'\in \partial G'$. The attaching space $M\sqcup_b M'$ is defined in the same way.

One immediately observes that there is a natural topological groupoid $G\sqcup G'\rr M\sqcup M'$ whose structural maps are defined as the disjoint union of the corresponding structural maps of $G$ and $G'$. Our main goal in what follows is to show that this groupoid structure descends to the attaching spaces giving rise to a topological groupoid $G\sqcup_B G'\rr M\sqcup_b M'$. Since $(B,b):(\partial G'\rightrightarrows \partial M')\to (G\rightrightarrows M)$ is a Lie groupoid morphism, it follows that the source and target maps $s\sqcup s', t\sqcup t':G\sqcup G'\to M\sqcup M'$ pass to the quotient, yielding surjective open maps

$$\overline{s},\overline{t}:G\sqcup_B G'\to M\sqcup_b M'.$$
%


As usual, we define the set of composable arrows $(G\sqcup_B G')^{(2)}$ as the fibered product induced by $\overline{s}$ and $\overline{t}$. Therefore, to define the composition map $\overline{m}:(G\sqcup_B G')^{(2)}\to G\sqcup_B G'$ we have to consider the following four cases. Recall that $(1,x')\sim(b(x'),2)$ are the only kinds of elements related in $M\sqcup M'$ so that we set 
$$\overline{m}([(g,2)]_B,[(h,2)]_B):=[(gh,2)]_B,\quad \overline{m}([(1,g')]_B,[(1,h')]_B):=[(1,g'h')]_B,$$

$$\overline{m}([(1,g')]_B,[(g,2)]_B):=[(B(g')g,2)]_B,\quad \overline{m}([(g,2)]_B,[(1,g')]_B):=[(gB(g'),2)]_B.$$


It is simple to check that the well definition of this composition follows from the fact that $\overline{s}$ and $\overline{t}$ are also well defined and $(B,b)$ is a Lie groupoid morphism. The associative property of $\overline{m}$ is satisfied because we have that $m$ and $m'$ are associative in $G$ and $G'$, respectively and $B:\partial G'\to G$ satisfies $B\circ m'=m\circ (B\times B)$. Indeed, we only have to be careful when verifying the following case:
\begin{eqnarray*}
	[(1,g')]_B\cdot ([(1,h')]_B\cdot [(g,2)]_B) & = & [(1,g')]_B\cdot[(B(h')g,2)]_B = [(B(g')(B(h')g),2)]_B\\
	& = & [((B(g')B(h'))g,2)]_B =  [(B(g'h')g,2)]\\
	& = & [(1,g'h')]_B\cdot [(g,2)]_B =  ([(1,g')]_B\cdot [(1,h')]_B)\cdot [(g,2)]_B.
\end{eqnarray*}
The case $[(g,2)]_B\cdot ([(1,g')]_B\cdot [(1,h')]_B)=([(g,2)]_B\cdot[(1,g')]_B)\cdot [(1,h')]_B$ may be verified in a similar fashion and the other ones follow more directly. 

As expected, the unit map $\overline{u}:M\sqcup_b M' \to G\sqcup_B G'$ and the inverse  $\overline{i}:G\sqcup_B G'\to G\sqcup_B G'$ are defined by passing to the quotient $u\sqcup u':M\sqcup M'\to G\sqcup G'$ and $i\sqcup i':G\sqcup G'\to G\sqcup G'$, respectively. It follows easily that both $\overline{u}$ and $\overline{i}$ satisfy the required conditions of the groupoid axioms.

 Summing up, we have obtained:
 \begin{proposition}\label{AttachingDiskProposition}
 	There exists a natural topological groupoid $G\sqcup_B G'\rightrightarrows M\sqcup_b M'$ whose structural maps are given by the ones $(\overline{s},\overline{t},\overline{m},\overline{u},\overline{i})$ defined as above.
 \end{proposition}
 This groupoid will be called the \textbf{attaching groupoid} of $G\rightrightarrows M$ with respect to the Lie groupoid morphism $(B,b)$. A very special case is obtained when taking both $G'=M'=D^\lambda$ a closed $\lambda$-disk and $\partial G'=\partial M'=\partial D^{\lambda}$ its corresponding $(\lambda-1)$-sphere with their underlying structure of unit groupoids. In this case we attach cells of the same dimension at both arrows and objects extending the groupoid structure. 
 
 
\subsection{Level subgroupoids}
 
We start this subsection by defining the gradient vector field of a Lie groupoid morphism  $F:G\to \mathbb{R}$ with respect to a $2$-metric. Let $G\rr M$ be a Lie groupoid equipped with a 2-metric $\eta^{(2)}$ and consider a Lie groupoid morphism $F:G\to \mathbb{R}$ covering a basic function $f:M\to \mathbb{R}$. The \textbf{gradient} of $F$ is defined as the pair $\nabla F=(\nabla(s^\ast f),\nabla f)$, where $\nabla(s^\ast f)$ and $\nabla f$ are the gradient vector fields of $s^\ast f$ and $f$ with respect to the induced $1$-metric $\eta^{(1)}$ on $G$ and the induced $0$-metric $\eta^{(0)}$ on $M$, respectively. We have the following result.

\begin{proposition}\label{Multiplicative-Gradient}
 The gradient $\nabla F$ is a multiplicative vector field on $G$.
\end{proposition}
 \begin{proof}
 	We know that $s,t:G\to M$ as well as $\pi_1,m,\pi_2: G^{(2)}\to G$ are Riemannian submersions. Thus, the vector fields $\nabla(f\circ s)$ and $\nabla f$ are $s$-related, $\nabla(f\circ t)$ and $\nabla f$ are $t$-related, and $\nabla((f\circ s)\circ m)$ and $\nabla(f\circ s)$ are $m$-related. Here $\nabla((f\circ s)\circ m)$ denotes the gradient vector field of $(f\circ s)\circ m$ with respect to $\eta^{(2)}$ on $G^{(2)}$ and we have that $dm \circ\nabla((f\circ s)\circ m)=\nabla(f\circ s)\circ m$. The crucial point is to prove that for $(g,h)\in G^{(2)}$ the identity
 $$\nabla((f\circ s)\circ m)(g,h)=(\nabla(f\circ s)(g),\nabla(f\circ t)(h)),$$
 holds true as long as $ds(g)(\nabla(f\circ s)(g))=\nabla f(s(g))=\nabla f(t(h))=dt(h)(\nabla(f\circ t)(h))$. However, the vector fields $\nabla((f\circ s)\circ \pi_1)$ and $\nabla(f\circ s)$ are $\pi_1$-related and the vector fields $\nabla((f\circ t)\circ \pi_2)$ and $\nabla(f\circ t)$ are $\pi_2$-related. Therefore, as $(f\circ s)\circ m=(f\circ s)\circ \pi_1=(f\circ t)\circ \pi_2$ since $f$ is basic, we obtain that
 $$ d\pi_1 \circ\nabla((f\circ s)\circ m)=\nabla(f\circ s)\circ \pi_1\quad\textnormal{and}\quad d\pi_2 \circ\nabla((f\circ s)\circ m)=\nabla(f\circ t)\circ \pi_2,$$
 
as desired. This clearly implies that $dm \circ(\nabla(f\circ s)\times \nabla(f\circ s))=\nabla(f\circ s)\circ m$ which completes the proof.
\end{proof}
It is clear that we also may define the gradient vector field as the pair $\nabla F=(\nabla(t^\ast f),\nabla f)$.
 \begin{remark}\label{Multiplicative-Gradient-Remark}
 	To have that $\nabla F$ is a multiplicative vector field it is very important to us since this is equivalent to requiring that the pair of flows $(\Phi^{\nabla(s^\ast f)}_{\tau},\Phi^{\nabla f}_{\tau})$ induces (local) automorphisms on the Lie groupoid; see \cite{MX}. Namely, the following identities hold true
	
 	$$s\circ \Phi^{\nabla(s^\ast f)}_{\tau}=\Phi^{\nabla f}_{\tau}\circ s,\qquad  t\circ \Phi^{\nabla(t^\ast f)}_{\tau}=\Phi^{\nabla f}_{\tau}\circ t, \qquad \Phi^{\nabla (f\circ s)}_{\tau}\circ m=m\circ (\Phi^{\nabla (f\circ s)}_{\tau}\times\Phi^{\nabla (f\circ s)}_{\tau}).$$ 
 \end{remark}

 Let us now define the level subgroupoids of the kind of Lie groupoid morphisms we are working with. Let $F:G\to \mathbb{R}$ be a Lie groupoid morphism covering a basic function $f:M\to \mathbb{R}$. Take $a\in \mathbb{R}$ and consider the level set $M^a=\lbrace x\in M:\ f(x)\leq a\rbrace$. It is simple to check that $s^{-1}(M^a)=t^{-1}(M^a)$ which is equivalent to saying that $M^a$ is a saturated submanifold. This implies that $G^a=G_{M^a}=s^{-1}(M^a)=t^{-1}(M^a)$, with $G^a$ denoting the level set of $F$ below $a$, so that we get a well defined topological subgroupoid $G^a\rightrightarrows M^a$ of $G\rightrightarrows M$ that we call the \textbf{level subgroupoid} of $F$ below $a$. It is important to notice that $G^a\rightrightarrows M^a$ can be thought of as a Lie groupoid with boundary in the sense that if $\partial M^a=\lbrace x\in M:\ f(x)=a\rbrace$ and $\textnormal{int}(M^a)=\lbrace x\in M:\ f(x)<a\rbrace$
 then both $\partial G^a \rightrightarrows \partial M^a$ and $\textnormal{int}(G^a) \rightrightarrows \textnormal{int}(M^a)$ are clearly Lie subgroupoids of $G\rightrightarrows M$.
 
 It is well known that the gradient vector field $X=\nabla f$ of a smooth function $f$ on a Riemannian manifold $(M,\eta^{(0)})$ is actually a \textbf{gradient-like vector field}. That is, it satisfies both $\textnormal{Zeroes}(X)=\textnormal{Crit}(f)$ and $df(X)>0$ on $M\backslash\textnormal{Crit}(f)$. Thus, as consequence of Proposition \ref{Multiplicative-Gradient} we obtain:


 \begin{proposition}\label{NonCriticalLevel}
 Suppose that $G\rr M$ is a proper groupoid, $[a,b]$ is a closed interval which does not contain critical values of $f$ and $f^{-1}[a,b]$ is compact. Then $G^a\rightrightarrows M^a$ and $G^b\rightrightarrows M^b$ are isomorphic groupoids. Furthermore, $G^a\rightrightarrows M^a$ is a deformation retraction of $G^b\rightrightarrows M^b$.
 \end{proposition}
 
 \begin{proof}
 
Let us follow some ideas stated in \cite{Mi,Ni} about the proof of this result in the classical case. It is clear that if $f^{-1}[a,b]\cap \textnormal{Crit}(f)=\emptyset$, then $(s^\ast f)^{-1}[a,b]\cap \textnormal{Crit}(F)=\emptyset$. We know that $\widetilde{X}=\nabla(f\circ s)$ and $X=\nabla f$ are gradient-like vector fields on $G$ and $M$, respectively. Thus, consider the smooth function $\mu:M\to [0,\infty)$ which is defined by $\| Xf\|^{-1}$ in $f^{-1}[a,b]$ and that vanishes outside a compact neighborhood of this set. We can similarly construct a smooth function $\widetilde{\mu}:G\to [0,\infty)$ by using instead $\widetilde{X}$ and $(s^\ast f)^{-1}[a,b]$. On the one hand, observe that since $\nabla F=(\widetilde{X},X)$ is a multiplicative vector field we get that
 	\begin{eqnarray*}
 		\mu(s(g))& = & \| (Xf)(s(g))\|^{-1}=\| df(s(g))(X(s(g))\|^{-1}\\
 		& = & \| df(s(g))(ds(g)(\widetilde{X}(g)))\|^{-1}=\| d(f\circ s)(g)(\widetilde{X}(g))\|^{-1}\\
 		& = & \| (\widetilde{X}F)(g)\|^{-1}=\widetilde{\mu}(g),
 	\end{eqnarray*}
 	what means that we have $\widetilde{\mu}=\mu\circ s$. On the other hand, let us consider the vectors fields $-\widetilde{\mu}\widetilde{X}$ and $-\mu X$. From \cite[Lem. 2.4]{Mi} and \cite[Lem. 4.4]{CMS} we obtain that they are complete vector fields since the pair $(-\widetilde{\mu}\widetilde{X},-\mu X)$ defines again a multiplicative vector field. Indeed, as consequence of Proposition \ref{Multiplicative-Gradient} we get
 	$$(ds\circ \widetilde{\mu}\widetilde{X})(g)=\widetilde{\mu}(g)ds(g)(\widetilde{X}(g))=\widetilde{\mu}(g)X(s(g))=\mu(s(g))X(s(g))=(\mu X\circ s)(g).$$
 	Given that $\nabla(f\circ s)=\nabla(f\circ t)$ we actually have that $\widetilde{\mu}=\mu\circ s=\mu\circ t$ and thus we can analogously obtain that $dt\circ \widetilde{\mu}\widetilde{X}=\mu X\circ t$. To prove the identity regarding the composition map observe that the formulas $s\circ m=s\circ \pi_2$ and $t\circ m=t\circ \pi_1$ imply that $\widetilde{\mu}(m(g,h))=\widetilde{\mu}(g)=\widetilde{\mu}(h)$ for all $(g,h)\in G^{(2)}$. Therefore, 
 	\begin{eqnarray*}
 		(dm\circ (\widetilde{\mu}\widetilde{X}\times \widetilde{\mu}\widetilde{X}))(g,h) &=& dm_{(g,h)}(\widetilde{\mu}(g)\widetilde{X}(g), \widetilde{\mu}(h)\widetilde{X}(h))\\
 		&=& \widetilde{\mu}(m(g,h))dm_{(g,h)}(\widetilde{X}(g),\widetilde{X}(h))\\
 		&=& \widetilde{\mu}(m(g,h))\widetilde{X}(m(g,h))\\
 		&=& (\widetilde{\mu}\widetilde{X} \circ m)(g,h).
 	\end{eqnarray*}
 	For all $\tau\in \mathbb{R}$, let $\widetilde{\Phi}_{\tau}:G\to G$ and $\Phi_{\tau}:M\to M$ denote the respective flows generated by the pair $(-\widetilde{\mu}\widetilde{X},-\mu X)$. From \cite{MX} we know that $(\widetilde{\Phi}_{\tau},\Phi_{\tau})$ is a Lie groupoid isomorphism for all $\tau\in \mathbb{R}$. As in the classical case (see \cite{Mi,Ni}) we get diffeomorphisms
 	$$\Phi_{b-a}(M^b)=M^a,\qquad \Phi_{a-b}(M^a)=M^b\quad\textnormal{and}\quad \widetilde{\Phi}_{b-a}(G^b)=G^a,\qquad \widetilde{\Phi}_{a-b}(G^a)=G^b$$
 	so that $(\widetilde{\Phi}_{b-a},\Phi_{b-a}):(G^b\rightrightarrows M^b)\to(G^a\rightrightarrows M^a)$ defines a Lie groupoid isomorphism with obvious inverse $(\widetilde{\Phi}_{a-b},\Phi_{a-b})$.
 	Finally, we can define two deformation retractions $H:[0,1]\times M^b\to M^b$ and $\widetilde{H}:[0,1]\times G^b\to G^b$ respectively as
 	$$H(\tau,x)=\Phi_{\tau\cdot(f(x)-a)^+}(x)\qquad\textnormal{and}\qquad \widetilde{H}(\tau,g)=\widetilde{\Phi}_{\tau\cdot((s^\ast f)(g)-a)^+}(g).$$ 

 	Here $r^+:=\max\lbrace r,0\rbrace$ for every real number $r$. Hence, we obtain again a Lie groupoid morphism which allows us to conclude that $G^a\rightrightarrows M^a$ is a deformation retraction of $G^b\rightrightarrows M^b$.
 \end{proof}

 Let us now analyze the case in which $f^{-1}[a,b]\cap \textnormal{Crit}(f)\neq \emptyset$. The main reference for the basic ideas we will be following in this case is for instance \cite[App. B]{F}. Let $\mathcal{O}$ be a non-degenerate critical orbit of $f$. We will denote by $\xi_0:\nu(\mathcal{O})\to \nu_-(\mathcal{O})$ and $\eta_0:\nu(\mathcal{O})\to \nu_+(\mathcal{O})$ the two mutually complementary projections. They induce bundle morphisms between the bundle projections $\pi:\nu(\mathcal{O})\to \mathcal{O}$ and $\pi_{\pm}:\nu_\pm(\mathcal{O})\to \mathcal{O}$, respectively. Moreover, for every $v\in \nu(\mathcal{O})$ we have that $v=\xi_0(v)+\eta_0(v)$ and the expression
 $$\| v\|_0^2=-Q_f(\xi_0(v))+Q_f(\eta_0(v)),$$
 defines a positive definite quadratic form (i.e. a norm) on $\nu(\mathcal{O})$. As expected, we can define respective mutually complementary projections $\xi_1:\nu(G_{\mathcal{O}})\to \nu_-(G_{\mathcal{O}})$ and $\eta_1:\nu(G_{\mathcal{O}})\to \nu_+(G_{\mathcal{O}})$ which enjoy of similar properties as above. Namely, every $\widetilde{v}\in \nu(G_{\mathcal{O}})$ can be rewritten as  $\widetilde{v}=\xi_1(\widetilde{v})+\eta_1(\widetilde{v})$ and we have the norm $\| \widetilde{v}\|_1^2=-Q_F(\xi_1(\widetilde{v}))+Q_F(\eta_1(\widetilde{v}))$ on $\nu(G_{\mathcal{O}})$. Furthermore:
 
 \begin{lemma}\label{NormalInvariance2}

 	Both $(\xi_1,\xi_0): (\nu(G_{\mathcal{O}})\rightrightarrows \nu(\mathcal{O}))\to (\nu_-(G_{\mathcal{O}})\rightrightarrows \nu_-(\mathcal{O}))$ and $(\eta_1,\eta_0): (\nu(G_{\mathcal{O}})\rightrightarrows \nu(\mathcal{O}))\to (\nu_+(G_{\mathcal{O}})\rightrightarrows \nu_+(\mathcal{O}))$ are Lie groupoid morphisms. Moreover, for all $x\in \mathcal{O}$ it holds that $\xi_0$ and $\eta_0$ are $G_x$-equivariant with respect to the action induced by the normal representation \eqref{eq:normalrep}.
 \end{lemma}
 
 \begin{proof}
 	If $\widetilde{v}\in \nu(G_{\mathcal{O}_x})$ then we get
 	$$(\xi_0\circ \overline{ds})(\widetilde{v}) = (\xi_0\circ \overline{ds})(\xi_1(\widetilde{v})+\xi_1(\widetilde{v})) = \xi_0(\overline{ds}(\xi_1(\widetilde{v}))+\overline{ds}(\eta_1(\widetilde{v})))=\overline{ds}(\xi_1(\widetilde{v})).$$
 	We can analogously prove that $\xi_0\circ \overline{dt}=\overline{dt}\circ \xi_1$. Now, if $\widetilde{v}, \widetilde{u}\in \nu(G_{\mathcal{O}_x})$ are such that $\overline{dm}(\widetilde{v},\widetilde{u})$ is defined then 
 	\begin{eqnarray*}
 		(\xi_1\circ \overline{dm})(\widetilde{v},\widetilde{u}) &=& (\xi_1\circ \overline{dm})(\xi_1(\widetilde{v})+\eta_1(\widetilde{v}),\xi_1(\widetilde{u})+\eta_1(\widetilde{u}))\\
 		&=& \xi_1(\overline{dm}(\xi_1(\widetilde{v}),\xi_1(\widetilde{u}))+\overline{dm}(\eta_1(\widetilde{v}),\eta_1(\widetilde{u})))=\overline{dm}(\xi_1(\widetilde{v}),\xi_1(\widetilde{u})),
 	\end{eqnarray*}
 	so that $\xi_1 \circ \overline{dm}=\overline{dm}\circ (\xi_1\times \xi_1)$. 
 	
 	The fact that $(\eta_1,\eta_0)$ is a Lie groupoid morphism may be shown in a similar way and the $G_x$-equivariance of $\xi_0$ and $\eta_0$ is consequence of Lemma \ref{NormalInvariance}.
 \end{proof}
 With this in mind we have:
 \begin{theorem}\label{CriticalLevel}
Suppose that $G\rr M$ is a proper groupoid and that $[a,b]$ is a closed interval such that $f^{-1}[a,b]$ is compact and the only non-degenerate critical orbit inside $f^{-1}(a,b)$ is $\mathcal{O}$. Then $G^{a}\cup_{B} D_-(G_{\mathcal{O}})\rightrightarrows M^{a}\cup_{b} D_-(\mathcal{O})$ is a deformation retraction of $G^{b}\rightrightarrows M^{b}$. 
 \end{theorem}
 \begin{proof}
 	Given that $\mathcal{O}\subset f^{-1}(a,b)\subset M$ is closed and embedded we get that $\mathcal{O}$ is also compact. Furthermore, $(s^\ast f)^{-1}[a,b]$ is also compact since $G$ is proper so that $G_{\mathcal{O}}$ is the only compact non-degenerate critical submanifold inside $(s^\ast f)^{-1}(a,b)$.

 	Consider the mutually complementary projections $\xi_1:\nu(G_{\mathcal{O}})\to \nu_-(G_{\mathcal{O}})$, $\eta_1:\nu(G_{\mathcal{O}})\to \nu_+(G_{\mathcal{O}})$ and $\xi_0:\nu(\mathcal{O})\to \nu_-(\mathcal{O})$, $\eta_0:\nu(\mathcal{O})\to \nu_+(\mathcal{O})$ as well as the norms
 	$$\| \widetilde{v}\|_1^2=-Q_F(\xi_1(\widetilde{v}))+Q_F(\eta_1(\widetilde{v}))\qquad \textnormal{and}\qquad\| v\|_0^2=-Q_f(\xi_0(v))+Q_f(\eta_0(v)),$$
 	for all $\widetilde{v}\in \nu(G_{\mathcal{O}})$ and $v\in \nu(\mathcal{O})$. Because of Theorem \ref{LieGroupoidMorseLemma} there is a full Lie groupoid tubular neighborhood $\phi: (\nu(G_{\mathcal{O}})_V\rr V)\xrightarrow[]{\cong} (G_U\rightrightarrows U)$ such that $(\phi_1^\ast F,\phi_0^\ast f)=(c+Q_F,c+Q_f)$,  	where $c=f(\mathcal{O})=F(G_{\mathcal{O}})$ is the common value of $f$ and $F$ on $G_{\mathcal{O}}$ and $\mathcal{O}$, respectively. By identifying $\nu(G_{\mathcal{O}})_V\cong G_{U}$ and $V\cong U$ we may think of $F$ on $G_{U}$ and $f$ on $U$ as respectively given by
 	$$F(\widetilde{v})=c+Q_F(\widetilde{v})=c-\| \xi_1(\widetilde{v})\|_1^2+\| \eta_1(\widetilde{v})\|_1^2,\qquad  f(v)=c+Q_f(v)=c-\| \xi_0(v)\|_0^2+\| \eta_0(v)\|_0^2.$$
 	Let us now follow some ideas stated in \cite{F,Mi,Ni} about the proof of this result in the classical case.  We choose $\epsilon>0$ small enough so that the interval $(c-\epsilon,c+\epsilon)$ is contained in $[a,b]$ and all points $\widetilde{v}\in \nu(G_{\mathcal{O}})$ with $\| \widetilde{v}\|_1^2\leq 2\epsilon$ belong to the neighborhood $ G_{U}$ and $v\in \nu(\mathcal{O})$ with $\| v\|_0^2\leq 2\epsilon$ belong to the neighborhood $U$. Similarly to how it was done in \cite[p. 16-19]{Mi} we construct two smooth functions which are a modification of the pair $(F,f)$. Namely, let $\mu:\mathbb{R}\to \mathbb{R}$ be the smooth function verifying the conditions
 	\begin{eqnarray*}
 		\mu(0)>\epsilon, &  &\\
 		\mu(r)=0 &  & \textnormal{for all}\ r\geq 2\epsilon,\\
 		-1<\mu'(r)\leq 0 &  & \textnormal{for all}\ r.
 	\end{eqnarray*}
 	Consider the functions $F_1:G\to \mathbb{R}$ and $f_1:M\to \mathbb{R}$ which respectively coincide with $F$ and $f$ outside of $ G_{U}$ and $U$ but within those neighborhoods they are given as
 	$$F_1(\widetilde{v})=F(\widetilde{v})-\mu(\| \xi_1(\widetilde{v})\|_1^2+2\| \eta_1(\widetilde{v})\|^2_1)=c-\| \xi_1(\widetilde{v})\|_1^2+\| \eta_1(\widetilde{v})\|_1^2-\mu(\| \xi_1(\widetilde{v})\|_1^2+2\| \eta_1(\widetilde{v})\|^2_1),$$
 	and
 	$$f_1(v)=f(v)-\mu(\| \xi_0(v)\|_0^2+2\| \eta_0(v)\|^2_1)=c-\| \xi_0(v)\|_0^2+\| \eta_0(v)\|_0^2-\mu(\| \xi_0(v)\|_0^2+2\| \eta_0(v)\|^2_1).$$

 	These are well defined smooth functions. Let us prove that $(F_1,f_1):(G\rightrightarrows M)\to (\mathbb{R}\rightrightarrows \mathbb{R})$ is also a Lie groupoid morphism. Since $\phi$ is a Lie groupoid isomorphism and $(F_1,f_1)$ agrees with $(F,f)$ outside $( G_{U},U)$ we only have to see what happens inside $( G_{U},U)$. As consequence of the classical Morse--Bott lemma we may identify the norms $\| \cdot \|_1$ and $\| \cdot \|_0$ with those norms defined by the restrictions of $\eta^{(1)}$ and $\eta^{(0)}$ on $\nu(G_{\mathcal{O}})$ and $\nu(\mathcal{O})$, respectively. As $s,t:G\to M$ are Riemannian submersions we obtain that $\| \widetilde{v} \|_1=\| ds(\widetilde{v}) \|_0$ and $\| \widetilde{v} \|_1=\| dt(\widetilde{v}) \|_0$. It is worth noticing that if we do not identify the norms as we did above we can also get the previous statement by using the fact that $(Q_F,Q_f)$, $(\xi_1,\xi_0)$, and $(\eta_1,\eta_0)$ are Lie groupoid morphisms. Therefore, by using again the latter fact we obtain that
 	\begin{eqnarray*}
 		(\overline{ds}^\ast f_1)(\widetilde{v}) &=& f_1(\overline{ds}(\widetilde{v}))=c+Q_f(\overline{ds}(\widetilde{v}))-\mu(\| \xi_0(\overline{ds}(\widetilde{v}))\|_0^2+2\| \eta_0(\overline{ds}(\widetilde{v})\|^2_0)\\
 		& =& c+(Q_f\circ \overline{ds})(\widetilde{v})-\mu(\| \overline{ds}(\xi_1(\widetilde{v}))\|_0^2+2\| \overline{ds}(\xi_1(\widetilde{v})\|^2_0)\\
 		&= & c+Q_F(\widetilde{v})-\mu(\| \xi_1(\widetilde{v})\|_1^2+2\| \eta_1(\widetilde{v})\|^2_1)=F_1.
 	\end{eqnarray*}
 	We can analogously prove that $\overline{dt}^\ast f_1=F_1$ which implies what we desired.
 	
 	Just as it was proved in \cite[p. 16-19]{Mi} we have the following assertions.
 	\begin{itemize}
 		\item The regions $F_1^{-1}(-\infty,c+\epsilon]$ and $f_1^{-1}(-\infty,c+\epsilon]$ coincide with the regions $G^{c+\epsilon}$ and $M^{c+\epsilon}$, respectively.
 		\item The functions $F$ and $F_1$ have the same critical arrows. Similarly, the functions $f$ and $f_1$ have the same critical points.
 		\item The sets $F_1^{-1}[c-\epsilon,c+\epsilon]$ and $f_1^{-1}[c-\epsilon,c+\epsilon]$ contain no critical points of $F_1$ and $f_1$, respectively.
 	\end{itemize} 
 	Thus, from Proposition \ref{NonCriticalLevel} we conclude that the level subgroupoid $F_1^{-1}(-\infty,c-\epsilon]\rightrightarrows f_1^{-1}(-\infty,c-\epsilon]$ is a deformation retraction of $G^{c+\epsilon}\rightrightarrows M^{c+\epsilon}$.
 	
 	Let us now consider the negative disk bundle $D_-(\mathcal{O})=\lbrace v\in U: \| \xi_0(v)\|_0^2\leq \epsilon,\ \eta_0(v)=0\rbrace$. As the value of $f_1$ on any point of $D_-(\mathcal{O})$ is less than $c-\epsilon$ we have that $M^{c-\epsilon}\cup D_-(\mathcal{O})\subset f_1^{-1}(-\infty,c-\epsilon]$. Moreover, the intersection $D_-(\mathcal{O})\cap M^{c-\epsilon}$ coincides with the negative sphere bundle $\partial D_-(\mathcal{O})=\lbrace v\in U: \| \xi_0(v)\|_0^2=\epsilon,\ \eta_0(v)=0\rbrace$. Therefore, we can consider the attaching space $M^{c-\epsilon}\cup_{b} D_-(\mathcal{O})$ with respect to the bundle projection $b:\partial D_-(\mathcal{O})\to M^{c-\epsilon}$ which actually we may be seen as the union $M^{c-\epsilon}\cup D_-(\mathcal{O})$ inside $f_1^{-1}(-\infty,c-\epsilon]$. As it was argued in \cite[p. 18-19]{Mi}, we have that the union $M^{c-\epsilon}\cup D_-(\mathcal{O})$ is a deformation retract of $f_1^{-1}(-\infty,c-\epsilon]$. The deformation retraction $r_\tau:f_1^{-1}(-\infty,c-\epsilon]\to f_1^{-1}(-\infty,c-\epsilon]$ is identical outside $U$ but within $U$ it acts as follows.
 	\begin{itemize} 
 		\item In the domain $\| \xi_0(v)\|^2_0\leq \epsilon$ (i.e. in $D_-(\mathcal{O})$) the deformation $r_t$ is given by the formula
 		$$r_\tau(v)=\xi_0(v)+\tau\eta_0(v).$$
 		\item In the domain $\epsilon \leq \| \xi_0(v)\|^2_0\leq \|\eta_0(v)\|_0^2+\epsilon$ we define $r_\tau$ by 
 		$$r_\tau(v)=\xi_0(v)+s_\tau(v)\eta_0(v),$$
 		where the number $s_\tau(v)\in [0,1]$ is defined by $\displaystyle s_\tau(v)=\tau+(1-\tau)\sqrt{\dfrac{\| \xi_0(v)\|^2_0-\epsilon}{\| \eta_0(v)\|^2_0}}$. The map $r_0$ takes values in $f^{-1}(c-\epsilon)$.
 		\item Within the domain $\|\eta_0(v)\|_0^2+\epsilon\leq\|\xi_0(v)\|_0^2$ (i.e. in $M^{c-\epsilon}$) we set $r_\tau$ to be the identity map, $\tau\in [0,1]$.
 	
 	\end{itemize}
 
 	The expressions above agree on the intersection of the three domains and thus they define a continuous map $r_\tau$ so that $r_1$ is the identity map and $r_0$ is a retraction of $f_1^{-1}(-\infty,c-\epsilon]$ onto $M^{c-\epsilon}\cap D_-(\mathcal{O})$. As expected, we can similarly define the attaching space $G^{c-\epsilon}\cup_{B} D_-(G_{\mathcal{O}})$ by using the bundle projection $B:\partial D_-(G_{\mathcal{O}})\to G^{c-\epsilon}$ as well as construct a deformation retraction $\widetilde{r_\tau}:F_1^{-1}(-\infty,c-\epsilon]\to F_1^{-1}(-\infty,c-\epsilon]$ by using the mutually complementary projections $\xi_1$, $\eta_1$ and the norm $\| \cdot \|_1$ to produce same formulas as we did above.
 	
 	The key point now is to prove that $(\widetilde{r_\tau},r_\tau)$ defines a Lie groupoid morphism. As $\widetilde{r_\tau}$ and $r_\tau$ are the identity outside $ G_{U}$ and $U$, respectively, we only have to check this on the three special cases mentioned above. We will use again the fact that $(\xi_1,\xi_0)$ and $(\eta_1,\eta_0)$ are Lie groupoid morphisms. In the first domain we have that
 	$$(r_\tau\circ \overline{ds})(\widetilde{v})=\xi_0(\overline{ds}(\widetilde{v}))+\tau\eta_0(\overline{ds}(\widetilde{v}))=\overline{ds}(\xi_1(\widetilde{v}))+\tau\overline{ds}(\eta_1(\widetilde{v}))=(\overline{ds}\circ \widetilde{r_\tau})(\widetilde{v}).$$
 	We can analogously prove that $r_\tau\circ \overline{dt}=\overline{dt}\circ \widetilde{r_\tau}$. Now, if $\widetilde{v}$ and $\widetilde{u}$ are such that $dm(\widetilde{v},\widetilde{u})$ is defined then
 	\begin{eqnarray*}
 		\widetilde{r_\tau}(\overline{dm}(\widetilde{v},\widetilde{u})) &=& \xi_1(\overline{dm}(\widetilde{v},\widetilde{u}))+\tau\eta_1(\overline{dm}(\widetilde{v},\widetilde{u}))=\overline{dm}(\xi_1(\widetilde{v}),\xi_1(\widetilde{u}))+\tau\overline{dm}(\eta_1(\widetilde{v}),\eta_1(\widetilde{u}))\\
 		&=& \overline{dm}((\xi_1(\widetilde{v}),\xi_1(\widetilde{u}))+\tau(\eta_1(\widetilde{v}),\eta_1(\widetilde{u})))= \overline{dm}(\xi_1(\widetilde{v})+\tau\eta_1(\widetilde{v}),\xi_1(\widetilde{u})+\tau\eta_1(\widetilde{u}))\\
 		&=& \overline{dm}(\widetilde{r_\tau}(\widetilde{v}),\widetilde{r_\tau}(\widetilde{u})),
 	\end{eqnarray*}
 	what means that $\widetilde{r_\tau}\circ \overline{dm}=\overline{dm}\circ (\widetilde{r_\tau}\times \widetilde{r_\tau})$. To verify the assertion in the second domain we only have to check that $s_\tau(\overline{ds}(\widetilde{v}))=\widetilde{s_\tau}(\widetilde{v})$, $s_\tau(\overline{dt}(\widetilde{v}))=\widetilde{s_\tau}(\widetilde{v})$ and $\widetilde{s_\tau}(\widetilde{v})=\widetilde{s_\tau}(\overline{dm}(\widetilde{v},\widetilde{u}))=\widetilde{s_\tau}(\widetilde{u})$. These formulas follows from the fact that $(\xi_1,\xi_0)$ and $(\eta_1,\eta_0)$ are Lie groupoid morphisms plus the identities
 	$$\| \overline{dm}(\widetilde{v},\widetilde{u}) \|_1=\| \overline{ds}(\widetilde{u}) \|_0=\| \widetilde{u} \|_1\qquad\textnormal{and}\qquad \| \overline{dm}(\widetilde{v},\widetilde{u}) \|_1=\| \overline{dt}(\widetilde{v}) \|_0=\| \widetilde{v} \|_1$$
 	which can be obtained by either the fact that $s$, $t$, $m$, $\pi_1$ and $\pi_2$ are Riemannian submersions or $(Q_F,Q_f)$, $(\xi_1,\xi_0)$, $(\eta_1,\eta_0)$ are Lie groupoid morphisms. The remaining computations are similar to those done when looking at the first domain. Finally, in the third domain the assertion directly follows since $\widetilde{r_\tau}$ and $r_\tau$ are the identity maps over there.
 	
 	In conclusion, the topological groupoid $G^{c-\epsilon}\cup_{B} D_-(G_{\mathcal{O}})\rightrightarrows M^{c-\epsilon}\cup_{b} D_-(\mathcal{O})$ is a deformation retraction of the level subgroupoid $F_1^{-1}(-\infty,c-\epsilon]\rightrightarrows f_1^{-1}(-\infty,c-\epsilon]$ which in turn, as consequence of Proposition \ref{NonCriticalLevel}, is a deformation retraction of the level subgroupoid $G^{c+\epsilon}\rightrightarrows M^{c+\epsilon}$. Hence, $G^{c-\epsilon}\cup_{B} D_-(G_{\mathcal{O}})\rightrightarrows M^{c-\epsilon}\cup_{b} D_-(\mathcal{O})$ is a deformation retraction of $G^{c+\epsilon}\rightrightarrows M^{c+\epsilon}$.
 \end{proof}

\begin{remark}
It is important to point out that the assumption from Proposition \ref{NonCriticalLevel} and Theorem \ref{CriticalLevel} asking for $f^{-1}[a,b]$ to be compact in $M$ may be relaxed just by requiring $\overline{f}^{-1}[a,b]$ to be compact in the orbit space $M/G$. Here $\overline{f}:M/G\to \mathbb{R}$ denotes the underlying continuous function defined through the basic function $f:M\to \mathbb{R}$. This fact will be clarified in Section \ref{S:7} where we will study some Morse theoretical features over differentiable stacks.
\end{remark}

It is simple to deduce from the proof of Theorem \ref{CriticalLevel} that Theorem \ref{LieGroupoidMorseLemma} and Lemmas \ref{NormalInvariance} and \ref{NormalInvariance2} imply that:
\begin{corollary}\label{QuotientLevelLemma}
	
Around the local model to the orbit $\mathcal{O}$ the deformation retraction $r_\tau$ is $G_x$-equivariant with respect to the action induced by the normal representation \eqref{eq:normalrep}.
\end{corollary}


\section{The stacky perspective}\label{S:7}

A stack can be thought of as a generalization of the notion of manifold which allows us to study higher symmetries and singular geometric features. The aim of this section is to adapt some of the Morse theory results obtained for Lie groupoids to the setting of differentiable stacks. It is worth mentioning that the fact that our notion of Morse Lie groupoid morphism is Morita invariant makes the passage clearer. Furthermore, as an interesting consequence, we will get Morse-like inequalities for certain separated differentiable stacks. Let $G\rightrightarrows M$ and $G'\rightrightarrows M'$ be two Lie groupoids. A \textbf{fraction} $\psi/\phi:G \to G'$ is by definition a pair of maps $(G\rightrightarrows M)\xleftarrow[\sim]{\it \phi}(H\rightrightarrows N)\xrightarrow[]{\it \psi} (G'\rightrightarrows M')$ where $\phi$ is a Morita map. Recall that $G\rightrightarrows M$ and $G'\rightrightarrows M'$ are said to be \textbf{Morita equivalent} if there exists a fraction $\psi/\phi:G\to G'$ where both $\phi$ and $\psi$ are Morita maps. It is well known that Morita equivalence is in fact an equivalence relation and that we can always assume that both $\phi$ and $\psi$ are Morita  fibrations \cite{dH,MM}. Accordingly, a \textbf{differentiable stack} is defined to be a Lie groupoid up to Morita equivalence and we write $[M/G]$ for the differentiable stack presented by the Lie groupoid $G\rightrightarrows M$. Two fractions $\psi_1/\phi_1$ and $\psi_2/\phi_2$ are \textbf{equivalent} if there are Morita maps $\alpha_1$ and $\alpha_2$ such that $\psi_1\circ \alpha_1$ is naturally isomorphic to $\psi_2\circ \alpha_2$ and $\phi_1\circ \alpha_1$ is naturally isomorphic to $\phi_2\circ \alpha_2$. This is also a equivalence relation. A class of fractions $[\psi/\phi]:[M/G]\to [M'/G']$ is named to be a \textbf{generalized} or \textbf{stacky map}. 

We are interested in studying Morse theory for stacky functions. More precisely, stacky maps from $[M/G]$ to $\mathbb{R}$ where we think of $\mathbb{R}$ as a differentiable stack presented by the unit groupoid $\mathbb{R}\rightrightarrows \mathbb{R}$. Since $\mathbb{R}$ has trivial isotropies every fraction $(G\rightrightarrows M)\xleftarrow[\sim]{\it \phi}(H\rightrightarrows N)\xrightarrow[]{\it \psi} (\mathbb{R}\rightrightarrows \mathbb{R})$ sends arrows over identities so that it descends to a usual Lie groupoid morphism. In other words, every stacky function $[M/G]\to \mathbb{R}$ is completely determined by a Lie groupoid morphism $F:G\to \mathbb{R}$ and, in turn, by a basic function $f:M\to \mathbb{R}$. This provides us with a simple way to establish a notion of Morse stacky map over $[M/G]$ as below. 

It is well known that Morita equivalences are completely characterized by the transversal information they carry with, see \cite[Thm. 3.4.1]{dH}. Such a characterization gives us some geometric intuition about the notion of stack $[M/G]$: it is an enhanced version of the orbit space $M/G$ endowed with certain smooth information which is encoded by the normal representations $G_x\curvearrowright \nu_x$. This perspective allows us to think of points in $[M/G]$ as elements lying inside $M/G$ (i.e. orbits), so that their tangent spaces can be modeled as follows. The \textbf{coarse tangent space} of the differentiable stack $[M/G]$ at $[x]=\mathcal{O}$ is by definition the coarse orbit space $\nu_x(\mathcal{O})/G_x\cong  \nu(\mathcal{O})/G_{\mathcal{O}}$ of the action groupoid determined by the normal representation on the orbit $\mathcal{O}$ through $x\in M$ (compare \cite{dho}). This will be denoted by $T_{[x]}[M/G]$. Therefore, if $F:[M/G]\to \mathbb{R}$ is a stacky function presented by a basic function $f:M\to \mathbb{R}$ then the \textbf{coarse differential} at $[x]\in [M/G]$ is the map $dF_{[x]}: T_{[x]}[M/G] \to \mathbb{R}$ defined by $dF_{[x]}([v]):=df(x)(v)$. This is well
defined in the sense that for $g\in G_x$ and $w\in T_xM$ such that $ds(g)(w)=v$ we have

$$dF_{[x]}(g\cdot [v])=df(x)(dt(g)(w))=df(x)(ds(g)(w))=dF_{[x]}([v]),$$
since $f$ is basic. In consequence, we say that $[x]\in [M/G]$ is a \textbf{critical point} of $F:[M/G]\to \mathbb{R}$ if $d_{[x]}F([v])=0$ for all $[v]\in T_{[x]}[M/G]$. This is clearly equivalent to requiring that $\mathcal{O}$ is a critical submanifold of $f:M\to \mathbb{R}$. We also define the \textbf{stacky Hessian} of $F:[M/G]\to \mathbb{R}$ at a critical point $[x]\in [M/G]$ as the pairing $\mathcal{H}_{[x]}(F):T_{[x]}[M/G]\times T_{[x]}[M/G]\to \mathbb{R}$ given by
$$
\mathcal{H}_{[x]}(F)([v_1],[v_2])=\overline{\mathcal{H}_{x}(f)}([v_1],[v_2]),
$$
where $\overline{\mathcal{H}_{x}(f)}$ denotes the restriction of the Hessian $\mathcal{H}_x(f)$ to the normal direction of $\nu_x(\mathcal{O})$. From Lemma \ref{NormalInvariance} it follows that the stacky Hessian $\mathcal{H}_{[x]}(F)$ is a well defined ``form'' over $T_{[x]}[M/G]$. Furthermore, it is simple to check that if $y\in \mathcal{O}$ then for $g\in G$ such that $t(g)=y$ it holds
\begin{equation}\label{StackyHessianRelation}
	\overline{\mathcal{H}_{y}(f)}= (\overline{ds(g)}\circ \overline{dt(g)}^{-1})^T\cdot \overline{\mathcal{H}_{x}(f)}\cdot (\overline{ds(g)}\circ \overline{dt(g)}^{-1}),
\end{equation}
so that $\mathcal{H}_x(f)$ is nondegenerate if and only if $\mathcal{H}_y(f)$ is nondegenerate. This justifies the following definition.
\begin{definition}
	A critical point $[x]\in [M/G]$ of a stacky function $F:[M/G]\to \mathbb{R}$ is said to be \textbf{nondegenerate} if $\mathcal{H}_{[x]}(F)$ nondegenerate. Accordingly, a \textbf{Morse stacky function} is a stacky function for which all of its critical points are nondegenerate.
\end{definition}

In other words, a stacky function $F:[M/G]\to \mathbb{R}$ is Morse if and only if it is presented by a Morse Lie groupoid morphism $G\to \mathbb{R}$.

Let us assume from now on that $[M/G]$ is separated (i.e. it is presented by a proper groupoid). Consider the quadratic form $Q_{[x]}:T_{[x]}[M/G]\to \mathbb{R}$ which is defined by the expression
$$Q_{[x]}(F)([v])=\frac{1}{2}\mathcal{H}_{[x]}(F)([v],[v]).$$ Using the stacky terminology introduced in \cite[s. 6]{dHF2} we can state a stacky version of the Morse lemma as follows.
\begin{proposition}[Stacky Morse lemma]\label{StackyMorseLemma}
	Let $[x]\in [M/G]$ be a nondegenerate critical point of a stacky function $F:[M/G]\to \mathbb{R}$. Then there are stacky neighborhoods $[V/\nu(G_{\mathcal{O}})_V]$ and $[U/G_U]$ of $[x]$ in $[\nu(\mathcal{O})/\nu(G_\mathcal{O})]$ and $[M/G]$, respectively, and a stacky isomorphism $\varphi: [V/\nu(G_{\mathcal{O}})_V]\to [U/G_U]$ fixing $[x]$ such that
	$$\varphi^\ast F=c+ Q_{[x]}(F).$$
\end{proposition}
\begin{proof}
	The Lie groupoid tubular neighborhood from Theorem \ref{LieGroupoidMorseLemma} which is constructed around the nondegenerate critical orbit $\mathcal{O}=[x]$ in such a way $f$ equals $c+Q_f$ near $\mathcal{O}$ induces the desired data. 
\end{proof}

Observe that after shrinking if necessary the stacky neighborhoods mentioned in the previous proposition we have that $[0]\in T_{[x]}[M/G]$ is the only critical point of $\varphi^\ast F$ inside $[V/\nu(G_{\mathcal{O}})_V]$ so that we get:
\begin{corollary}\label{IsolatedProperty}
	The nondegenerate critical points of a Morse stacky function are isolated in $M/G$.
\end{corollary}

Furthermore, just as in both the classical and the equivariant cases (compare \cite[Prop. 4,2]{W}), from the previous fact it follows that:

\begin{corollary}\label{FiniteStackyCritical}
If $M/G$ is compact then any stacky Morse function $F:[M/G]\to \mathbb{R}$ has a finite amount of nondegenerate critical points.
\end{corollary}

We define the \textbf{index data} of a nondegenerate critical point $[x]$ of a Morse stacky function $F:[M/G]\to \mathbb{R}$ as the pair $\lambda(F,[x]):=(\lambda(f,\mathcal{O}),G_x)$ where $\lambda(f,\mathcal{O})$ is the integer number $\mathrm{rk}(\nu_-(\mathcal{O}))$ for any basic function $f:M\to \mathbb{R}$ presenting $F$ and $G_x$ is the isotropy group at $x$. This is well defined in the sense that if $y\in \mathcal{O}$ then the normal representations $G_x\curvearrowright \nu_x(\mathcal{O})$ and $G_y\curvearrowright \nu_y(\mathcal{O})$ are isomorphic and the Identity \eqref{StackyHessianRelation} holds. Accordingly, based on Proposition \ref{NormalActionFeatures}, the \textbf{index} of $[x]$ will be defined as 
$$\dim \nu_-(\mathcal{O}_{x})_x/G_{x}=2\dim \nu_-(\mathcal{O}_{x})_x-\dim \nu_-(\mathcal{O}_{x})_x\times G_{x}=\lambda(f,\mathcal{O}_x)-\dim G_{x}.$$

Let $F:[M/G]\to \mathbb{R}$ be a stacky function and fix $a\in \mathbb{R}$. We define the \textbf{stacky level} of $F$ below $a$ as the set $[M/G]^a=\lbrace [x]\in [M/G]: F([x])\leq a\rbrace$. One can describe $[M/G]^a$ as a substack with boundary in the following sense. Suppose that $H\rr N$ is a Lie groupoid together with Morita fibrations $\phi:H\to G$ and $\psi:H\to G'$. Recall that for every $f'\in C^{\infty}(M')^{G'}$ we have that $\psi^*f'\in C^{\infty}(N)^H$. Also, there exists a unique $f\in C^{\infty}(M)^G$ with $\phi^*f=\psi^*f'$. This defines an isomorphism $C^{\infty}(M')^{G'}\to C^{\infty}(M)^G$ by sending $f'\mapsto f$, which actually preserves our Morse--Bott type condition along critical orbits (see Proposition \ref{MoritaInvariance}). Consider the level set $M^a=\lbrace x\in M: f(x)\leq a\rbrace$ with its boundary $\partial M^a=\lbrace x\in M: f(x)=a\rbrace$. It follows that $M^a$ is saturated so that we get a level subgroupoid $G^a$ of $G$. We can analogously define level subgroupoids $H^a$ and $G'^a$ of $H$ and $G'$, respectively, where $H^a$ is defined by either $\phi^*f$ or $\psi^*f'$. Note that we have well defined Morita fractions 

$$\textnormal{int}(G^a)\xleftarrow[\sim]{\it \phi}\textnormal{int}(H^a)\xrightarrow[\sim]{\it \psi} \textnormal{int}(G'^a)\quad\textnormal{and}\quad \partial G^a\xleftarrow[\sim]{\it \phi}\partial H^a\xrightarrow[\sim]{\it \psi} \partial G'^a,$$

so that we may think of $G^a\xleftarrow[\sim]{\it \phi}H^a\xrightarrow[\sim]{\it \psi} G'^a$ as a Morita fraction preserving boundaries. Hence, if $F:[M/G]\to \mathbb{R}$ is a stacky function presented by a basic function $f:M\to \mathbb{R}$ then the stacky level of $F$ below $a\in\mathbb{R}$ equals
$$[M^a/G^a]=[\textnormal{int}(M^a)/\textnormal{int}(G^a)]\cup [\partial M^a/\partial G^a].$$

Two Riemannian metrics $\eta_1$ and $\eta_2$ on $G\rightrightarrows M$ are said to be \textbf{equivalent} if they induce the same inner products on the normal vector spaces over the groupoid orbits \cite{dHF2}. More generally, we define a \textbf{Riemannian Morita map} (resp. \textbf{fibration}) $\phi: (H \rightrightarrows N) \to (G\rightrightarrows M)$ as a Morita map between Riemannian groupoids that induces isometries on the normal vector spaces to the groupoid orbits $\nu_z(\mathcal{O}^H)\to \nu_{\phi(z)}(\mathcal{O}^G)$ (resp. Riemannian submersion at the level of objects). By using this terminology we have that $\eta_1$ and $\eta_2$ are equivalent if and only if the identity $\textnormal{id}: (G \rightrightarrows M,\eta_1) \to (G\rightrightarrows M,\eta_2)$ is a Riemannian Morita map. Let us consider again the Morita fraction $G\xleftarrow[\sim]{\it \phi} H\xrightarrow[\sim]{\it \psi} G'$. From \cite[Prop. 6.3.1]{dHF2} we know that if $\eta^G$ is a Riemannian metric on $G$ then there exists a Riemannian metric $\eta^H$ on $H$ that makes the fibration $\phi:H\to G$ Riemannian. We can slightly modify $\eta^H$ by a cotangent averaging procedure so that we get another Riemannian metric $\tilde{\eta}^H$ on $H$ which descends to $G'$ defining a Riemannian metric $\eta^{G'}$ making of the fibration $\psi:H\to G'$ Riemannian. It turns out that these pullback and pushforward constructions are well-defined and mutually inverse modulo equivalence of metrics. This is because $\eta^H$ and $\tilde{\eta}^H$ turn out to be equivalent, see the proof of Theorem 6.3.3 in \cite{dHF2}. In this case we refer to $(G,\eta^G)$ and $(G',\eta^{G'})$ as being \textbf{Morita equivalent Riemannian groupoids}. It suggests a definition for Riemannian metrics over differentiable stacks. Namely, a \textbf{stacky metric} on the orbit stack $[M/G]$ presented by a Lie groupoid $G\rightrightarrows M$ is defined to be an equivalence class $[\eta]$ of a Riemannian metric $\eta$ on $G$. For further details the reader is recommended to visit \cite{dHF2}.

The \textbf{tangent stack} of $[M/G]$ is by definition the differentiable stack $T[M/G]:=[TM/TG]$ that is presented by the tangent groupoid $TG\rr TM$, see \cite{H2}. The multiplicative vector field defined by the gradient of $F:G\to\mathbb{R}$ with respect to the groupoid metric $\eta$ defines, by means of the Dictionary Lemmas from \cite[s. 2.6]{BX}, a stacky vector field on $[M/G]$ in the sense of \cite[Def. 4.14]{H2}. By using the identification of the coarse tangent space $T_{[x]}[M/G]$ with the orbit space $\nu_x(\mathcal{O})/G_x\cong  \nu(\mathcal{O})/G_{\mathcal{O}}$ we may think of this stacky vector field as a stacky map $\nabla F: [M/G]\to T[M/G]$ given by $[x]\mapsto [\nabla f(x)]$, which is clearly well defined. We shall refer to it as the \textbf{stacky gradient vector field} of $F:[M/G]\to \mathbb{R}$ with respect to $[\eta]$. Observe that $\nabla F$ satisfies
$$
[\eta](\nabla F[x],[v])  =  [\eta]([\nabla f(x)],[v])=\eta(\nabla f(x),v)=df(x)(v)=dF_{[x]}([v]).
$$

\begin{remark}
From the point of view of Lie groupoids, Theorem 4.15 in \cite{H2} says that the category of multiplicative vector fields on $G$ depends, up to equivalence, only on the Morita equivalence class of $G$. If we think of the differentiable stack $[M/G]$ as the equivalence class of $G$ in the enlarged 2-category of Lie groupoids, principal bi-bundle and isomorphisms then the category of stacky vector field on $[M/G]$ is equivalent to category of multiplicative vector fields on $G$.
\end{remark}
Let us consider the underlying continuous map $\overline{f}:M/G\to \mathbb{R}$.

\begin{proposition}\label{StackyNonCriticalLevel}
	Let $[M/G]$ be a separated stack and $[a,b]\subset \mathbb{R}$ be a real interval such that $\overline{f}^{-1}[a,b]$ is compact in $M/G$. If $\overline{f}^{-1}[a,b]$ has no critical points of $F$ then $[M/G]^a$ and $[M/G]^b$ are stacky isomorphic. Furthermore, $[M/G]^a$ is a stacky deformation retraction of $[M/G]^b$.
\end{proposition}
\begin{proof}
We shall follow the proof of \cite[Thm. 7.5]{H} closely and apply some of the results proved in \cite{H2} for stacky vector fields and flows. On the one hand, a simple computation shows that the smooth function $\Vert \nabla f \Vert^2:M\to \mathbb{R}$ is basic since $f$ is so and the gradient vector field $\nabla F$ on $G$ is horizontal with respect to both Riemannian submersions $s$ and $t$. On the other hand, since $[M/G]$ is a separated stack we may construct ``stacky'' partitions of unity for $[M/G]$ (see \cite[Def. 2.13 \& Prop. 2.14]{H2}). In consequence, from \cite[Lem. 3.11 \& Lem. 3.12]{H} we may find a stacky function $\rho:[M/G]\to \mathbb{R}$ with compact support in $M/G$ and with $\overline{\rho}=1/\overline{\Vert \nabla f \Vert^2}$ in $\overline{f}^{-1}[a,b]$. Therefore, we can form the stacky vector field $\tilde{X}=\rho \nabla F$ on $[M/G]$, which has compact support in $M/G$, and then take its stacky flow $\Phi: [M/G]\times \mathbb{R}\to [M/G]$ by using \cite[Thm. 5.12]{H2}. 

The vector field $\tilde{X}$ can, by \cite[Thm. 4.15]{H2} and the proof of \cite[Prop. 4.17]{H2}, be presented by a multiplicative vector field $X=(X_1,X_0)$ on $G\rr M$ that is not equal to the zero section only on a subgroupoid of $G$ whose image in $M/G$ has compact closure. That is to say, the vector fields $X_1$ and $X_0$ are compactly-supported in $F^{-1}[a,b]\subset G$ and $f^{-1}[a,b]\subset M$, respectively. Thus, motivated by the proof of Proposition \ref{NonCriticalLevel}, we may clearly identify $X_1=\tilde{\rho} \nabla F$ and $X_0=\rho \nabla f$ where $\rho:M\to \mathbb{R}$ is the compactly-supported smooth function on $M$ with $\rho=1/\Vert \nabla f \Vert^2$ inside $f^{-1}[a,b]$. This is clearly basic and  $\tilde{\rho}:G\to \mathbb{R}$ is given by either $s^\ast \rho$ or $t^\ast \rho$. Observe that $\tilde{\rho}$ is also compactly-supported in $F^{-1}[a,b]$.

Hence, the arguments above allow us to think of the stacky vector field $\tilde{X}:[M/G]\to T[M/G]$ as being determined by the assignment $[x]\mapsto [X_0(x)]$. Also, by using the Dictionary Lemmas from \cite[s. 2.6]{BX} as in the proof of \cite[Prop. 6.2]{H2}, the stacky flow $\Phi: [M/G]\times \mathbb{R}\to [M/G]$ of $\tilde{X}$ is determined by $\Phi([x],\tau)=[\varphi^0_\tau(x)]$ where $(\varphi^1_\tau,\varphi^0_\tau)$, for all $\tau\in \mathbb{R}$, is the $1$-parameter family of Lie groupoid automorphisms on $G$ determined by the flow of the multiplicative vector field $X$. In particular, by using these identifications it holds that if $\Phi([x],\tau)\in \overline{f}^{-1}[a,b]$ then $\tilde{X}\cdot F(\Phi([x],\tau))=1$ so that the result follows either by proceeding  without any chance as in the proof of \cite[Thm. 3.1]{Mi} or by using Proposition \ref{NonCriticalLevel} directly. 
\end{proof}

Let us pick an orbit $\mathcal{O}$ of $G$ and denote by $\mathcal{O}^H=\phi^{-1}(\mathcal{O})$ and $\mathcal{O}'=\psi(\mathcal{O}^H)$ the respective orbits of $H$ and $G'$. From \cite[Prop. 6.4.1]{dHF2} it follows that there is an induced Morita fraction
$$\nu(G_{\mathcal{O}})\xleftarrow[\sim]{\it \overline{d\phi}} \nu(H_{\mathcal{O}^H})\xrightarrow[\sim]{\it \overline{d\psi}} \nu(G'_{\mathcal{O}'}).$$

Observe that if in addition $\mathcal{O}$ is a nondegenerate critical orbit then the previous fraction induces a Morita fraction between the negative normal groupoids $\nu_{-}(G_{\mathcal{O}})\xleftarrow[\sim]{\it \overline{d\phi}} \nu_{-}(H_{\mathcal{O}^H})\xrightarrow[\sim]{\it \overline{d\psi}} \nu_{-}(G'_{\mathcal{O}'})$ (see Lemma \ref{NegativeNormalLG}). More importantly, if $G$ and $G'$ are the Morita equivalent Riemannian groupoids as we described above then there are Morita fractions between the unit and sphere groupoids:
$$D_{-}(G_{\mathcal{O}})\xleftarrow[\sim]{\it \overline{d\phi}} D_{-}(H_{\mathcal{O}^H})\xrightarrow[\sim]{\it \overline{d\psi}} D_{-}(G'_{\mathcal{O}'})\quad\textnormal{and}\quad \partial D_{-}(G_{\mathcal{O}})\xleftarrow[\sim]{\it \overline{d\phi}} \partial D_{-}(H_{\mathcal{O}^H})\xrightarrow[\sim]{\it \overline{d\psi}} \partial D_{-}(G'_{\mathcal{O}'}).$$ 

The stack $e^\lambda_{[x]}:=[D_{-}(\mathcal{O})/D_{-}(G_{\mathcal{O}})]$ will be called \textbf{stacky $\lambda$-cell} at $[x]$ and its boundary $\partial e^\lambda_{[x]}:=[\partial D_{-}(\mathcal{O})/\partial D_{-}(G_{\mathcal{O}})]$ will be called \textbf{stacky $(\lambda-1)$-sphere} at $[x]$. 

\begin{remark}
We may think of the attaching space $[M/G]^a\cup_{\partial e^\lambda_{[x]}} e^\lambda_{[x]}$ as being a topological stack presented by the attaching groupoid $G^{a}\cup_{B} D_-(G_{\mathcal{O}})$. This is because we can consider the topological Morita fraction between the attaching groupoids	$$G^a\cup_B D_{-}(G_{\mathcal{O}})\xleftarrow[\sim]{\it \phi\cup \overline{d\phi}} H^a\cup_B D_{-}(H_{\mathcal{O}^H})\xrightarrow[\sim]{\it \psi\cup \overline{d\psi}} G'^a\cup_B D_{-}(G'_{\mathcal{O}'}).$$
\end{remark}

Whit this notation we have:

\begin{proposition}\label{StakyCriticalLevel}
	Let $[M/G]$ be a separated stack and $[a,b]\subset \mathbb{R}$ be a real interval such that $\overline{f}^{-1}[a,b]$ is compact in $M/G$. If $\overline{f}^{-1}[a,b]$ contains no critical points besides $[x]$ of index data $\lambda(F,[x])$ then $[M/G]^b$ is stacky homotopy equivalent to $[M/G]^a\cup_{\partial e^\lambda_{[x]}} e^\lambda_{[x]}$.
\end{proposition}
\begin{proof}
We start by noting that by Lemma \ref{NormalInvariance} and Proposition \ref{NormalActionFeatures} we may split the coarse tangent space $T_{[x]}[M/G]=T_{[x]}^-[M/G]\oplus T_{[x]}^+[M/G]$ where $T_{[x]}^-[M/G]=\nu_-(\mathcal{O})_x/G_x$ and $T_{[x]}^+[M/G]=\nu_+(\mathcal{O})_x/G_x$. In particular, by Lemma \ref{NormalInvariance2} it follows that the two mutually complementary projections $\xi_0:\nu(\mathcal{O})\to \nu_-(\mathcal{O})$ and $\eta_0:\nu(\mathcal{O})\to \nu_+(\mathcal{O})$ descend to define coordinates over $T_{[x]}^-[M/G]$ and $T_{[x]}^+[M/G]$, respectively, and a norm on $T_{[x]}[M/G]$ as:
$$\| [v]\|^2=-Q_{[x]}(F)(\xi_0([v]))+Q_{[x]}(F)(\eta_0([v])),$$
which agrees with the norm induced by $[\eta]$. Now, by Proposition \ref{StackyMorseLemma} we have that there are stacky neighborhoods $[V/\nu(G_{\mathcal{O}})_V]$ and $[U/G_U]$ of $[x]$ in $[\nu(\mathcal{O})/\nu(G_\mathcal{O})]$ and $[M/G]$, respectively, and a stacky isomorphism $\varphi: [V/\nu(G_{\mathcal{O}})_V]\to [U/G_U]$ fixing $[x]$ such that $\varphi^\ast F=c+ Q_{[x]}(F)$. Under the identification $[V/\nu(G_{\mathcal{O}})_V]\cong [U/G_U]$ we may think of $F$ on $[U/G_U]$ as given respectively by
$$F([v])=c+Q_{[x]}(F)([v])=c-\| \xi_0([v])\|^2+\| \eta_0([v])\|^2.$$

Let us consider the stacky function $F_1:[M/G]\to \mathbb{R}$ determined by the basic function $f_1:M\to \mathbb{R}$ constructed in Theorem \ref{CriticalLevel}. Such a stacky function can be described in the way we just did above with $F$. Therefore, the result follows by following the same steps in \cite[Thm. 7.6]{H} after using Proposition \ref{StackyNonCriticalLevel}, Theorem \ref{CriticalLevel} together the Dictionary Lemmas from \cite[s. 2.6]{BX}, and Corollary \ref{QuotientLevelLemma}.
\end{proof}

Recall that as consequence of Proposition \ref{NormalActionFeatures} we know that the orbit spaces associated to $e^\lambda_{[x]}$ and $\partial e^\lambda_{[x]}$ are respectively given by $D_-(\mathcal{O})_x/G_x$ and $\partial D_-(\mathcal{O})_x/G_x$. Thus:

\begin{corollary}\label{OrbitDecomposition}
The orbit space $(M/G)^{b}$ has the homotopy type of $(M/G)^{a}$ with a copy of $D_-(\mathcal{O})_x/G_x$ attached along $\partial D_-(\mathcal{O})_x/G_x$.
\end{corollary}

\subsection{Morse inequalities for the orbit space} 

By following similar arguments as those used by Hepworth in \cite{H} it is possible to prove Morse-like inequalities for certain separated stacks. We shall follow \cite[Sec. 7.3]{H} closely. Suppose that $[M/G]$ is a separated differentiable stack with $M/G$ compact and let $F:[M/G]\to \mathbb{R}$ be a stacky Morse function presented by a basic function $f:M\to \mathbb{R}$. From Corollary \ref{IsolatedProperty} it follows that the critical points of $F$ are isolated so that we may take a finite sequence $q_0<q_1<q_2<\cdots <q_r\in \mathbb{R}$ such that each interval $(q_j,q_{j+1})$ contains only one critical value of $F$ and such that all critical values lie inside such intervals. Let us denote by $[x^j_1],\cdots,[x^j_{k_j}]$ the critical points inside $\overline{f}^{-1}(q_j,q_{j+1})$. Therefore, by an inductive process it follows from Corollary \ref{OrbitDecomposition} that:
\begin{corollary}\label{OrbitCellDecomposition}
There is a decomposition $M/G=\bigcup_{j=1}^r(M/G)^{q_j}$ where each $(M/G)^{q_{j+1}}$ has the homotopy type of $(M/G)^{q_j}$ with copies of $D_-(\mathcal{O}_{x^j_{l}})_{x^j_{l}}/G_{x^j_{l}}$ attached along $\partial D_-(\mathcal{O}_{x^j_{l}})_{x^j_{l}}/G_{x^j_{l}}$ for $l=1,\cdots, k_j$.
\end{corollary}
Recall that the \textbf{Betti numbers} of $M/G$ are by definition $b_j=\dim H_j(M/G, \mathbb{R})$ and the \textbf{Poincaré polynomial} is given by
$$\mathcal{P}_\tau(M/G)=\sum b_j\tau^j.$$
\begin{definition}
A critical point $[x]$ of a stacky Morse function $F:[M/G]\to \mathbb{R}$ is said to be \textbf{orientable} if the action of $G_x$ on $\nu_-(\mathcal{O})_x$ from Proposition \ref{NormalActionFeatures} is orientation-preserving. The \textbf{Morse polynomial} of $F$ is defined as
$$\mathcal{M}_\tau(M/G)=\sum_{[x]\in \textnormal{Crit}(f)_{o-p}}\tau^{\dim \nu_-(\mathcal{O}_{x})_x/G_{x}},$$
where $\textnormal{Crit}(f)_{o-p}$ denotes the set of orientable critical points of $F$.
\end{definition}
We are now in conditions to state that:
\begin{theorem}\label{MorseInequalities}
There is a polynomial $\mathcal{R}_\tau(M/G)$ with non-negative integer coefficients such that
$$\mathcal{M}_\tau(M/G)=\mathcal{P}_\tau(M/G)+(1+\tau)\mathcal{R}_\tau(M/G).$$
In particular, if $\mathcal{M}_\tau(M/G)$ has no consecutive powers of $\tau$ then $\mathcal{M}_\tau(M/G)=\mathcal{P}_\tau(M/G)$.
\end{theorem}
\begin{proof}
The proof of this result is similar to \cite[Thm. 7.11]{H}. We shall sketch its main ideas here for the sake of completeness. It is well known that the function
\begin{equation}\label{eq:MorseInequality1}
S_j(X,Y)=\dim H_j(X,Y; \mathbb{R})-\dim H_{j-1}(X,Y; \mathbb{R})+\cdots\pm \dim H_0(X,Y; \mathbb{R}),
\end{equation}
is subadditive in the sense that if $Z\subset Y\subset X$ then we get $S_j(X,Z)\leq S_j(X,Y)+S_j(Y,Z)$. Thus, from Corollary \ref{OrbitCellDecomposition} it follows that 
$$S_j(M/G,\emptyset)=S_j((M/G)^{q_r},(M/G)^{q_0})\leq \sum_{v} S_j((M/G)^{q_v},(M/G)^{q_{v-1}}).$$
Using again Corollary \ref{OrbitCellDecomposition} and the excision theorem for relative homology we obtain
\begin{eqnarray*}
H_j((M/G)^{q_v},(M/G)^{q_{v-1}};\mathbb{R}) &=&H_j\left((M/G)^{q_{v-1}}\cup \bigcup_{l} D_-(\mathcal{O}_{x^v_{l}})_{x^v_{l}}/G_{x^v_{l}},(M/G)^{q_{v-1}};\mathbb{R} \right)\\
&=& \bigoplus_{l}H_j\left(D_-(\mathcal{O}_{x^v_{l}})_{x^v_{l}}/G_{x^v_{l}},\partial D_-(\mathcal{O}_{x^v_{l}})_{x^v_{l}}/G_{x^v_{l}};\mathbb{R} \right)
\end{eqnarray*}

where the homology $H_j\left(D_-(\mathcal{O}_{x^v_{l}})_{x^v_{l}}/G_{x^v_{l}},\partial D_-(\mathcal{O}_{x^v_{l}})_{x^v_{l}}/G_{x^v_{l}};\mathbb{R} \right)$ equals $\mathbb{R}$ if $[x^v_{l}]$ is orientable and $j= \dim \nu_-(\mathcal{O}_{x^v_{l}})_{x^v_{l}}/G_{x^v_{l}}$. Otherwise, it equals $0$. Hence, the sum \eqref{eq:MorseInequality1} becomes
\begin{equation}\label{eq:MorseInequality2}
S_j(M/G,\emptyset)\leq A_j-A_{j-1}+\cdots\pm A_0,
\end{equation}
where $A_l$ denotes the number of orientable critical points $[x]$ of $F$ with $l= \dim \nu_-(\mathcal{O}_{x})_x/G_{x}$. However, as consequence of \cite[Lem. 3.43]{BH3} it holds that Inequality \eqref{eq:MorseInequality2} is actually equivalent to the claim of the proposition so that the result follows.
\end{proof}
It is clear that the polynomial equality in the previous proposition only depends on the Morita equivalence class of $G$. Therefore, we have actually obtained Morse-like inequalities for the differentiable stack $[M/G]$.

\subsection{Stratified Morse theory}\label{S:7.2}
In this short subsection we establish a bridge between our approach to study Morse theory over the orbit space of a proper Lie groupoid and the approach provided by the well known stratified Morse theory in the sense of Goresky and MacPherson in \cite{GM}. Let $N$ be a smooth manifold and let $X\subseteq N$ be a subset with a Whitney stratification $\mathcal{S}$, visit \cite[s. 4.1]{CM}. For $\mathcal{S}_j,\mathcal{S}_k\in \mathcal{S}$ we write $\mathcal{S}_j\leq \mathcal{S}_k$ to denote the usual ordering given by $\mathcal{S}_j\subseteq \overline{\mathcal{S}_k}$. For each $x\in X$ we denote by $\mathcal{S}_x$ the stratum of $\mathcal{S}$ containing $x$. Fix a smooth function $\tilde{f}:N\to \mathbb{R}$ and set $f:=\tilde{f}|_{X}$. Such an $f$ is simply referred to as a smooth function on $X$. The \textbf{stratified critical point set} of $f$ is by definition $\textnormal{Crit}_\mathcal{S}(f):=\bigcup_j\textnormal{Crit}_{\mathcal{S}_j}(f)$ where $\textnormal{Crit}_{\mathcal{S}_j}(f)$ stands to denote the usual critical point set of $f$ on $\mathcal{S}_j$. For $\mathcal{S}_j$ we define the \textbf{conormal space to $\mathcal{S}_j$ in $N$} as the vector bundle $T^\ast_{\mathcal{S}_j}N$ over $\mathcal{S}_j$ whose fiber at $x\in \mathcal{S}_j$ consist of the covector in $T^\ast _xN$ which vanish on the tangent space to $\mathcal{S}_j$ at $x$. Note that $x\in \textnormal{Crit}_\mathcal{S}(f)$ if and only if $d\tilde{f}(x)\in T^\ast_{\mathcal{S}_x}N$. The set of \textbf{degenerate conormal covectors} to a stratum $\mathcal{S}_j$ is defined to be 
$$D^\ast_{\mathcal{S}_j}N:=T^\ast_{\mathcal{S}_j}N\cap \bigcup_{\mathcal{S}_j< \mathcal{S}_k} \overline{T^\ast_{\mathcal{S}_k}N}=\left( \bigcup_{\mathcal{S}_j< \mathcal{S}_k} \overline{T^\ast_{\mathcal{S}_k}N}\right)|_{\mathcal{S}_j}.$$
That is, the fiber $(D^\ast_{\mathcal{S}_x}N)_x$ consists of limits at $x$ of conormal covectors to larger strata or, equivalently, conormal covectors to $\mathcal{S}_x$ at $x$ which vanish on limiting tangent spaces from larger strata.
\begin{definition}\cite{GM}\label{StratifiedMorse}
A point $x\in X$ is a said to be a \textbf{nondegenerate critical point} of $f$ if and only if $x$ is a nondegenerate critical point of $f|_{\mathcal{S}_x}$ and $d\tilde{f}(x)\notin D^\ast_{\mathcal{S}_x}N$. Accordingly, the function $f$ is called a \textbf{stratified Morse function} if and only if all of the critical points of $f$ are nondegenerate.
\end{definition}

\begin{remark}
It is worth mentioning that the classical definition of stratified Morse function also assumes $f$ to be proper and to have distinct critical values. However, because of our purposes in this subsection such additional conditions will not be required.
\end{remark}

Let $G\rr M$ be a proper Lie groupoid. It is well known that the orbit space $M/G$ of $G$ has several natural Whitney stratifications, compare \cite{CM,PTW}. In particular, in \cite{CM} were given canonical stratifications for both $M$ and $M/G$ by Morita types which generalize the canonical stratifications induced by a proper Lie group action. Namely, the \textbf{Morita type equivalence} is the equivalence relation on $M$ given by $x\sim_{\mathcal{M}} y$ if and only if the normal representations $G_x\curvearrowright \nu_x(\mathcal{O}_x)$ and $G_y\curvearrowright \nu_y(\mathcal{O}_y)$ are isomorphic. The \textbf{partition by Morita types}, denoted by $\mathcal{P}_\mathcal{M}(M)$, is defined to be the resulting
partition. Each member of $\mathcal{P}_\mathcal{M}(M)$ is called a \textbf{Morita type}. If $[G_x, \nu_x(\mathcal{O}_x)]=\alpha$ denotes the isomorphism class of the normal representation $G_x\curvearrowright \nu_x(\mathcal{O}_x)$ then the element in $\mathcal{P}_\mathcal{M}(M)$ corresponding to $\alpha$ will be denoted by $M_{(\alpha)}=\lbrace x\in M: [G_x, \nu_x(\mathcal{O}_x)]=\alpha\rbrace$. We also denote by $M_{(x)}$ the Morita type of a point $x\in M$. It is simple to see that points in the same orbit belong to the same Morita type and hence we also obtain a partition by Morita types on the orbit space, $\mathcal{P}_\mathcal{M}(M/G)$. The orbit projection map $\pi: M\to M/G$ takes Morita types $M_{(\alpha)}$ in $M$ onto Morita types $X_{(\alpha)}=\pi(M_{(\alpha)})$ in $M/G$. Observe that by its very
definition, the Morita type of $\mathcal{O}\in M/G$ only depends on Morita invariant information. The \textbf{canonical Whitney stratification} on $M$, denoted by $\mathcal{S}_{G}(M)$, is the partition on $M$ obtained by passing to connected components of $\mathcal{P}_\mathcal{M}(M)$. The canonical
Whitney stratification on the orbit space $M/G$, denoted by $\mathcal{S}(M/G)$, is the partition on $M/G$ obtained
by passing to connected components of  $\mathcal{P}_\mathcal{M}(M/G)$, see \cite[s. 4.6]{CM}. As important features of these canonical Whitney stratifications we have that if our Lie groupoid has only one Morita type then the orbit space $M/G$ is a smooth manifold and the canonical projection $\pi : M \to M/G$ becomes a submersion whose fibers are the orbits. More importantly, let $G_{(\alpha)}\rr M_{(\alpha)}$ denote the groupoid $G_{(\alpha)}=s^{-1}(M_{(\alpha)})$ over a fixed Morita type $M_{(\alpha)}$. It follows that $G_{(\alpha)}\rr M_{(\alpha)}$ is a Lie groupoid, $X_{(\alpha)}$ is a smooth manifold and the  the canonical projection $\pi : M_{(\alpha)}  \to X_{(\alpha)}$ is a submersion. Furthermore, any Morita equivalence between two proper Lie groupoids induces an isomorphism of differentiable stratified spaces between their orbit spaces, visit \cite{CM}.

Suppose either that the orbit space $M/G$ is compact or has a finite number of Morita types. In these cases it follows that $M/G$ can be embedded into an affine space $\mathbb{R}^d$, compare \cite{FaSea} and \cite{CM}, respectively. We are now in conditions to establish a simple correspondence between our approach to study Morse theory over the orbit space $M/G$ and the approach provided by stratified Morse theory. Indeed:

\begin{proposition}\label{StratifiedComparison}
Let $G\rr M$ be a proper Lie groupoid having either compact orbit space $M/G$ or else a finite number of Morita types. There is a natural correspondence between Morse Lie groupoid morphisms on $G$ and stratified Morse functions on $M/G$ as those defined in Definition \ref{StratifiedMorse}.
\end{proposition} 
\begin{proof}
Let $F:G\to \mathbb{R}$ be a Morse Lie groupoid morphism induced by a basic function $f:M\to \mathbb{R}$ and let $\overline{f}:M/G\to \mathbb{R}$ denote the corresponding function over the orbit space $M/G$. By using the embedding $M/G\hookrightarrow \mathbb{R}^{d}$ we can think of $\overline{f}$ as a smooth function on $M/G$ (see \cite[p. 827]{CM}), so that we may use standard partitions of unity on $\mathbb{R}^d$ to construct a smooth function $\tilde{f}:\mathbb{R}^d\to \mathbb{R}$ such that $\tilde{f}|_{M/G}=\overline{f}$. Every critical point $[x]$ of $\overline{f}$ in $M/G$ belongs to some stratum $X_{(\alpha)}$ and we have the relation $\overline{f}|_{X_{(\alpha)}}=f|_{M_{(\alpha)}}\circ \pi$. This automatically guarantees that $[x]$ is a stratified nondegenerate critical point of $\overline{f}$ since $f$ is Morse--Bott with nondegenerate critical orbits, $\pi|_{M_{(\alpha)}}$ is a surjective submersion, and the normal Hessian of $f$ is invariant under the normal representation. The condition $d\tilde{f}(x)\notin D^\ast_{X_{(\alpha)}}\mathbb{R}^d$ follows from the fact that $f:M\to \mathbb{R}$ is basic which implies it preserves the stratification $\mathcal{S}(M/G)$. Conversely, if $\overline{f}:M/G\to \mathbb{R}$ is a stratified Morse function induced by $\tilde{f}:\mathbb{R}^d\to \mathbb{R}$ then by taking average with respect to some proper Haar measure system over $G$ as applied in \cite[s. 3]{CM} we can construct a basic smooth function $f:M\to \mathbb{R}$ such that $\overline{f}=f\circ \pi$. So, the results follow by arguing similarly as above.
\end{proof}

The following observations are some of the advantages that our approach to study Morse theory over $M/G$ have.
\begin{remark}
Firstly, our focus aims at extending classical Morse theory to the context of differentiable stacks, so that its study passing through Lie groupoids is more suitable for our purposes. Also, the latter provides a cleaner way to extract several fundamental results of Morse theory without assuming properness, compactness or finiteness on the Morita types equivalences, which are requirements that seems to be necessary in order to guarantee that the orbit space admits some kind of embedding into an affine space. Evidences of this fact are mentioned in Remark \ref{MorseLemmaOtherWay}, Subsection \ref{S:4.2}, Proposition \ref{NormalActionFeatures}. Even though sometimes we further assume properness for the Lie groupoids we are working with, results as those provided in Proposition \ref{NonCriticalLevel} and Theorem \ref{CriticalLevel} in the groupoid category are clearly natural and easier to obtain by following the procedure we propose. Secondly, as we will see in next sections, our approach also provides a way to recover some algebraic/topological constructions underlying the Lie groupoid structure in a natural way. For instance, the total cohomology of the Bott-Shulman-Stasheff double complex as well as its $2$-equivariant version can be obtained by using our techniques. The latter construction can be applied to compute the equivariant cohomology of certain toric symplectic stacks, see Example \ref{EToricStack}. Thirdly, our approach seems to be applicable to satisfactorily extend other topological and geometrical constructions derived from classical Morse theory to more general contexts, e.g. Novikov type inequalities for certain separated differentiable stacks as done in \cite{Va}. 
\end{remark}

In \cite{MT} it was shown that proper Lie groupoids are real analytic. Hence, after adapting the results obtained in \cite{CM} to the real analytic world we may get as an application of Proposition \ref{StratifiedComparison} together with the Pignoni's density result proved in \cite{Pi} that:
\begin{corollary}
Under the same hypothesis of Proposition \ref{StratifiedComparison} it follows that the set of smooth functions $\mathbb{R}^d\to \mathbb{R}$ which restrict to stratified Morse functions on $M/G$ and induce Morse Lie groupoid morphisms on $G$ form an open and dense subset with respect to the strong topology on $C^\infty(\mathbb{R}^d)$.
\end{corollary}

 
 \section{Morse--Smale dynamics}\label{S:8}

 Let us now start by adapting some notions of the Morse--Smale dynamics to the Lie groupoid setting. Our goal here is to define the stable and unstable Lie groupoids of a Morse Lie groupoid morphism as well as to study some of their elementary properties. For more details regarding the classical notions see \cite{AB,BH2,Ni} and the references therein. Let $F:(G\rightrightarrows M)\to (\mathbb{R}\rightrightarrows \mathbb{R})$ be a Morse Lie groupoid morphism induced by a smooth basic function $f:M\to \mathbb{R}$. Throughout this section, apart from assuming that $G\rightrightarrows M$ is proper, we will also assume that one of either $G$ or $M$ is compact. This automatically implies that the other one is compact as well. Under these assumptions it is clear that the multiplicative gradient vector field $\nabla F$ is given by a pair of complete vector fields.

 \begin{remark}\label{CompactRelaxed}
 	 
 As consequence of Theorems 4.15 and 5.12 in \cite{H2} together with the Dictionary Lemmas from \cite[s. 2.6]{BX}, it follows that the compactness assumption for either $M$ or $G$ may be relaxed by requiring only that $M/G$ is compact. This is enough to have globally defined gradient flows since $\nabla F$ is a multiplicative vector field. It is worth mentioning that from Lemma 4.4 in \cite{CMS} it follows that assuming the completeness of $\nabla f$ would be also sufficient to have that $\nabla F$ is complete, again because it is multiplicative.
 \end{remark}
 
  Let $\widetilde{\Phi_{\tau}}: G\to G$ and $\Phi_{\tau}:M\to M$ denote the flow of the vector fields $-\nabla(s^\ast f)$ and $-\nabla f$, respectively (descending flows). If $x\in \textnormal{Crit}(f)$, then the \textbf{stable manifold} $W^{s}(x)$ and the \textbf{unstable manifold} $W^{u}(x)$ of $f$ at $x$ are respectively defined as
 
 \begin{equation}\label{eq:stableunstable}
 W^{s}(x)=\lbrace y\in M: \lim_{\tau\to \infty}\Phi_{\tau}(y)=x\rbrace\qquad\textnormal{and}\qquad W^{u}(x)=\lbrace y\in M: \lim_{\tau\to -\infty}\Phi_{\tau}(y)=x\rbrace.
 \end{equation}

 The stable and unstable manifolds of a critical arrow $g\in\textnormal{Crit}(F)$ are similarly defined by using the descending flow $\widetilde{\Phi_{\tau}}$. Let $S_\lambda \subseteq \textnormal{Crit}(f)$ denote the set formed by the orbits in $M$ with same index $\lambda$. We may assume that $S_{\lambda}$ consists of orbits with the same dimension, otherwise we split $S_{\lambda}$ into components consisting of orbits with the same dimension. Hence, we may assume that $S_{\lambda}$ is a manifold which, being saturated, yields a well-defined Lie groupoid $G_{S_\lambda}\rightrightarrows S_\lambda$ defined by the restriction of $G\rr M$ to $S_{\lambda}$. It is important to observe that $S_\lambda$ is a non-degenerate critical submanifold for $f$ of index $\lambda$ and as consequence of what we did before we may conclude that $G_{S_\lambda}$ is also a non-degenerate critical submanifold for $F$ with same index $\lambda$. Thus, as we have that every point in $S_\lambda$ is a critical point for $f$ we define the stable and unstable submanifolds of $S_\lambda$ as the respective disjoint unions
 $$W^{s}(S_\lambda)=\bigcup_{x\in S_\lambda}W^{s}(x)\qquad\textnormal{and}\qquad W^{u}(S_\lambda)=\bigcup_{x\in S_\lambda}W^{u}(x).$$
 The stable and unstable submanifolds $W^{s}(G_{S_\lambda})$ are $W^{u}(G_{S_\lambda})$ are defined in the same way since every arrow in $G_{S_\lambda}$ is a critical arrow for $F$. As a consequence of the multiplicativity of $-\nabla F=(-\nabla(s^\ast f),-\nabla f)$ we obtain the following expected result.
 \begin{lemma}\label{Stable/UnstableGroupoids}
 	The Lie groupoid structure of $G\rightrightarrows M$ can be naturally restricted  to define two Lie groupoids $W^{s}(G_{S_\lambda})\rightrightarrows W^{s}(S_\lambda)$ and $W^{u}(G_{S_\lambda})\rightrightarrows W^{u}(S_\lambda)$.
 \end{lemma}
 
 \begin{proof}

 	We will only show why the Lie groupoid structure may be well restricted for defining $W^{s}(G_{S_\lambda})\rightrightarrows W^{s}(S_\lambda)$ since the other case follows analogously. This is carried out in the following steps. First, if $h\in W^{s}(G_{S_\lambda})$, then $h\in W^{s}(g)$ for some critical arrow $g\in G_{S_\lambda}$. Thus, the identity $s\circ \widetilde{\Phi_{\tau}}=\Phi_{\tau}\circ s$ implies that
 	$$\lim_{\tau\to \infty}\Phi_{\tau}(s(h))=\lim_{\tau\to \infty}s(\widetilde{\Phi_{\tau}}(h))=s\left(\lim_{\tau\to \infty}\widetilde{\Phi_{\tau}}(h) \right)=s(g).$$
 	This means that $s(h)\in W^{s}(s(g))\subset W^{s}(S_\lambda)$ and therefore $s: W^{s}(G_{S_\lambda})\to W^{s}(S_\lambda)$ is well restricted. As for each critical arrow $g\in G_{S_\lambda}$ we have that $t(g)\in S_\lambda$ is a critical point, then by arguing in the exactly same way with the identity  $t\circ \widetilde{\Phi_{\tau}}=\Phi_{\tau}\circ t$ we get that $t: W^{s}(G_{S_\lambda})\to W^{s}(S_\lambda)$ is well restricted. Now, let us consider the fibered product space $(W^{s}(G_{S_\lambda}))^{(2)}$ defined through $s$ and $t$. If $(h_1,h_2)\in (W^{s}(G_{S_\lambda}))^{(2)}$ then there exists a pair of critical arrows $(g_1,g_2)\in G_{S_\lambda}\times G_{S_\lambda}$ such that $h_1\in W^{s}(g_1)$ and $h_2\in W^{s}(g_2)$. On the one hand, observe that $(g_1,g_2)\in G_{S_\lambda}^{(2)}$. Indeed, since $s(h_1)=t(h_2)$ and $(\widetilde{\Phi_{\tau}},\Phi_{\tau})$ is a Lie groupoid morphism we obtain
 	$$s(g_1)=s\left(\lim_{\tau\to \infty}\widetilde{\Phi_{\tau}}(h_1)\right)=\lim_{\tau\to \infty}\Phi_{\tau}(s(h_1))=\lim_{{\tau}\to \infty}\Phi_{\tau}(t(h_2))=t\left(\lim_{\tau\to \infty}\widetilde{\Phi_{\tau}}(h_2)\right)=t(g_2).$$
	
 	On the other hand, from the identity $\widetilde{\Phi_{\tau}}\circ m=m\circ(\widetilde{\Phi_{\tau}}\times \widetilde{\Phi_{\tau}})$ we get that
 	$$\lim_{\tau\to \infty}\widetilde{\Phi_{\tau}}(h_1h_1)=\lim_{\tau\to \infty}m(\widetilde{\Phi_{\tau}}(h_1),\widetilde{\Phi_{\tau}}(h_2))=m\left(\lim_{\tau\to \infty}\widetilde{\Phi_{\tau}}(h_1),\lim_{\tau\to \infty}\widetilde{\Phi_{\tau}}(h_2)\right)=g_1g_2.$$
	
 	Given that we know that the composition of two composable critical arrows produces again a critical arrow we have that $h_1h_2\in W^{s}(g_1g_2)\subset W^{s}(G_{S_\lambda})$. Hence, the composition map $m:(W^{s}(G_{S_\lambda}))^{(2)}\to W^{s}(G_{S_\lambda})$ is also well restricted. Finally, let us now consider the inversion map $i:G\to G$ and the unit map $u:M\to G$ of the Lie groupoid $G\rightrightarrows M$. As  $(\widetilde{\Phi_{\tau}},\Phi_{\tau})$ is a Lie groupoid morphism we have that $\widetilde{\Phi_{\tau}}\circ i=i\circ \widetilde{\Phi_{\tau}}$ and $\widetilde{\Phi_{\tau}}\circ u=u\circ \Phi_{\tau}$. Recall that if $g$ is a critical arrow, then $g^{-1}$ is also a critical arrow and if $x$ is a critical point, then $1_x$ is a critical arrow. So, by arguing as in the previous items we obtain that $i:W^{s}(G_{S_\lambda})\to W^{s}(G_{S_\lambda})$ and $u:W^{s}(S_\lambda)\to W^{s}(G_{S_\lambda})$ are well restricted as well.

 \end{proof}
 The Lie groupoids introduced in Lemma \ref{Stable/UnstableGroupoids} will be respectively called \textbf{stable} and \textbf{unstable Lie groupoids} associated to the critical submanifold $S_\lambda$. Consider now the smooth \textbf{endpoint maps} $l_0: W^{s}(S_\lambda)\to S_\lambda$ and $u_0: W^{u}(S_\lambda)\to S_\lambda$ which are respectively defined by
 
 \begin{equation}\label{eq:endpointmaps}
 l_0(x)=\lim_{\tau\to \infty}\Phi_{\tau}(x)\qquad\textnormal{and}\qquad u_0(x)=\lim_{\tau\to -\infty}\Phi_{\tau}(x).
 \end{equation}

 It was shown in \cite[Prop. 3,2]{AB} that the endpoint maps \eqref{eq:endpointmaps} are smooth locally trivial fibrations. In order to state the groupoid analogue of this property, we need the following definition.
 
 \begin{definition}\cite{mackbook}
 A \textbf{Lie groupoid fibration} is a Lie groupoid morphism $\phi^1:G\to G'$ covering a surjective submersion $\phi^0:M\to N$ with the property that

\begin{equation}\label{eq:gpdfibration}
\hat{\phi}:G\to G'\times_NM; g\mapsto (\phi^1(g),s(g)),
\end{equation}
is a surjective submersion. 


 \end{definition}

At the level of arrows, we can similarly define smooth endpoint maps  $l_1: W^{s}(G_{S_\lambda})\to G_{S_\lambda}$ and $u_1: W^{u}(G_{S_\lambda})\to G_{S_\lambda}$ by using the descending flow $\widetilde{\Phi_{\tau}}$. Recall that we have the equality of indexes $\lambda(G_{S_\lambda},F)=\lambda(S_\lambda,f)=\lambda$. Thus, as consequence of \cite[Prop. 3,2]{AB} we obtain that $u_1$ and $l_1$ are smooth fiber bundles with fibers diffeomorphic to the disks $D^\lambda$ and $D^{k-\lambda}$, respectively. Here $k=\textnormal{codim}(G_{S_\lambda})=\textnormal{codim}(S_\lambda)$ seen as submanifolds of $G$ and $M$, respectively. So, motivated by this and Lemma \ref{Stable/UnstableGroupoids} one has the following result.

 \begin{proposition}\label{prop:multiplicativeendpoint}
 The endpoint maps $l=(l_1,l_0):(W^{s}(G_{S_\lambda})\rightrightarrows W^{s}(S_\lambda))\to (G_{S_\lambda}\rightrightarrows S_\lambda)$ and  $u=(u_1,u_0):(W^{u}(G_{S_\lambda})\rightrightarrows W^{u}(S_\lambda))\to (G_{S_\lambda}\rightrightarrows S_\lambda)$ are Lie groupoid fibrations.
 \end{proposition}

 \begin{proof}
 	On the one hand, the fact that both $l$ and $u$ define Lie groupoid morphisms follows by arguing exactly as we did in Lemma \ref{Stable/UnstableGroupoids}. Indeed, let us consider for instance the case of $u$. As the pair $(\widetilde{\Phi_{\tau}},\Phi_{\tau})$ determines a Lie groupoid morphism we have that
 	$$(s\circ u_1)(h)=s\left(\lim_{\tau\to -\infty} \widetilde{\Phi_{\tau}}(h)\right)=\lim_{\tau\to -\infty}s(\widetilde{\Phi_{\tau}}(h))=\lim_{\tau\to -\infty}\Phi_{\tau}(s(h))=(u_0\circ s)(h).$$
 	We can analogously get that $t\circ u_1=u_0\circ t$. If $(h_1,h_2)\in (W^{u}(G_{S_\lambda}))^{(2)}$ then
 	\begin{eqnarray*}
 		(m\circ (u_1\times u_1))(h_1,h_2) & = & m\left(\lim_{\tau\to -\infty}\widetilde{\Phi_{\tau}}(h_1),\lim_{\tau\to -\infty}\widetilde{\Phi_{\tau}}(h_2)\right) = \lim_{\tau\to -\infty} (m\circ (\widetilde{\Phi_{\tau}}\times \widetilde{\Phi_{\tau}}))(h_1,h_2)\\
 		& = &  \lim_{\tau\to -\infty} (\widetilde{\Phi_{\tau}}(m(h_1,h_2)) = (u_1\circ m)(h_1,h_2).
 	\end{eqnarray*}
 	On the other hand, let us consider the fibred product given by the diagram below 
 	$$\xymatrix{
 		G_{S_\lambda}\times_{S_\lambda}W^{u}(S_\lambda) \ar@<.5ex>[d]_{\pi_{1}}\ar[r]^{\qquad\pi_{2}} & W^{u}(S_\lambda) \ar@<-.5ex>[d]^{u_0}\\
 		G_{S_\lambda} \ar[r]_{s} & S_\lambda. 
 	}$$
 	
Now we show that $u$ has the fibration property \eqref{eq:gpdfibration}, that is, the map $\hat{u}:W^{u}(G_{S_\lambda})\to G_{S_\lambda}\times_{S_\lambda}W^{u}(S_\lambda), h\mapsto (u_1(h),s(h))$, is a surjective submersion. Indeed, since $\pi_1$ and $\pi_2$ are surjective submersions, $\pi_1\circ\hat{u}=u_1$  and $\pi_2\circ\hat{u}=s$, then $u$ is a surjective submersion because $u_1$ and $s$ are so. Hence, as $u_0:W^{u}(S_\lambda)\to S_\lambda$ is also a surjective submersion we conclude that $u$ is in fact a Lie groupoid fibration.
 \end{proof}

It is clear that the canonical projections $\pi^-_1:\nu_-(G_{S_\lambda})\to G_{S_\lambda}$ and $\pi^-_0:\nu_-(S_\lambda)\to S_\lambda$ allow us to define a groupoid fibration from $\nu_-(G_{S_\lambda})\rightrightarrows \nu_-(S_\lambda)$ onto $G_{S_\lambda}\rightrightarrows S_\lambda$. Obviously, for the positive normal groupoid a similar fibration $\pi^{+}$ onto $G_{S_\lambda}\rightrightarrows S_\lambda$ can be defined. Therefore, the naturality of the previous result is behind the following interesting fact, which can be thought of as the groupoid version of the so-called \textbf{stable/unstable manifold theorem} (see for instance \cite[Thm A.9]{AB} or else \cite[Thm 4,15]{BH3}).

\begin{theorem}[Stable/unstable groupoid theorem]\label{Stable/unstableGroupoidTheorem}
	
	There exists a Lie groupoid isomorphism from a full Lie groupoid open neighborhood of the zero section $G_{S_\lambda}\rightrightarrows S_\lambda$ in $\nu_-(G_{S_\lambda})\rightrightarrows \nu_-(S_\lambda)$ and a full Lie groupoid open neighborhood of $G_{S_\lambda}\rightrightarrows S_\lambda$ in $W^{u}(G_{S_\lambda})\rightrightarrows W^{u}(S_\lambda)$ which intertwines the groupoid fibrations $\pi^{-}$ and $u$. Such an isomorphism behaves as the identity over $G_{S_\lambda}\rightrightarrows S_\lambda$. Analogously, same assertion holds true between the groupoid fibrations $\pi^{+}$ and $l$ onto $G_{S_\lambda}\rightrightarrows S_\lambda$.
	
\end{theorem}
\begin{proof}
	First of all, it follows that around a base point $y\in S_\lambda$ there are coordinates in an open neighborhood $U\subset M$ such that $U\cong S_\lambda\times \nu_+(S_\lambda)_s\times \nu_-(S_\lambda)_s$ so that each element $z\in U$ may be rewritten as $z=(z_0,z_+,z_-)$ \cite[Appx. A.3]{AB}.  We focus on proving the statement for $\pi^{-}$ and $u$ since the other asserted case follows in a similar fashion. Pick $y\in S_\lambda$ and $v_-\in \nu_+(S_\lambda)_{y'}$. As consequence of \cite[Thm A.9]{AB} (see also \cite[Thm 4,15]{BH3}) we know that there exists a unique integral curve $c(\tau)=\Phi_{y'}(\tau):=\Phi_\tau(y')$ of $-\nabla f$ starting at $y'\in M$ such that $c(0)_-=v_-$ and $\displaystyle \lim_{\tau\to -\infty} c(\tau)= \lim_{\tau\to -\infty}\Phi_\tau(y')=y$. Furthermore, after shrinking $U$ if necessary, for varying $(y, v_-)$ the map $\psi^0$ sending $y\mapsto c(0)_0$ defines a diffeomorphism between an open neighborhood $U$ of the zero section $S_\lambda$ in $\nu_-(S_\lambda)$ and an open neighborhood $V$ of $S_\lambda$ in $W^u(S_\lambda)$ which intertwines the
	projection $\pi^{-}_0$ and endpoint map $u_0$. Again, after shrinking $U$ if necessary and by using instead the descending flow of $-\nabla F$, we may similarly define around $g\in G_{S_\lambda}$ an open neighborhood $\tilde{U}$ of the zero section $G_{S_\lambda}$ in $\nu_-(G_{S_\lambda})$ of the form $\tilde{U}=s^{-1}(U)\cap t^{-1}(U)\cong G_{S_\lambda}\times \nu_+(G_{S_\lambda})_g\times \nu_-(G_{S_\lambda})_g$ enjoying of the following properties. Each element $\tilde{z}\in \tilde{U}$ may be rewritten as $\tilde{z}=(\tilde{z}_0,\tilde{z}_+,\tilde{z}_-)$ and there is a (uniquely defined) map $\psi^1$ sending $g\mapsto \tilde{c}(0)_0$ which defines a diffeomorphism between $\tilde{U}$ and an open neighborhood $\tilde{V}$ of $G_{S_\lambda}$ in $W^u(G_{S_\lambda})$ that intertwines the
	projection $\pi^{-}_1$ and endpoint map $u_1$.
	
	Let us check that $\psi=(\psi^1,\psi^0)$ defines a Lie groupoid morphism over the open Lie groupoid neighborhood $\tilde{U}\rr U$ of $G_{S_\lambda}\rr S_\lambda$. This will be consequence of having that $\nabla F$ is a multiplicative vector field. Take $g\in \tilde{U}$, $\tilde{v}_-\in \nu_{-}(G_{S_\lambda})_{g'}$ and $\tilde{c}(\tau)=\widetilde{\Phi}_{g'}(\tau)=\widetilde{\Phi}_\tau(g')$ as described above. We know that $s\circ \widetilde{\Phi}_\tau=\Phi_\tau\circ s$ which in turn yields the identity $s\circ \widetilde{\Phi}_h=\Phi_{s(h)}$ for all $h\in G$. The latter formula implies that $\displaystyle s(g)= \lim_{\tau\to -\infty}\Phi_{s(g')}(\tau)=\lim_{\tau\to -\infty}\Phi_\tau(s(g'))$ and $\overline{ds}(g')(\tilde{v}_-)=c(0)_-$ where in this case we  are denoting $c(\tau)=\Phi_{s(g')}(\tau)=\Phi_\tau(s(g'))$. Thus, by the uniqueness in \cite[Thm A.9]{AB} it follows that
	$$(s\circ \psi^1)(g)=s(\tilde{c}(0)_0)=(s\circ \widetilde{\Phi}_{g'})(0)_0=\Phi_{s(g')}(0)_0=c(0)_0=(\psi^0\circ s)(g).$$
	The identity $t\circ \psi^1=\psi^0\circ t$ can be verified in a similar manner. Let us now consider pairs $g\in \tilde{U}$, $\tilde{v}_-\in \nu_{-}(G_{S_\lambda})_{g'}$ and $h\in \tilde{U}$, $\tilde{w}_-\in \nu_{-}(G_{S_\lambda})_{h'}$ such that $s(g)=t(h)$, $s(g')=t(h')$, $\overline{ds}(g')(\tilde{v}_-)=\overline{dt}(h')(\tilde{w}_-)$ and $\tilde{a}(\tau)=\widetilde{\Phi}_{g'}(\tau)=\widetilde{\Phi}_\tau(g')$ with $\tilde{b}(\tau)=\widetilde{\Phi}_{h'}(\tau)=\widetilde{\Phi}_\tau(h')$ as described above. It is simple to verify that the equality $\widetilde{\Phi}_\tau\circ m=m\circ (\widetilde{\Phi}_\tau\times \widetilde{\Phi}_\tau)$ implies the formula $m\circ (\widetilde{\Phi}_{h_1}\times\widetilde{\Phi}_{h_2})=\widetilde{\Phi}_{h_1h_2}$ for all $(h_1,h_2)\in G^{(2)}$. It follows that $\displaystyle gh= \lim_{\tau\to -\infty}\tilde{\Phi}_{g'h'}(\tau)=\lim_{\tau\to -\infty}\tilde{\Phi}_{\tau}(g'h')$ and $\overline{dm}(g',h')(\tilde{v}_-,\tilde{w}_-)=\tilde{a}(0)_-$ with $\tilde{c}(\tau)=\tilde{\Phi}_{g'h'}(\tau)=\tilde{\Phi}_{\tau}(g'h')$. Hence, again by the uniqueness in \cite[Thm A.9]{AB} we obtain that
	
	$$m(\psi^1(g),\psi^1(h))=m(\tilde{a}(0)_0,\tilde{b}(0)_0)=m(\widetilde{\Phi}_{g'}(0),\widetilde{\Phi}_{h'}(0))_0=\widetilde{\Phi}_{h_1h_2}(0)_0=\tilde{c}(0)_0=\psi^1(gh).$$
	
	Observe that $\psi$ is the identity over $G_{S_\lambda}\rightrightarrows S_\lambda$ since the descending flows fix critical points. The proof is completed by observing that in a proper groupoid every groupoid neighborhood contains a full groupoid neighborhood, see for instance \cite[Lem. 5.3]{dHF}.
	
\end{proof}
 
We consider now the notion of moduli space of gradient flow lines in the Lie groupoid setting. Some of the ideas stated below will be used to construct a double complex which in turn will allow us to recover the total cohomology of the Bott--Shulman--Stasheff double complex of the Lie groupoid we are working with. Note that there exists a natural action of $\mathbb{R}$ on $W^{u}(S_\lambda)$ (resp. on $W^{s}(S_\lambda)$) defined by $r\cdot y=\Phi_r(y)$ for all $r\in \mathbb{R}$. This is well defined since if $y\in W^{u}(S_\lambda)$ we have that
 \begin{equation}\label{FlowAction}
 	\lim_{\tau\to -\infty}\Phi_{\tau}(r\cdot y)=\lim_{\tau\to -\infty}\Phi_{\tau}(\Phi_r(y))=\lim_{\tau\to -\infty}\Phi_{\tau+r}(y)\in W^{u}(S_\lambda).
 \end{equation}
 In particular, it is simple to see that these actions induce a well defined free action of $\mathbb{R}$ on any intersection $W^{u}(S_{\lambda_i})\cap W^{s}(S_{\lambda_j})$ for $\lambda_i\neq \lambda_j$. We will refer to this action as \textbf{action by flow translation}. The \textbf{moduli space of gradient flow lines} in $M$ associated to the non-degenerate critical saturated submanifolds $S_{\lambda_i}$ and $S_{\lambda_j}$ is defined as the quotient space obtained from the action by flow translation:
 $$\mathcal{M}(S_{\lambda_i},S_{\lambda_j}):=(W^{u}(S_{\lambda_i})\cap W^{s}(S_{\lambda_j}))/\mathbb{R}.$$
 The moduli space of gradient flow lines $\mathcal{M}(G_{S_{\lambda_i}},G_{S_{\lambda_j}})$ is equally defined by using the action induced by the descending flow $\widetilde{\Phi_{\tau}}$. Thus, motivated by the Morse--Bott transversality condition (see \cite{AB,BH2}) we set up the following definition.

 \begin{definition}\label{MStransversalityDef}
 
 A Morse Lie groupoid morphism $F:(G\rightrightarrows M)\to (\mathbb{R}\rightrightarrows \mathbb{R})$ is said to satisfy the \textbf{Morse--Smale transversality} condition with respect to a $2$-metric $\eta^{(2)}$ on $G\rightrightarrows M$ if  for any two critical non-degenerate saturated submanifolds $S_{\lambda_i}$ and $S_{\lambda_j}$ with respect to $f$ we have that $W^u(y)\pitchfork W^s(S_{\lambda_j})$ for all $y\in S_{\lambda_i}$.
 \end{definition}

 It is important to notice that, unlike the classical Morse case, in the Morse--Bott case it is not always possible to perturb a Riemannian metric to make a given Morse-Bott function satisfy the Morse-Bott-Smale transversality condition. See \cite[Rmk. 2.4]{L} for an interesting counterexample. More importantly, such a condition is not always satisfied for $G$-invariant Morse functions; compare \cite[s. 5]{L}. As a consequence, not every Morse Lie groupoid morphism satisfies the Morse--Smale transversality condition in general. As it is argued in \cite{AB}, when assuming the Morse--Smale transversality condition from Definition \ref{MStransversalityDef}, we get that the spaces $W^{u}(S_{\lambda_i})\cap W^{s}(S_{\lambda_j})$ and $\mathcal{M}(S_{\lambda_i},S_{\lambda_j})$ are smooth manifolds since the action by flow translation is free and proper. The endpoint maps \eqref{eq:endpointmaps} descend to the moduli spaces yielding new endpoint maps $(l_0)_j^i:\mathcal{M}(S_{\lambda_i},S_{\lambda_j})\to S_{\lambda_j}$ and $(u_0)_j^i:\mathcal{M}(S_{\lambda_i},S_{\lambda_j})\to S_{\lambda_i}$ given respectively by
 
 \begin{equation}\label{eq:moduliendpointmaps}
 (l_0)_j^i([x])=l_0(x)\qquad\textnormal{and}\qquad(u_0)_j^i([x])=u_0(x),
 \end{equation}

\noindent One can see that these maps are well defined in a similar manner as in \eqref{FlowAction}. Also, the endpoint map $(u_0)_j^i$ on the moduli space $\mathcal{M}(S_{\lambda_i},S_{\lambda_j})$ is a locally trivial fibration, see \cite{AB}. 
 \begin{lemma}\label{MStransversal}
 	Let $F:G\to \mathbb{R}$ be a Morse Lie groupoid morphism having the Morse-Smale transversality property. Then the non-degenerate critical submanifolds $G_{S_{\lambda_i}}$ and $G_{S_{\lambda_j}}$ satisfy the Morse-Smale transversality condition with respect to $F$.
 \end{lemma}
 
 \begin{proof}
 	Let $y\xleftarrow[]{\it g}x$ be an arrow in $W^u(G_{S_{\lambda_i}})$ such that $g\in W^{u}(h)\cap W^s(G_{S_{\lambda_j}})$. Given that $s:G\to M$ is a submersion we have that it is transverse to any submanifold in $M$ so that, in particular, we obtain that $s \pitchfork  W^{u}(s(h))$ and $s \pitchfork W^s(S_{\lambda_j})$. Since $x\in W^{u}(s(h))\cap  W^s(S_{\lambda_j})$ the Morse-Smale transversality condition with respect to $f$ gives us
 	\begin{eqnarray*}
 		T_g W^{u}(h) +T_gW^s(G_{S_{\lambda_j}}) & = & T_g(s^{-1}(W^u(s(h))))+T_g(s^{-1}(W^s(S_{\lambda_j})))\\
 		& = & ds_{g}^{-1}(T_x W^{u}(s(h)))+ds_{g}^{-1}(T_xW^s(S_{\lambda_j}))\\
 		& = & ds_{g}^{-1}(T_x W^{u}(s(h))+T_xW^s(S_{\lambda_j}))\\
 		& = & ds_{g}^{-1}(T_x M) =  T_g G.
 	\end{eqnarray*}
 	So, the result follows.
 \end{proof}
 As consequence of this elementary observation we can analogously define smooth manifolds $W^{u}(G_{S_{\lambda_i}})\cap W^{s}(G_{S_{\lambda_j}})$ and $\mathcal{M}(G_{S_{\lambda_i}},G_{S_{\lambda_j}})$ and smooth maps $(l_1)_j^i:\mathcal{M}(G_{S_{\lambda_i}},G_{S_{\lambda_j}})\to G_{S_{\lambda_j}}$ and $(u_1)_j^i:\mathcal{M}(G_{S_{\lambda_i}},G_{S_{\lambda_j}})\to G_{S_{\lambda_i}}$ so that $(u_1)_j^i$ has the structure of a locally trivial bundle.

 It is simple to check that by the nature of the structure of our stable and unstable Lie groupoids we may naturally define a Lie groupoid $W^{u}(G_{S_{\lambda_i}})\cap W^{s}(G_{S_{\lambda_j}})\rightrightarrows W^{u}(S_{\lambda_i})\cap W^{s}(S_{\lambda_j})$. More importantly:
 
 \begin{proposition}[Groupoid of gradient flow lines]\label{ModuliGroupoid}
 	There exists a unique structure of Lie groupoid $\mathcal{M}(G_{S_{\lambda_i}},G_{S_{\lambda_j}})\rightrightarrows \mathcal{M}(S_{\lambda_i},S_{\lambda_j})$ on the moduli spaces of gradient lines, making the canonical projection $W^{u}(G_{S_{\lambda_i}})\cap W^{s}(G_{S_{\lambda_j}})\to \mathcal{M}(G_{S_{\lambda_i}},G_{S_{\lambda_j}})$ a Lie groupoid fibration.
 \end{proposition}
 \begin{proof}
 
 The action of $\mathbb{R}$ on $W^{u}(G_{S_{\lambda_i}})\cap W^{s}(G_{S_{\lambda_j}})\rightrightarrows W^{u}(S_{\lambda_i})\cap W^{s}(S_{\lambda_j})$ is given by Lie groupoid automorphisms. Since this action is free and proper, then the quotient inherits a Lie groupoid structure with the desired properties. See for instance the proof of Proposition 3.4 in \cite{HV}. 

\end{proof}
As an application of Proposition \ref{prop:multiplicativeendpoint} we easily get:
\begin{corollary}
The pair of endpoint maps $u_j^i=((u_1)_j^i,(u_0)_j^i):(\mathcal{M}(G_{S_{\lambda_i}},G_{S_{\lambda_j}})\rr \mathcal{M}(S_{\lambda_i},S_{\lambda_j}))\to (G_{S_{\lambda_i}}\rr S_{\lambda_i})$ defines a Lie groupoid fibration.
\end{corollary}

Note that because of the well definition of index we have that $W^{u}(S_{\lambda_i})\cap W^{s}(S_{\lambda_j})$ does not contain critical points of $f$ since $S_{\lambda_i}\cap S_{\lambda_j}=\emptyset$ for $\lambda_i\neq \lambda_j$. This in particular implies that the action by flow translation on this space is free. Namely, the latter statement follows from the fact that a point is a singularity of a vector field if and only if its flow fixes such a point. In particular, the gradient vector field of $f$ is nonzero at the regular points of $f$ which implies that its flow does not fix such points. Same conclusion can be obtained at the arrow level. This key observation allows us to define the composition map of the Lie groupoid from Proposition \ref{ModuliGroupoid} explicitly. Indeed, the source and target maps are respectively defined by setting $\overline{s}:\mathcal{M}(G_{S_{\lambda_i}},G_{S_{\lambda_j}})\to \mathcal{M}(S_{\lambda_i},S_{\lambda_j})$ as $\overline{s}[g]=[s(g)]$ and $\overline{t}:\mathcal{M}(G_{S_{\lambda_i}},G_{S_{\lambda_j}})\to \mathcal{M}(S_{\lambda_i},S_{\lambda_j})$ as $\overline{t}[g]=[t(g)]$. They are well-defined since our gradient vector field is multiplicative. Furthermore, if we take $[g]$ and $[h]$ in $\mathcal{M}(G_{S_{\lambda_i}},G_{S_{\lambda_j}})$ such that $\overline{s}[g]=\overline{t}[h]$ then there exists $r\in \mathbb{R}$ such that $\Phi_r(s(g))=t(h)$ which is equivalent to have $s(\widetilde{\Phi_r}(g))=t(h)$ so that we can multiply $m(\widetilde{\Phi_r}(g),h)=\widetilde{\Phi_r}(g)h$. Thus, we define $\overline{m}:\mathcal{M}(G_{S_{\lambda_i}},G_{S_{\lambda_j}})^{(2)}\to \mathcal{M}(G_{S_{\lambda_i}},G_{S_{\lambda_j}})$ as $\overline{m}([g],[h]):=[\widetilde{\Phi_r}(g)h]$.

Having this in mind, an interesting consequence of Proposition \ref{ModuliGroupoid} is the following.
\begin{remark}
As it can be viewed for instance in \cite{AB}, the Morse--Smale transversality assumption allows to show that there exists a way for constructing a compactification of these moduli spaces of gradient flow lines. Without going into too much details, it can be done by applying a series of fibered products with respect to the endpoint maps $u$ and $\l$. Therefore, as consequence of Proposition \ref{ModuliGroupoid}, \cite[Lem. 3.3]{AB} and \cite[Prop. 4.4.1]{dH} we may extend the Lie groupoid structure previously described to the respective compactifications. Here it is important to have in mind that $u$ and $l$ define Lie groupoid morphisms where $u$ is composed by locally trivial fibrations.
\end{remark}


\section{The double complex of a Morse Lie groupoid morphism}\label{S:9}

Given a Morse function $f:M\to \mathbb{R}$, the Morse-Smale-Witten complex of the pair $(M,f)$ is a complex whose homology computes the singular homology of the manifold $M$. In the setting of Morse-Bott theory, there are several versions of the Morse-Smale-Witten complex which compute either the homology or the de Rham cohomology of the manifold. We end the paper by introducing the groupoid version of the Morse complex. Given a Morse Lie groupoid morphism $F:G\to \mathbb{R}$ we construct a double cochain complex whose total cohomology recovers the total cohomology of the Bott--Shulman--Stasheff double complex associated to $G$. Our construction relies on Austin and Braam version of the Morse-Bott complex \cite{AB}.


\subsection{The Austin-Braam complex}\label{sec:AustinBraam}

Let $f:M\to \mathbb{R}$ be a Morse-Bott function and pick a Riemannian metric on $M$. Any non-degenerate critical submanifold $S\subseteq M$ determines endpoint maps $l_S:W^s(S)\to S$ and $u_S:W^u(S)\to S$ defined as in \eqref{eq:endpointmaps}. Hence, for non-degenerate critical submanifolds $S,S'$ for $f$, one considers the endpoint maps $l_{S'}:\mathcal{M}(S,S')\to S'$ and $u_S:\mathcal{M}(S,S')\to S$ defined on the moduli spaces of gradient flow lines \eqref{eq:moduliendpointmaps}. It was shown in \cite{AB} that the endpoint maps are locally trivial fibrations. 

Denote by $S_{i}$ the set of critical points of index $i$ and by $C^{i,j}:=\Omega^j(S_i)$ the set of $j$-forms on $S_i$. The cochain complex defined in \cite{AB} is built out of certain maps $\partial_r:C^{i,j}\to C^{i+r, j-r+1}$ which are the composition of a pullback of forms followed by integration over the fibers of the endpoint maps on the moduli spaces. This requires the following:

\begin{assumption}[Austin-Braam]\label{ABAssumption}
	
	\vspace{.2cm} 
\begin{enumerate} 

\item $f$ is weakly self indexing in the sense that $\mathcal{M}(S_i,S_j)=\emptyset$ for $i\leq j$.
\item For all $i,j$ and $x\in S_i$ we have that $W_i^u(x):=u_i^{-1}(x)$ intersects $W^s(S_j)$ transversally.
 		
\item the critical submanifolds $S_i$ and the negative normal bundles $\nu_-(S_i)$ are orientable for all $i$.
 \item $M$ is compact and orientable.
 \end{enumerate}
  \end{assumption}

With this, Austin and Braam defined the maps $\partial_r:C^{i,j}\to C^{i+r, j-r+1}$ as

\begin{equation}\label{eq:ABmaps}
\partial_r(\omega):=\begin{cases} d\omega, \quad r=0\\
(-1)^j(u_{i}^{i+r})_\ast(l_{i}^{i+r})^{*}(\omega), \quad \textnormal{otherwise}
\end{cases}
\end{equation}
where $(u_{i}^{i+r})_\ast$ denotes the integration of forms over the fibers of the bundle \eqref{eq:moduliendpointmaps}. One of the main results of \cite{AB} establishes that for each $j$

$$\sum_{m=0}^{j}\partial_{j-m}\circ \partial_m=0,$$
hence the maps $\partial_r$ fit together into a cochain complex

\begin{equation}\label{eq:ABcomplex}
C^{k}:=\bigoplus_{i=0}^{k}C^{i,k-i}; \quad \partial:=\partial_1\oplus \cdots \oplus \partial_m
\end{equation}
whose cohomology is isomorphic to the de Rham cohomology of $M$. In the next subsection we will see the Lie groupoid version of this complex, yielding a double complex whose total cohomology is isomorphic to the Bott-Shulmann-Stasheff cohomology.


\subsection{The double complex}
Throughout this section a Lie groupoid is denoted as $G^{(1)} \rightrightarrows G^{(0)}$. Let $F_1:(G^{(1)}\rightrightarrows G^{(0)})\to (\mathbb{R}\rightrightarrows \mathbb{R})$ be a Morse Lie groupoid morphism which is induced by a smooth basic function $F_0:G^{(0)}\to \mathbb{R}$. We assume that $G^{(1)} \rightrightarrows G^{(0)}$ is proper and either $G^{(1)}$ or $G^{(0)}$ is compact (see Remark \ref{CompactRelaxed}).

Let us consider the nerve $G^{(\bullet)}$ of $G^{(1)} \rightrightarrows G^{(0)}$ that is depicted as 
 $$G^{(\bullet)}:\ \cdots \begin{tikzpicture}[scale=1.0,baseline=-0.1cm, inner sep=1mm,>=stealth]
 	\node (-2) at (-4,0) {$G^{(n)}$};
 	\node (-1) at (-2,0) {$\cdots$};
 	\node (0) at (0,0) {$G^{(2)}$};
 	\node (1) at (2,0)  {$G^{(1)}$};
 	\node (2) at (4,0) {$G^{(0)}.$};
 	\tikzset{every node/.style={fill=white}} 
 	\draw[->,  thick] (-5.5,0.2) to  (-4.5,0.2); 
 	\draw[->,  thick] (-5.5,0.1) to  (-4.5,0.1); 
 	\draw[->,  thick] (-5.5,0) to  (-4.5,0);
 	\draw[->,  thick] (-5.5,-0.1) to  (-4.5,-0.1);
 	\draw[->,  thick] (-5.5,-0.2) to  (-4.5,-0.2);
 	\draw[->,  thick] (-3.5,0.1) to  (-2.5,0.1); 
 	\draw[->,  thick] (-3.5,0) to  (-2.5,0);
 	\draw[->,  thick] (-3.5,-0.1) to  (-2.5,-0.1);
 	\draw[->,  thick] (-3.5,-0.2) to  (-2.5,-0.2);
 	\draw[->,  thick] (-1.5,0.1) to  (-0.5,0.1); 
 	\draw[->,  thick] (-1.5,0) to  (-0.5,0);
 	\draw[->,  thick] (-1.5,-0.1) to  (-0.5,-0.1);
 	\draw[->,  thick] (-1.5,-0.2) to  (-0.5,-0.2);
 	\draw[->,  thick] (0.5,0.1) to  (1.5,0.1);  
 	\draw[->,  thick] (0.5,0) to  (1.5,0);  
 	\draw[->,  thick] (0.5,-0.1) to  (1.5,-0.1);  
 	\draw[->,  thick] (2.5,0.1) to  (3.5,0.1);
 	\draw[->,  thick] (2.5,-0.1) to  (3.5,-0.1);
 \end{tikzpicture}$$
 The manifolds $G^{(n)}$ of $n$-\textbf{composable arrows} are given by
 $$G^{(n)}=G^{(1)}\times_{G^{(0)}}\cdots \times_{G^{(0)}}G^{(1)}=\lbrace (g_n,\cdots,g_1):\ s(g_j)=t(g_{j-1}); \ j=2,\cdots,n\rbrace,$$
 and the left arrows represent the \textbf{face maps} $d_k^n:G^{(n)}\to G^{(n-1)}$ for $k=0,\cdots,n$. These are defined as $d_0^1=t$, $d_1^1=s$ and
 $$d_k^n(g_n,\cdots,g_1)= \left\{ \begin{array}{lcc}
 	(g_n,\cdots,g_2) &   \textnormal{if}  & k=0 \\
 	(g_n,\cdots,g_{k+2},g_{k+1}g_{k},g_{k-1},\cdots,g_1) &  \textnormal{if} & k=1,\cdots,n-1 \\
 	(g_{n-1},\cdots,g_1) &  \textnormal{if}  & k=n.
 \end{array}
 \right.$$
 The face maps are surjective submersions and they satisfy the so-called \textbf{simplicial identities}
 \begin{equation}\label{DC1}
 	d_{k}^{n-1}\circ d_{k'}^n=d_{k'-1}^{n-1}\circ d_{k}^n,\qquad k<k'.
 \end{equation}
By our initial assumptions it follows that $G^{(n)}$ is compact since it is closed and is contained inside $G^{(1)}\times \cdots \times G^{(1)}$. From \cite{dHF} we know that the Lie groupoid $G^{(1)} \rightrightarrows G^{(0)}$ admits an $n$-metric $\eta^{(n)}$ on $G^{(n)}$. It is important to remember that we can push $\eta^{(n)}$ forward with the different face maps $d_k^n:G^{(n)}\to G^{(n-1)}$ to define an $(n-1)$-metric $\eta^{(n-1)}$ on $G^{(n-1)}$ where $\eta^{(n-1)}=(d_k^n)_\ast\eta^{(n)}=(d_{k'}^n)_\ast\eta^{(n)}$ for all $k,k'$ and every $d_k^n:(G^{(n)},\eta^{(n)})\to (G^{(n-1)},\eta^{(n-1)})$ becomes a Riemannian submersion. One can use this process to obtain $r$-metrics $\eta^{(r)}$ on $G^{(r)}$ in such a way that $d_k^r:(G^{(r)},\eta^{(r)})\to (G^{(r-1)},\eta^{(r-1)})$ is a Riemannian submersion for every $0\leq r\leq n-1$.

Applying the nerve functor to $F_1:G^{(1)}\to \mathbb{R}$ yields a simplicial function $F_\bullet=(F_n)_{n\in\mathbb{N}}$ between the nerves:
$$F_{\bullet}:\ \begin{tikzpicture}[scale=1.0,baseline=-0.1cm, inner sep=1mm,>=stealth]
	\node (-4) at (-6,0) {$\cdots$};
	\node (-3) at (-6,-2) {$\cdots$};
	\node (-2) at (-4,0) {$G^{(n)}$};
	\node (-1) at (-2,0) {$\cdots$};
	\node (0) at (0,0) {$G^{(2)}$};
	\node (1) at (2,0)  {$G^{(1)}$};
	\node (2) at (4,0) {$G^{(0)}$};
	\node (3) at (-4,-2) {$\mathbb{R}$};
	\node (4) at (-2,-2) {$\cdots$};
	\node (5) at (0,-2) {$\mathbb{R}$};
	\node (6) at (2,-2)  {$\mathbb{R}$};
	\node (7) at (4,-2) {$\mathbb{R},$};
	\tikzset{every node/.style={fill=white}} 
	\draw[->,  thick] (-5.5,0.2) to  (-4.5,0.2); 
	\draw[->,  thick] (-5.5,0.1) to  (-4.5,0.1); 
	\draw[->,  thick] (-5.5,0) to  (-4.5,0);
	\draw[->,  thick] (-5.5,-0.1) to  (-4.5,-0.1);
	\draw[->,  thick] (-5.5,-0.2) to  (-4.5,-0.2);
	\draw[->,  thick] (-3.5,0.1) to  (-2.5,0.1); 
	\draw[->,  thick] (-3.5,0) to  (-2.5,0);
	\draw[->,  thick] (-3.5,-0.1) to  (-2.5,-0.1);
	\draw[->,  thick] (-3.5,-0.2) to  (-2.5,-0.2);
	\draw[->,  thick] (-1.5,0.1) to  (-0.5,0.1); 
	\draw[->,  thick] (-1.5,0) to  (-0.5,0);
	\draw[->,  thick] (-1.5,-0.1) to  (-0.5,-0.1);
	\draw[->,  thick] (-1.5,-0.2) to  (-0.5,-0.2);
	\draw[->,  thick] (0.5,0.1) to  (1.5,0.1);  
	\draw[->,  thick] (0.5,0) to  (1.5,0);  
	\draw[->,  thick] (0.5,-0.1) to  (1.5,-0.1);  
	\draw[->,  thick] (2.5,0.1) to  (3.5,0.1);
	\draw[->,  thick] (2.5,-0.1) to  (3.5,-0.1);
	\draw[->,  thick]  (-2) to node[midway,left] {$F_{n}$} (3);
	\draw[->,  thick]  (-1) to (4);
	\draw[->,  thick]  (0) to node[midway,left] {$F_{2}$} (5);
	\draw[->,  thick]  (1) to node[midway,left] {$F_{1}$} (6);
	\draw[->,  thick]  (2) to node[midway,right] {$F_{0}$} (7);
	\draw[->,  thick]  (3) to (4);
	\draw[->,  thick]  (4) to (5);
	\draw[->,  thick]  (5) to (6);
	\draw[->,  thick]  (6) to (7);
	\draw[->,  thick]  (-5.5,-2) to (3);
\end{tikzpicture}$$
which, for $n\geq 2$, it is inductively defined as $F_n = (d_{k}^n)^\ast F_{n-1}=(d_{k'}^n)^\ast F_{n-1}$ for all $k,k'$. These smooth functions are well defined since $F_1$ is a Lie groupoid morphism and the simplicial identities \eqref{DC1} hold true. Recall that, in our case, the set of critical points of $F_0$ is given by a disjoint union of finite compact and connected Lie groupoid orbits $\textnormal{Crit}(F_0)=\dot{\bigcup}\mathcal{O}$ since $G_0$ is compact and $G^{(1)} \rightrightarrows G^{(0)}$ is proper.

\begin{lemma}\label{DCL1}
There exists a sub-nerve $G^{(\bullet)}_i$ of $G^{(\bullet)}$ formed by non-degenerate critical submanifolds of index $i$ for $F_\bullet$ of the form:
$$G^{(\bullet)}_i:\ \cdots \begin{tikzpicture}[scale=1.0,baseline=-0.1cm, inner sep=1mm,>=stealth]
	\node (-2) at (-4,0) {$G^{(n)}_i$};
	\node (-1) at (-2,0) {$\cdots$};
	\node (0) at (0,0) {$G^{(2)}_i$};
	\node (1) at (2,0)  {$G^{(1)}_i$};
	\node (2) at (4,0) {$G^{(0)}_i,$};
	\tikzset{every node/.style={fill=white}} 
	\draw[->,  thick] (-5.5,0.2) to  (-4.5,0.2); 
	\draw[->,  thick] (-5.5,0.1) to  (-4.5,0.1); 
	\draw[->,  thick] (-5.5,0) to  (-4.5,0);
	\draw[->,  thick] (-5.5,-0.1) to  (-4.5,-0.1);
	\draw[->,  thick] (-5.5,-0.2) to  (-4.5,-0.2);
	\draw[->,  thick] (-3.5,0.1) to  (-2.5,0.1); 
	\draw[->,  thick] (-3.5,0) to  (-2.5,0);
	\draw[->,  thick] (-3.5,-0.1) to  (-2.5,-0.1);
	\draw[->,  thick] (-3.5,-0.2) to  (-2.5,-0.2);
	\draw[->,  thick] (-1.5,0.1) to  (-0.5,0.1); 
	\draw[->,  thick] (-1.5,0) to  (-0.5,0);
	\draw[->,  thick] (-1.5,-0.1) to  (-0.5,-0.1);
	\draw[->,  thick] (-1.5,-0.2) to  (-0.5,-0.2);
	\draw[->,  thick] (0.5,0.1) to  (1.5,0.1);  
	\draw[->,  thick] (0.5,0) to  (1.5,0);  
	\draw[->,  thick] (0.5,-0.1) to  (1.5,-0.1);  
	\draw[->,  thick] (2.5,0.1) to  (3.5,0.1);
	\draw[->,  thick] (2.5,-0.1) to  (3.5,-0.1);
\end{tikzpicture}$$
where $G^{(n)}_i=\dot{\bigcup}  \lbrace G^{(n)}_{\mathcal{O}}\subset \textnormal{Crit}(F_n):\ \lambda(F_n,G^{(n)}_{\mathcal{O}})=i\rbrace$ and $G^{(n)}_{\mathcal{O}}$ denotes the manifold of $n$-composable arrows of $G_\mathcal{O}\rightrightarrows \mathcal{O}$.
\end{lemma}

\begin{proof}
This follows by applying simple arguments from the previous sections since $F_1$ is a Morse Lie groupoid morphism. Namely, $G^{(0)}_i$ is a saturated submanifold with $G^{(1)}_i=s^{-1}(G^{(0)}_i)=t^{-1}(G^{(0)}_i)$ and all the face maps of $G^{(\bullet)}$ are Riemannian submersions verifying the simplicial identities \eqref{DC1}. 
\end{proof}
Let us now consider the collection of vector fields $-\nabla F_\bullet=(-\nabla F_n)_{n\in\mathbb{N}}$ where $-\nabla F_n$ denotes the negative gradient vector field of $F_n$ on $G^{(n)}$ with respect to $\eta^{(n)}$. As the face maps $d_k^n:G^{(n)}\to G^{(n-1)}$ are Riemannian submersions we actually have that $-\nabla F_\bullet$ defines a simplicial vector field on the nerve $G^{(\bullet)}$ since $\eta^{(n-1)}=(d_k^n)_\ast\eta^{(n)}$ so that $d(d_k^n)\circ \nabla F_n=\nabla F_{n-1}\circ d_k^n$. It turns out that the collection of descending flows $\Phi^{n}_{\tau}:G^{(n)}\to G^{(n)}$, which are defined for all $\tau\in \mathbb{R}$ since $G^{(n)}$ is compact, verifies the relations
\begin{equation}\label{DC2}
	d_k^n\circ \Phi^{n}_{\tau}=\Phi^{n-1}_{\tau}\circ d_k^n,\qquad k=0,\cdots,n\quad \textnormal{and}\quad \forall \tau\in\mathbb{R},
\end{equation}
thus obtaining a simplicial automorphism $\Phi^{\bullet}_\tau:G^{(\bullet)}\to G^{(\bullet)}$ for all $\tau\in \mathbb{R}$. The collection of smooth endpoint maps associated to the collection of stable and unstable submanifolds $W^u(G^{(n)}_i)$ and $W^s(G^{(n)}_i)$ will be denoted by $u_i(n):W^u(G^{(n)}_i)\to G^{(n)}_i$ and $l_i(n):W^s(G^{(n)}_i)\to G^{(n)}_i$. Recall that these maps are locally trivial fiber bundles respectively defined by sending $A\mapsto \lim_{\tau\to -\infty} \Phi^{n}_{\tau}(A)$ and $A\mapsto \lim_{\tau\to \infty} \Phi^{n}_{\tau}(A)$. It follows directly from Identities \eqref{DC2} that
\begin{equation}\label{DC3}
	d_k^n\circ u_i(n)=u_i(n-1)\circ d_k^n\quad \textnormal{and}\quad d_k^n\circ l_i(n)=l_i(n-1)\circ d_k^n,
\end{equation}
for all $k=0,\cdots,n$. For indexes $i$ and $j$ we can consider again the moduli spaces of gradient flow lines
$$\mathcal{M}^n(G^{(n)}_i,G^{(n)}_j)=W^u(G^{(n)}_i)\cap W^s(G^{(n)}_j)/\mathbb{R},$$
which are the quotient spaces defined through the action by flow translation.
\begin{lemma}\label{DCL2}
The simplicial structure of the nerve $G^{(\bullet)}$ may be restricted to define the stable and unstable sub-nerves $W^s(G^{(\bullet)}_i)$ and $W^u(G^{(\bullet)}_i)$ of the sub-nerve $G^{(\bullet)}_i$ and a topological sub-nerve structure $\mathcal{M}^\bullet(G^{(\bullet)}_i,G^{(\bullet)}_j)$ between the moduli spaces of gradient flow lines.	
\end{lemma}
\begin{proof}
Note that the first statement follows by flowing to $\infty$ and $-\infty$ at both sides of Equations \eqref{DC2} so that $d_k^n: W^s(G^{(n)}_i)\to W^s(G^{(n-1)}_i)$ and $d_k^n: W^u(G^{(n)}_i)\to W^u(G^{(n-1)}_i)$ are well restricted. Furthermore, the maps $\overline{d_k^n}:\mathcal{M}^n(G^{(n)}_i,G^{(n)}_j)\to \mathcal{M}^{n-1}(G^{(n-1)}_i,G^{(n-1)}_j)$ defined by sending $[A]_n\mapsto [d_k^n(A)]_{n-1}$ are well defined, open and surjective since $d_k^n$ are so.
\end{proof}
We are interested in inducing well behaved structures of smooth manifold and orientability over our moduli spaces of gradient flow lines. To obtain this we require the following analogue of the Austin-Braam assumption (see Subsection \ref{sec:AustinBraam}).

 \begin{assumption}\label{StrongAssumption}
\vspace{.2cm} 	
\begin{enumerate}
 \item[(a)] $\mathcal{M}^0(G^{(0)}_i,G^{(0)}_j)=\emptyset$ if $i\leq j$ ($F_0$ is weakly self-indexing).
 \item[(b)] For all $i,j$ and $x\in G^{(0)}_i$ we have that $W_i^u(x):=u_i(0)^{-1}(x)$ intersects $W^s(G^{(0)}_j)$ transversally:
 		$$W_i^u(x)\pitchfork W^s(G^{(0)}_j).$$
 \item[(c)] Both the critical submanifolds $G^{(0)}_i$ and the negative normal bundles $\nu_-(G^{(0)}_i)$ are orientable for all $i$.
 \item[(d)] $G^{(n)}$ is orientable for all $n\in \mathbb{N}$.
 \end{enumerate}
 \end{assumption}
 \begin{remark}
 Although condition (d) from previous assumption may seem to be somewhat restrictive there are many cases where such a requirement can be achieved. For instance, if $G^{(0)}$ is orientable then every manifold conforming the nerve of: the unit groupoid, the pair groupoid, the action groupoid defined through a smooth action of a Lie group on $G^{(0)}$, and \'etale Lie groupoids over $G^{(0)}$, is orientable.
 \end{remark}

\begin{lemma}
Conditions (a)-(b)-(c) from Assumption \ref{StrongAssumption} are satisfied at every level of the nerve configuration. 
\end{lemma}
\begin{proof}
First, by Lemma \ref{DCL2} and proceeding by induction over $n$ it is simple to check that $\mathcal{M}^{n-1}(G^{(n-1)}_i,G^{(n-1)}_j)=\emptyset$ if $i\leq j$ implies that $\mathcal{M}^n(G^{(n)}_i,G^{(n)}_j)=\emptyset$ if $i\leq j$. The step $n=1$ is consequence of (a) and the well definition of $\overline{d_k^1}$.

Proceeding again by induction over $n$, we may prove that if the Morse--Smale transversality condition (b) holds true at the level $n-1$ then it also holds true at the level $n$. The case $n=1$ was argued in Lemma \ref{MStransversal}. The inductive cases follow by arguing in the exactly same way but with the submersions $d_k^n$.

Finally, let us now look at the orientability requirements (c). On the one hand, by the Transverse Submanifold Theorem, it is well known that because of conditions (c) and (d) together the fact that $d_k^n$ are submersions with $G^{(n)}_i=(d_k^n)^{-1}(G^{(n-1)}_i)$, it follows that if $G^{(n-1)}_i$ is orientable then so is $G^{(n)}_i$. The induced orientation on $G^{(n)}_i$ is the one obtained from the orientability of $G^{(n)}$, $G^{(n-1)}$ and $G^{(n-1)}_i$. On the other hand, recall that in Proposition \ref{NegativeNormalLG} we proved that the Lie groupoid structure of $\nu(G^{(1)}_{i})\rightrightarrows \nu(G^{(0)}_i)$ can be well restricted to define a Lie groupoid between $\nu_-(G^{(1)}_{i})\rightrightarrows \nu_-(G^{(0)}_i)$. To do so, it was mainly used the fact that the structural maps of our Lie groupoid are Riemannian submersions, the simplicial identities \eqref{DC1}, and the relation between the Hessian forms of $F_1$ and $F_0$. In consequence, by using Lemma \ref{DCL1} we may easily extend some of these arguments to construct a sub-nerve $\nu_-(G^{(\bullet)}_{i})$ of $\nu_-(G^{(\bullet)})$ where the face maps at every level are the collection of fiberwise isomorphisms $\overline{d(d_k^n)}$. Therefore, proceeding again by  induction over $n$ we will have that if $\nu_-(G^{(n-1)}_{i})$ is orientable, then so is $\nu_-(G^{(n)}_{i})$. Indeed, by following \cite[c. VII]{GHV} and \cite[c. 8]{Li1}, if $\iota_x:\nu_-(G^{(0)}_{i})_x\hookrightarrow \nu_-(G^{(0)}_{i})$ denotes the canonical inclusion of each fiber and $\omega$ is an $i$-form on $\nu_-(G^{(0)}_{i})$ inducing orientations $\iota_x^\ast\omega_x$ at every fiber $\nu_-(G^{(0)}_{i})_x$, then when pulling $\iota_x^\ast\omega_x$ back by the isomorphism $\overline{ds}_{g}:\nu_-(G^{(1)}_{i})_g\to \nu_-(G^{(0)}_{i})_x$ we get orientations $(\iota_x\circ \overline{ds}_{g})^\ast\omega_x=(\overline{ds}\circ \iota_g)^\ast\omega_x$ for each fiber $\iota_g: \nu_-(G^{(1)}_{i})_g\hookrightarrow \nu_-(G^{(1)}_{i})$ so that $\overline{ds}^\ast \omega$ is an $i$-form inducing an orientation on $\nu_-(G^{(1)}_{i})$. Hence, the other orientations are inductively obtained by using any of the fiberwise isomorphisms $\overline{d(d_k^n)}$.
\end{proof}
The first important consequence of the previous result is that because of (b), for every $n\in \mathbb{N}$ the moduli of gradient flow lines $\mathcal{M}^n(G^{(n)}_i,G^{(n)}_j)$ is a smooth manifold where the new endpoint maps
 \begin{equation}\label{DC4}
 	u^i_j(n):\mathcal{M}^n(G^{(n)}_i,G^{(n)}_j)\to G^{(n)}_i\quad\textnormal{and}\quad l^i_j(n):\mathcal{M}^n(G^{(n)}_i,G^{(n)}_j)\to G^{(n)}_j
 \end{equation}
defined by $u^i_j(n)([A]_n):=u_i(n)(A)$ and $l^i_j(n)([A]_n):=l_j(n)(A)$, are well defined smooth maps such that $u^i_j(n)$ has the structure of locally trivial bundle; see \cite{AB}. Therefore, as the action by flow translations is free and proper we get that the maps $\overline{d_k^n}$ are smooth and from Identities \eqref{DC3} we obtain that the following diagrams are commutative
 $$\xymatrix{
 	\mathcal{M}^n(G^{(n)}_i,G^{(n)}_j) \ar@<.5ex>[d]_{u^i_j(n)}\ar[r]^{\overline{d_k^n}\quad} & \mathcal{M}^{n-1}(G^{(n-1)}_i,G^{(n-1)}_j)  \ar@<-.5ex>[d]^{u^i_j(n-1)}\\
 	G^{(n)}_i \ar[r]_{d_k^n} & G^{(n-1)}_i
 }\qquad \xymatrix{
 	\mathcal{M}^n(G^{(n)}_i,G^{(n)}_j) \ar@<.5ex>[d]_{l^i_j(n)}\ar[r]^{\overline{d_k^n}\quad} & \mathcal{M}^{n-1}(G^{(n-1)}_i,G^{(n-1)}_j)  \ar@<-.5ex>[d]^{l^i_j(n-1)}\\
 	G^{(n)}_j \ar[r]_{d_k^n} & G^{(n-1)}_j
 }$$
 so that we have the commutative identities
 \begin{equation}\label{DC5}
 	d_k^n\circ u^i_j(n)=u^i_j(n-1)\circ \overline{d_k^n}\quad \textnormal{and} \quad d_k^n\circ l^i_j(n)=l^i_j(n-1)\circ \overline{d_k^n}.
 \end{equation}
In other words, we have obtained simplicial smooth maps

\begin{equation}\label{eq:nerveendpointmaps}
u^i_j(\bullet):\mathcal{M}^\bullet(G^{(\bullet)}_i,G^{(\bullet)}_j)\to G^{(\bullet)}_i \text{ and } l^i_j(\bullet):\mathcal{M}^\bullet(G^{(\bullet)}_i,G^{(\bullet)}_j)\to G^{(\bullet)}_j
\end{equation}
with $u^i_j(\bullet)$ a locally trivial simplicial fibration.

Let us define our double cochain complex. For that, we will define maps which are the composition of the pullback operation followed by the integration along the fibers via \eqref{eq:nerveendpointmaps}.


For each $n\in \mathbb{N}$ we set $C^{i,j}(G^{(n)}):=\Omega^j(G^{(n)}_i)$ and define the operator 

$\partial_r^n:C^{i,j}(G^{(n)})\to C^{i+r,j-r+1}(G^{(n)})$ as
 $$\partial_r^n(\omega)= \left\{ \begin{array}{lcc}
 	d\omega &   \textnormal{if}  & r=0 \\
 	 (-1)^j(u^{i+r}_i(n))_\ast(l^{i+r}_i(n))^\ast(\omega) &  \textnormal{otherwise,} &  
 \end{array}
 \right.$$
 where $(u^{i+r}_i(n))_\ast$ is integration along the fiber of the bundle \eqref{DC4}. Let us now set 
 
 \begin{equation}\label{eq:groupoidAB}
 C^{p}(G^{(n)}):=\bigoplus_{i+j=p}C^{i,j}(G^{(n)})=\bigoplus_{i+j=p}\Omega^j(G^{(n)}_i)\quad\textnormal{with}\quad \partial^n=\sum \partial_i^n.
 \end{equation}
 
\noindent As it was shown for the classical case in \cite{AB}, the operator $\partial^n$ is a boundary operator, i.e., $(\partial^n)^2=0$. Let us now consider the simplicial differentials of the sub-nerve $G^{(\bullet)}_i$ from Lemma \ref{DCL1}:
 $$\delta^n_i=\sum(-1)^k(d_k^n)^\ast:C^{i,j}(G^{(n-1)})\to C^{i,j}(G^{(n)}).$$
 
\begin{lemma}\label{KeyLemmaDoubleComplex}
 The following diagram commutes for all $i$ and $r$
 $$\xymatrix{
 	\Omega^j(G^{(n-1)}_i) \ar@<.5ex>[d]_{\delta^n_i}\ar[r]^{\partial_r^{n-1}\quad} & \Omega^{j-r+1}(G^{(n-1)}_{i+r})  \ar@<-.5ex>[d]^{\delta^{n}_{i+r}}\\
 	\Omega^j(G^{(n)}_i) \ar[r]_{\partial_r^n\quad} & \Omega^{j-r+1}(G^{(n)}_{i+r}).
 }$$
 \end{lemma}
\begin{proof}
Observe that if $r=0$ then
 $$\delta^n_{i+0}\circ \partial_0^{n-1} = \partial_0^{n}\circ\delta^n_{i}\Longleftrightarrow \delta^n_{i}\circ d = d\circ\delta^n_{i},$$
 since the simplicial differential already commute with the de Rham differentials. Otherwise,
 $$(\partial_r^n\circ \delta^n_{i})(\omega)=\partial_r^n\left(\sum(-1)^k(d_k^n)^\ast(\omega)\right)=\sum(-1)^k\partial_r^n((d_k^n)^\ast(\omega)),$$
 where, as consequence of the base-change formula of the integration along the fiber operation together with Identities \eqref{DC5}, we obtain
 \begin{eqnarray*}
 	\partial_r^n((d_k^n)^\ast(\omega)) &=& (-1)^j(u^{i+r}_i(n))_\ast(l^{i+r}_i(n))^\ast((d_k^n)^\ast(\omega))\\
 	&=& (-1)^j(u^{i+r}_i(n))_\ast(d_k^n\circ l^{i+r}_i(n))^\ast(\omega)\\
 	&=& (-1)^j(u^{i+r}_i(n))_\ast(l^{i+r}_i(n-1)\circ \overline{d_k^n})^\ast(\omega)\\
 	&=& (-1)^j(u^{i+r}_i(n))_\ast (\overline{d_k^n})^\ast ((l^{i+r}_i(n-1))^\ast(\omega))\\
 	&=& (-1)^j(d_k^n)^\ast(u^{i+r}_i(n-1))_\ast((l^{i+r}_i(n-1))^\ast(\omega))\\
 	&=& (d_k^n)^\ast((-1)^j(u^{i+r}_i(n-1))_\ast((l^{i+r}_i(n-1))^\ast(\omega)))\\
 	&=& (d_k^n)^\ast(\partial_r^{n-1}(\omega)).
 \end{eqnarray*}
 Thus, we have that
$$
 	(\partial_r^n\circ \delta^n_{i})(\omega)=\sum(-1)^k\partial_r^n((d_k^n)^\ast(\omega))=\sum(-1)^k(d_k^n)^\ast(\partial_r^{n-1}(\omega))=(\delta^n_{i+r}\circ \partial_r^{n-1})(\omega).
$$
\end{proof}

Having the previous fact in mind we define $\overline{\delta}^n:C^{p}(G^{(n-1)})\to C^{p}(G^{(n)})$ as
 $$\overline{\delta}^n(\omega)= \left\{ \begin{array}{lcc}
 	\delta^n_{i}(\omega) &   \textnormal{if}  & w\in 	\Omega^{j}(G^{(n-1)}_i) \\
 	 0 &  \textnormal{otherwise}. &  
 \end{array}
 \right.$$
 This operator may be thought of as $\displaystyle \overline{\delta}^n=\sum \delta^n_{i}$ verifying $\delta^n_{i}\circ \delta^n_{i'}=\delta^n_{i'}\circ \delta^n_{i}=0$ since if $i\neq i'$ then $\delta^n_{i}\circ \delta^n_{i'}=0$ by definition and if $i=i'$ then $(\delta^n_{i})^2=0$ also holds because $\delta^n_{i}$ is a simplicial differential of the sub-nerve $G^{(\bullet)}_i$. In particular, we have that $(\overline{\delta}^n)^2=0$. 
 
 Note that we have defined two boundary operators $\partial$ and $\overline{\delta}$ verifying $\partial^2=0$, $\overline{\delta}^2=0$ and, moreover, from the commutativity property stated in Lemma \ref{KeyLemmaDoubleComplex} we get that they also satisfy $\partial\circ \overline{\delta}=\overline{\delta}\circ \partial$.

Summing up, we have obtained that:
\begin{proposition}
The triple $(C^\bullet(G^{(\bullet)}),\partial,\overline{\delta})$ determines a double cochain complex which may be depicted as
 $$\begin{tikzpicture}[scale=1.0,baseline=-0.1cm, inner sep=1mm,>=stealth]
 	\node (1) at (0,0) {$C^0(G^{(0)})$};
 	\node (2) at (3,0) {$C^0(G^{(1)})$};
 	\node (3) at (6,0) {$C^0(G^{(2)})$};
 	\node (4) at (9,0) {$\cdots$};
 	\node (5) at (0,2) {$C^1(G^{(0)})$};
 	\node (6) at (3,2) {$C^1(G^{(1)})$};
 	\node (7) at (6,2) {$C^1(G^{(2)})$};
 	\node (8) at (9,2) {$\cdots$};
 	\node (9) at (0,4) {$C^2(G^{(0)})$};
 	\node (10) at (3,4) {$C^2(G^{(1)})$};
 	\node (11) at (6,4) {$C^2(G^{(2)})$};
 	\node (12) at (9,4) {$\cdots$};
 	\node (13) at (0,6) {$\vdots$};
 	\node (14) at (3,6) {$\vdots$};
 	\node (15) at (6,6)  {$\vdots$};
 	\tikzset{every node/.style={fill=white}} 
 	\draw[->,  thick] (1) to node[midway,above] {$\overline{\delta}$} (2); 
 	\draw[->,  thick] (2) to node[midway,above] {$\overline{\delta}$} (3);
 	\draw[->,  thick] (3) to node[midway,above] {$\overline{\delta}$} (4);
 	\draw[->,  thick] (5) to node[midway,above] {$\overline{\delta}$} (6); 
 	\draw[->,  thick] (6) to node[midway,above] {$\overline{\delta}$} (7);
 	\draw[->,  thick] (7) to node[midway,above] {$\overline{\delta}$} (8);
 	\draw[->,  thick] (9) to node[midway,above] {$\overline{\delta}$} (10); 
 	\draw[->,  thick] (10) to node[midway,above] {$\overline{\delta}$} (11);
 	\draw[->,  thick] (11) to node[midway,above] {$\overline{\delta}$} (12);
 	\draw[->,  thick] (1) to node[midway,left] {$\partial$} (5);
 	\draw[->,  thick] (5) to node[midway,left] {$\partial$} (9); 
 	\draw[->,  thick] (9) to node[midway,left] {$\partial$} (13);
 	\draw[->,  thick] (2) to node[midway,left] {$\partial$} (6);
 	\draw[->,  thick] (6) to node[midway,left] {$\partial$} (10); 
 	\draw[->,  thick] (10) to node[midway,left] {$\partial$} (14);
 	\draw[->,  thick] (3) to node[midway,left] {$\partial$} (7);
 	\draw[->,  thick] (7) to node[midway,left] {$\partial$} (11); 
 	\draw[->,  thick] (11) to node[midway,left] {$\partial$} (15);
 \end{tikzpicture}$$
\end{proposition} 

We can make a total complex out of this double cochain complex by setting  $\displaystyle C^n_{T}(G)=\bigoplus_{p+q=n}C^p(G^{(q)})$ and defining the total differential $\partial_{T}:C^n_{T}(G)\to C^{n+1}_{T}(G)$ as

$$\partial_{T}(\omega)=(\overline{\delta}+(-1)^q\partial)(\omega),\quad \omega\in C^p(G^{(q)}).$$
The sing change is introduced in order that $\partial_{T}^2=0$.

\begin{definition}
The \textbf{groupoid Morse cohomology} associated to the Morse Lie groupoid $F_1:(G^{(1)}\rightrightarrows G^{(0)})\to (\mathbb{R}\rightrightarrows \mathbb{R})$ and the $n$-metric $\eta^{(n)}$ on $G^{(n)}$ is defined to be the cohomology of $(C^\bullet_T(G),\partial_{T})$.
\end{definition}
 
Finally, we exhibit a morphism of double complexes between the Bott--Shulman-Stasheff double complex of $G^{(1)} \rightrightarrows G^{(0)}$ and the double complex constructed above. The \textbf{Bott-Shulman-Stasheff double complex} $(\Omega^\bullet(G^{(\bullet)}),d,\delta)$ may be depicted as
 	$$\begin{tikzpicture}[scale=1.0,baseline=-0.1cm, inner sep=1mm,>=stealth]
 		\node (1) at (0,0) {$\Omega^0(G^{(0)})$};
 		\node (2) at (3,0) {$\Omega^0(G^{(1)})$};
 		\node (3) at (6,0) {$\Omega^0(G^{(2)})$};
 		\node (4) at (9,0) {$\cdots$};
 		\node (5) at (0,2) {$\Omega^1(G^{(0)})$};
 		\node (6) at (3,2) {$\Omega^1(G^{(1)})$};
 		\node (7) at (6,2) {$\Omega^1(G^{(2)})$};
 		\node (8) at (9,2) {$\cdots$};
 		\node (9) at (0,4) {$\Omega^2(G^{(0)})$};
 		\node (10) at (3,4) {$\Omega^2(G^{(1)})$};
 		\node (11) at (6,4) {$\Omega^2(G^{(2)})$};
 		\node (12) at (9,4) {$\cdots$};
 		\node (13) at (0,6) {$\vdots$};
 		\node (14) at (3,6) {$\vdots$};
 		\node (15) at (6,6)  {$\vdots$};
 		\tikzset{every node/.style={fill=white}} 
 		\draw[->,  thick] (1) to node[midway,above] {$\delta$} (2); 
 		\draw[->,  thick] (2) to node[midway,above] {$\delta$} (3);
 		\draw[->,  thick] (3) to node[midway,above] {$\delta$} (4);
 		\draw[->,  thick] (5) to node[midway,above] {$\delta$} (6); 
 		\draw[->,  thick] (6) to node[midway,above] {$\delta$} (7);
 		\draw[->,  thick] (7) to node[midway,above] {$\delta$} (8);
 		\draw[->,  thick] (9) to node[midway,above] {$\delta$} (10); 
 		\draw[->,  thick] (10) to node[midway,above] {$\delta$} (11);
 		\draw[->,  thick] (11) to node[midway,above] {$\delta$} (12);
 		\draw[->,  thick] (1) to node[midway,left] {$d$} (5);
 		\draw[->,  thick] (5) to node[midway,left] {$d$} (9); 
 		\draw[->,  thick] (9) to node[midway,left] {$d$} (13);
 		\draw[->,  thick] (2) to node[midway,left] {$d$} (6);
 		\draw[->,  thick] (6) to node[midway,left] {$d$} (10); 
 		\draw[->,  thick] (10) to node[midway,left] {$d$} (14);
 		\draw[->,  thick] (3) to node[midway,left] {$d$} (7);
 		\draw[->,  thick] (7) to node[midway,left] {$d$} (11); 
 		\draw[->,  thick] (11) to node[midway,left] {$d$} (15);
 	\end{tikzpicture}$$
 	where the vertical differential is the de Rham differential and the horizontal differential is the simplicial differential associated to the nerve $G^{(\bullet)}$ of $G^{(1)} \rightrightarrows G^{(0)}$. We can similarly define a total cohomology for this double complex which is usually denoted by $H^{\bullet}_{dR}(G)$.
 	
 	Let us consider the collection of maps $\Psi^\bullet:=\lbrace \Psi^n \rbrace_{n\in\mathbb{N}}$ constructed as follows. Recall that for each $n\in\mathbb{N}$ the map $u_i(n):W^u(G^{(n)}_i)\to G^{(n)}_i$ has the structure of locally trivial bundle. Thus, by using again integration along the fiber, we define $\Psi^n_i:\Omega^{p}(G^{(n)})\to C^{i,p-i}(G^{(n)})$ as
 	$$\Psi^n_i(\omega):=(u_i(n))_\ast \left(\omega|_{W^u(G^{(n)}_i)}\right)\qquad \omega\in \Omega^{p}(G^{(n)}),$$
 	and $\Psi^n=\bigoplus\Psi^n_i:\Omega^{p}(G^{(n)})\to C^p(G^{(n)})$. As consequence of the results due to Austin and Braam in \cite{AB} we have that $\Psi^\bullet\circ d=\partial \circ\Psi^\bullet$ and, more importantly, this collection of maps induces isomorphisms between the cohomology groups
 	$$H^\bullet(C^\bullet(G^{(n)}),\partial)\cong H^\bullet_{dR}(\Omega^\bullet(G^{(n)}),d).$$
 	To see that the collection $\Psi^\bullet$ defines a morphism of double complexes it remains to check that $\Psi^\bullet\circ \delta=\overline{\delta} \circ\Psi^\bullet$. This will be consequence of showing the following identity.
 	\begin{lemma}
 	For all $i$ and $n$, the following holds
 	$$\Psi^n_i\circ \delta =\overline{\delta}\circ \Psi^{n-1}_i.$$
 	\end{lemma}
 \begin{proof}
 Because of Identities \eqref{DC3}, which in turn allow us to conclude that every face map $d_k^n$ is well restricted to the stable/unstable manifolds, and by using the base-change property of the integration along the fiber operation, we obtain that 
 \begin{eqnarray*}
 	(\Psi^n_i\circ \delta)(\omega) &=& \Psi^n_i\left(\sum(-1)^k (d_k^n)^\ast(\omega)\right)= \sum(-1)^k \Psi^n_i((d_k^n)^\ast(\omega))\\
 	&=& \sum(-1)^k (u_i(n))_\ast \left((d_k^n)^\ast(\omega)|_{W^u(G^{(n)}_i)}\right)\\
 	&=& \sum(-1)^k (d_k^n)^\ast(u_i(n-1))_\ast \left(\omega|_{W^u(G^{(n-1)}_i)}\right)= (\overline{\delta}\circ \Psi^{n-1}_i)(\omega),
 \end{eqnarray*}
for all $\omega\in \Omega^{p}(G^{(n-1)})$ as desired.
 \end{proof}
So, $\Psi^\bullet$ defines a morphism of double complexes between $(\Omega^\bullet(G^{(\bullet)}),d,\delta)$ and $(C^\bullet(G^{(\bullet)}),\partial,\overline{\delta})$ inducing isomorphisms between the vertical cohomologies of the complexes. Therefore, as consequence of all the facts stated above, by using a usual argument of spectral sequences (see for instance \cite[p. 108]{BT}) we conclude:

 \begin{theorem}\label{MorseCohomology}
 	The total cohomology of the double complex $(C^\bullet(G^{(\bullet)}),\partial,\overline{\delta})$ is isomorphic to the total cohomology of the Bott--Shulman--Stasheff double complex of $G^{(1)} \rightrightarrows G^{(0)}$:
 	$$H_{T}^\bullet(G,\partial_{T})\cong H^{\bullet}_{dR}(G).$$
 \end{theorem}

In particular, the total cohomology on the left hand side is Morita invariant, something that should be expected since the notion of Morse Lie groupoid morphism we are working with is also Morita invariant. Additionally, by arguing as before, it is simple to check that if $G\leftarrow H \rightarrow G'$ is a Morita equivalence between compact proper Lie groupoids such that $G^{(n)}$, $H^{(n)}$, and $G'^{(n)}$ are orientable for all $n\in \mathbb{N}$ then Assumption \ref{StrongAssumption} is preserved by Proposition \ref{MoritaInvariance}. This key fact can be used to compute the Bott--Shulman-Stasheff cohomology of certain Lie groupoids by using elementary examples of Morse Lie groupoid morphisms.

\begin{example}
Let $M$ be a smooth manifold and $f:M\to \mathbb{R}$ be a smooth function satisfying Assumption \ref{ABAssumption}. From \cite{AB} it follows that the groupoid Morse complex associated to the induced Morse Lie groupoid morphism $F:(M\rr M) \to (\mathbb{R}\rr \mathbb{R})$ is the standard Morse complex of $(M,f)$. Also, the Bott--Shulman--Stasheff complex of $M\rr M$ is the ordinary de Rham complex of $M$. Therefore, as explained above, the computation of the Morse complex can be used to obtain the Bott--Shulman-Stasheff cohomology of those compact and orientable Lie groupoids which are Morita equivalent to $M$. See \cite[s. 3.6]{AB} for explicit computations in the standard case.
\end{example}

\begin{example}
Suppose that $G$ is a compact Lie group acting on a compact oriented manifold $M$. Let $f:M\to \mathbb{R}$ be a $G$-invariant Morse function satisfying the Morse--Smale transversality condition. The Bott--Shulman-Stasheff cohomology of the action groupoid $G\ltimes M\rr M$ is isomorphic to $G$-equivariant cohomology associated to the Cartan complex, i.e. the equivariant cohomology of the action \cite{Behrend1}. Once again, from \cite{AB} we get that the groupoid Morse complex of the induced Morse Lie groupoid morphism $F:(G\ltimes M\rr M) \to (\mathbb{R}\rr \mathbb{R})$ recovers the standard $G$-equivariant cohomology of the action. This fact can be viewed as a particular instance of the construction in Subsection \ref{SubEquivariant}. Consult \cite[s. 5.4]{AB} and \cite[s. 5]{L} for explicit computations. 
\end{example}

\begin{example}
If $G\rr M$ is a proper and \'etale Lie groupoid then its Bott--Shulman-Stasheff cohomology agrees with the basic cohomology of $G$, which turns out to be isomorphic to the singular cohomology $H^\bullet(M/G,\mathbb{R})$, compare \cite{PPT,TuX}. Thus, Morse Lie groupoid morphisms on $G\rr M$ verifying Assumption \ref{StrongAssumption} allow us to recover the singular cohomology of the orbit space $M/G$. A particularly interesting instance of this fact was recently checked in \cite{ChoHong,A-liDu} for the case of effective orientable orbifolds by providing an explicit isomorphism. Let $G$ be a compact connected Lie group acting on a compact oriented manifold $M$ by orientation preserving diffeomorphisms. If the action is effective and locally free then the quotient space $M/G$ has the structure of an effective orbifold, so that the differentiable stack $[M/G\ltimes M]$ represented by the action groupoid $G\ltimes M\rr M$ can be thought of as an orientable orbifold. Hence, the groupoid Morse cohomology associated to a Morse Lie groupoid morphism $(G\ltimes M\rr M) \to (\mathbb{R}\rr \mathbb{R})$ verifying the Morse--Smale transversality condition equals the singular cohomology $H^\bullet(M/G,\mathbb{R})$ since the action is locally free, see \cite{A-liDu}. 
\end{example}

\begin{remark}
Fukaya in \cite{Fu} gave a construction which is similar to Austin--Braam's model, but using singular chains instead of differential forms. One of the main features of his construction is that it can be directly generalized to the infinite-dimensional case so that it allowed him to address problems in Floer homology. We plan to apply Fukaya's approach to our setting with the hope of starting the study of Floer homology in the context of Lie groupoids and differentiable stacks.
\end{remark}

\subsection{2-Equivariant cohomology}\label{SubEquivariant}

Under certain some additional assumptions and motivated by the Austin--Braam's equivariant Morse complex (see \cite[s. 5.1]{AB}), we apply the constructions from the previous section to recover the equivariant cohomology associated to a $2$-action on a Lie groupoid as defined in \cite{OBT}. Let us consider a Lie groupoid $G^{(1)} \rightrightarrows G^{(0)}$ and a Lie groupoid morphism $F_1:(G^{(1)} \rightrightarrows G^{(0)})\to (\mathbb{R}\rightrightarrows \mathbb{R})$ induced by a basic function $F_0:G^{(0)}\to \mathbb{R}$. Suppose that $K^{(1)} \rightrightarrows K^{(0)}$ is a Lie $2$-group, with $K^{(1)}$ compact, determining a $2$-action on $G^{(1)} \rightrightarrows G^{(0)}$ such that $F_0$ is $K^{(0)}$-invariant. We further assume that $F_1$ is a Morse--Bott Lie groupoid morphism (see Remark \ref{EquiRemark}), satisfying the assumptions we considered in the previous section. The set of $n$-composable arrows $K^{(n)}$ inherits a canonical Lie group structure from the direct product $(K^{(1)})^n$ and the $2$-action above allows us to define a smooth left action of $K^{(n)}$ on $G^{(n)}$ in a canonical way. In other words, we have a well defined simplicial left action of the nerve $K^{(\bullet)}$ on the nerve $G^{(\bullet)}$. It is simple to check that the face maps $(d_k^n)_K:K^{(n)}\to K^{(n-1)}$ and $(d_k^n)_G:G^{(n)}\to G^{(n-1)}$ satisfy the equivariant relations

\begin{equation}\label{EquiRelations}
(d_k^n)_G(k\cdot g)=(d_k^n)_K(k)\cdot (d_k^n)_G(g),
\end{equation}
for all $k\in K^{(n)}$ and $g\in G^{(n)}$. This is consequence of the simplicial identities on $G^{(\bullet)}$ and the fact that we are working with a $2$-action. Therefore, Formula \eqref{EquiRelations} implies that the simplicial function $F_\bullet$ is $K^{(\bullet)}$-invariant in the sense that $F_n$ is $K^{(n)}$-invariant for all $n$ since $F_0$ is $K^{(0)}$-invariant. Let us now consider an $n$-metric $\eta^{(n)}$ on $G^{(n)}$.

We need to use a notion of isometric Lie $2$-group action on a Riemannian groupoid which was introduced in \cite{HV}.

\begin{definition}\label{invariantmetric}
The $2$-action of $K^{(1)} \rightrightarrows K^{(0)}$ on $G^{(1)} \rightrightarrows G^{(0)}$ is said to be \textbf{isometric} if the action of $K^{(n)}$ on $(G^{(n)},\eta^{(n)})$ is by isometries.
\end{definition}

A detailed study of several geometric aspects concerning isometric $2$-actions as well as the structure of Riemannian groupoids (for instance, existence, 2-equivariant groupoid linearizations, 2-equivariant tubular neighborhoods, transversal isometries, Lie $2$-algebra of weak Killing vector fields, among other things), are provided in \cite{HV}. As it was proven therein, the fact that $K^{(n)}$ acts on $(G^{(n)},\eta^{(n)})$ isometrically is enough to ensure that the action of $K^{(r)}$ on $(G^{(r)},\eta^{(r)})$ is isometric for all $0\leq r\leq n-1$. For our purposes, it is important to mention that when $K^{(1)}$ is compact and $G^{(1)} \rightrightarrows G^{(0)}$ is proper then invariant $n$-metrics in the sense of Definition \ref{invariantmetric} there always exist \cite{HV}. So, let us assume that our $2$-action is by isometries. Given that the functions $F_n$ are $K^{(n)}$-invariant it follows that their negative gradient vector fields $-\nabla F_n$ are $K^{(n)}$-invariant so that their descending flows $\Phi^{n}_\tau$ are $K^{(n)}$-equivariant. Note that the saturated submanifold $G^{(0)}_i$ which is formed by the non-degenerate groupoid orbits of $G^{(1)}\rightrightarrows G^{(0)}$ having index $i$ is $K^{(0)}$-invariant since the fact that $F_0$ is $K^{(0)}$-invariant implies that $\textnormal{Crit}(F_0)$ is formed by Lie group $K^{(0)}$-orbits. In consequence, it is simple to check that for each $n\in\mathbb{N}$ it holds that $G^{(n)}_i$ is $K^{(n)}$-invariant, thus obtaining that the endpoint maps $u_i(n)$ and $l_i(n)$ are $K^{(n)}$-equivariant since our descending flows are $K^{(n)}$-equivariant. More importantly, if we consider the left action of $K^{(n)}$ on $\mathcal{M}^n(G^{(n)}_i,G^{(n)}_j)$ defined by $k\cdot [A]_n=[k\cdot A]_n$ then we get two new $K^{(n)}$-equivariant endpoint maps $u^i_j(n)$ and $l^i_j(n)$.

\begin{remark}
As it was commented in \cite{AB}, it follows that due to the $K^{(0)}$-equivariance of $F_0$ our weakly self-indexing requirement from Assumption \ref{StrongAssumption} is implied by the transversality assumption. Similarly, the assumption asking the endpoint maps to induce fibrations is an immediate consequence of the presence of a transitive $K^{(0)}$-action on the components of the critical point set.
\end{remark}

Let us briefly introduce the notion of equivariant cohomology associated to a Lie $2$-group action on a Lie groupoid as defined in \cite{OBT}. The $2$-action naturally allows us to define a \textbf{double Lie groupoid}
\begin{eqnarray*}\label{eq:doublegroupoid}
	\xymatrix{
		K^{(1)}\times G^{(1)} \ar@<-0.5ex>[d]\ar@<0.5ex>[d]\ar@<0.5ex>[r] \ar@<-0.5ex>[r]& G^{(1)} \ar@<-0.5ex>[d]\ar@<0.5ex>[d] \\
		K^{(0)}\times G^{(0)} \ar@<0.5ex>[r]\ar@<-0.5ex>[r]                               & G^{(0)},
	}
\end{eqnarray*}
where the horizontals are given by action groupoids and the verticals are Lie groupoid products. If we consider the nerve configuration associated to this double Lie groupoid then we obtain a \textbf{bisimplicial smooth manifold} so that we may work with the triple complex $C^{\bullet,\bullet,\bullet}$ where
$$C^{n,p,q}=\Omega^p((K^{(n)})^q\times G^{(n)}), $$
with differentials given by the de Rham differential, the simplicial differential associated to the actions groupoids, and the simplicial differential associated to the product groupoids.

\begin{definition}\cite{OBT}
The \textbf{equivariant cohomology} of the $2$-action of $K^{(1)}\rightrightarrows K^{(0)}$ on $G^{(1)}\rightrightarrows G^{(0)}$ is defined to be the total cohomology determined by the triple complex mentioned above. This will be denoted by $H^\bullet_{K}(G)$.
\end{definition}

Two important features of this cohomology is that it is Morita invariant and can be recovered by the \textbf{Cartan model} as follows. Let us consider the nerve configuration $\mathfrak{k}^{(\bullet)}$ associated to the Lie $2$-algebra $\mathfrak{k}^{(1)} \rightrightarrows \mathfrak{k}^{(0)}$ of $K^{(1)}\rightrightarrows K^{(0)}$. The notion of simplicial equivariant forms introduced in \cite[Appx. C]{Me} comes from a double complex $(C^{\bullet,\bullet}_{CM}, d_K,\delta_K)$ given by

\begin{equation}\label{eq:simplicialCartan}
C^{n,k}_{CM}=\Omega^k_{K^{(n)}}(G^{(n)})= \bigoplus_{k=2p+q}(S^p((\mathfrak{k}^{(n)})^\ast)\otimes \Omega^q(G^{(n)}))^{K^{(n)}},
\end{equation}

\noindent where $(\mathfrak{k}^{(n)})^\ast$ denotes the dual vector space of the Lie algebra $\mathfrak{k}^{(n)}$, $d_K$ is the Cartan differential, and $\delta_K: \Omega^k_{K^{(n)}}(G^{(n)})\to \Omega^k_{K^{(n+1)}}(G^{(n+1)})$ is defined by $\delta_K=\delta_{\mathfrak{k}}\otimes \delta_{G}$ with $\delta_{\mathfrak{k}}$ the simplicial differential of $\mathfrak{k}^{(\bullet)}$ and $\delta_{G}$ the simplicial differential of $G^{(\bullet)}$. As it was proven in \cite{OBT}, the total cohomology of this double complex is isomorphic to the equivariant cohomology $H^\bullet_{K}(G)$ since $K^{(1)}$ is assumed to be compact.

We claim that it is possible to recover the equivariant cohomology defined above by following the ideas from \cite{AB} together with what we did in the previous section. For that, we consider a Morse--Bott Lie groupoid morphism $F_1:G^{(1)}\to \mathbb{R}$ covering a basic function $F_0:G^{(0)}\to \mathbb{R}$ which is $K^{(0)}$-invariant.  Using the Cartan model we define the equivariant version of the double complex \eqref{eq:groupoidAB}:

\begin{equation}\label{eq:equivariantgroupoidAB}
C^{p}(G^{(n)})=\bigoplus_{i+r=p}\Omega^r_{K^{(n)}}(G_i^{(n)})=\bigoplus_{i+j+2v=p}(\Omega^j(G_i^{(n)})\otimes S^v((\mathfrak{k}^{(n)})^\ast))^{K^{(n)}},
\end{equation}
with differential operators $\partial_K^n$ and $\overline{\delta}_K^n$ defined as follows. On the one hand, as before we split $\partial_K^n:C^{p}(G^{(n)})\to C^{p+1}(G^{(n)})$ as the sum $\partial_K^n=\sum_v(\partial_K^n
)_v$ where, for $\omega\otimes \phi\in (\Omega^j(G_i^{(n)})\otimes S^v((\mathfrak{k}^{(n)})^\ast))^{K^{(n)}}$, we have that $(\partial_K^n
)_0(\omega\otimes \phi)=d_K(\omega\otimes \phi)$ is the Cartan differential and for $v>0$ we set $(\partial_K^n
)_v(\omega\otimes \phi)=\partial_v^n\omega\otimes \phi$. On the other hand, we define $\overline{\delta}_K^n:C^{p}(G^{(n)})\to C^{p}(G^{(n+1)})$ as $\overline{\delta}_K^n(\omega\otimes \phi)=\overline{\delta}^n\omega\otimes \delta_{\mathfrak{g}}\phi$. It is simple to check that these two operators $\partial_K$ and $\overline{\delta}_K$ commute and that $\overline{\delta}_K^2=0$. Moreover, from \cite{AB} we also get that $\partial_K^2=0$. Therefore, we have actually obtained a double cochain complex $(C^\bullet(G^{(\bullet)}),\partial_K,\overline{\delta}_K)$, as claimed above. 

Finally, let us now exhibit a morphism of double complexes between the double complex $(C^{\bullet,\bullet}_{CM}, d_K,\delta_K)$ which is obtained by using the Cartan model \eqref{eq:simplicialCartan} and $(C^\bullet(G^{(\bullet)}),\partial_K,\overline{\delta}_K)$ defined in \eqref{eq:equivariantgroupoidAB}. Let $\Theta^\bullet:C^\bullet(G^{(\bullet)})\to C^{\bullet,\bullet}_{CM}$ be formed by a collection of maps $\lbrace \Theta^n \rbrace_{n\in\mathbb{N}}$ defined by $\Theta^n(\omega\otimes \phi)=\Psi^n(\omega)\otimes \phi$. From a straightforward computation it follows that $\Theta^\bullet\circ \delta_K=\overline{\delta}_K \circ\Theta^\bullet$. Also, as a consequence of what it was proven in \cite{AB} we have that $\Theta^\bullet\circ d_K=\partial_K \circ\Theta^\bullet$, and, more importantly, this collection of maps induces isomorphisms between the cohomology groups
$$H^\bullet(C^\bullet(G^{(n)}),\partial_K)\cong H^\bullet(\Omega_{K^{(n)}}^\bullet(G^{(n)}),d_K)=H^\bullet_{K^{(n)}}(G^{(n)}),\qquad n\in \mathbb{N}.$$
Here $H^\bullet_{K^{(n)}}(G^{(n)})$ denotes the equivariant cohomology obtained through the Cartan model associated to the action of $K^{(n)}$ on $G^{(n)}$. So, $\Theta^\bullet$ defines a morphism of double complexes between $(C^{\bullet,\bullet}_{CM}, d_K,\delta_K)$ and $(C^\bullet(G^{(\bullet)}),\partial_K,\overline{\delta}_K)$ inducing isomorphisms between the vertical cohomologies of the complexes. Just as in the non-equivariant case, by means of a spectral sequence argument, we conclude that:

\begin{proposition}\label{EquivariantMorseCohomology}
	The total cohomology $H_{T}^\bullet(G,\partial_{TK})$ of the double complex $(C^\bullet(G^{(\bullet)}),\partial_K,\overline{\delta}_K)$ is isomorphic to the total cohomology of double complex determined by the Cartan model associated to the $2$-action of $K^{(1)} \rightrightarrows K^{(0)}$ on $G^{(1)} \rightrightarrows G^{(0)}$. That is,
	$$H_{T}^\bullet(G,\partial_{TK})\cong H^\bullet_{K}(G).$$
\end{proposition}

Toric symplectic stacks \cite{hoffman} are presented by $0$-symplectic groupoids with a Hamiltonian action of a 2-torus \cite{hsz}. Hence, as an application of Proposition \ref{EquivariantMorseCohomology} one may compute the equivariant cohomology of a toric symplectic stack by means of groupoid Morse theory.

\begin{example}[\textbf{Equivariant cohomology of toric symplectic stacks}]\label{EToricStack}

Let  $G^{(1)} \rightrightarrows G^{(0)}$ be a 0-symplectic groupoid equipped with a Hamiltonian $(K^{(1)} \rightrightarrows K^{(0)})$-groupoid 2-action with moment map $\mu$ verifying Proposition \ref{MomentMapsMorse}. For every $\xi\in \mathfrak{k}/\mathfrak{h}$ we have a Morse Lie groupoid morphism $\mu^\xi:(G^{(1)} \rightrightarrows G^{(0)})\to (\mathbb{R} \rightrightarrows \mathbb{R})$ for which $\mu_0^\xi$ is $K^{(0)}$-invariant since $K^{(0)}$ is abelian. Therefore, if there exists $\xi\in \mathfrak{k}/\mathfrak{h}$ such that $\mu^\xi$ verifies Assumption \ref{StrongAssumption}, then Proposition \ref{EquivariantMorseCohomology} allows to compute the equivariant cohomology associated to the $2$-action of the foliation Lie $2$-group $(K^{(1)} \rightrightarrows K^{(0)})$ on the $0$-symplectic groupoid $G^{(1)} \rightrightarrows G^{(0)}$ by using the equivariant version of the groupoid Morse cohomology. However, this is in general a difficult task since powerful tools as a 2-equivariant version of the Atiyah--Bott localization theorem are necessary to compute 2-equivariant cohomology groups \cite{OBT}. Let us visualize this situation in the simplest case. Let $(G^{(0)},\omega)$ be a compact $S^1$-Hamiltonian pre-symplectic manifold with moment map $\mu:G^{(0)}\to \mathbb{R}$ satisfying the Morse--Smale transversality condition, see \cite{LS,RatiZung}. Assume that the $S^1$ action is clean and denote by $\mathcal{F}$ the foliation of $G^{(0)}$ associated to $\textnormal{ker}(\omega)$. From \cite{LS} we know that the critical point set of $\mu$ equals the set of fixed leaves $\mathcal{L}$ of $\mathcal{F}$ by the $S^1$ action
$$\textnormal{Crit}(\mu):=G^{(0)}(\mathcal{F})^{S^1}=\lbrace \mathcal{L}: k\cdot \mathcal{L}=\mathcal{L}\ \textnormal{for all}\ k\in S^{1}\rbrace=\lbrace x\in G^{(0)}:\partial_\theta(x)\in T_x\mathcal{F}\rbrace.$$
Here $\partial_\theta$ stands for the fundamental vector field of the $S^1$ action. We denote the set of leaves of index $i$ by $G^{(0)}(\mathcal{F})^{S^1}_i$. Recall that $i$ is always even. Let us check that $(\partial_{S^1}^0)_v=0$ for all $v>0$, provided that the $S^1$ action fixes the points of the critical leaves. Firstly, the moduli space $\mathcal{M}^{0}(G^{(0)}(\mathcal{F})^{S^1}_{j+v},G^{(0)}(\mathcal{F})^{S^1}_j)$ inherits an $S^1$ action which commutes with the endpoint maps. But such an action fixes the endpoints, so that
$$\iota_{\partial_\theta}(l(0)^{j+v}_j)^\ast (\beta)=(l(0)^{j+v}_j)^\ast \iota_{\partial_\theta}(\beta)=0,\qquad \beta\in \Omega^\bullet(G^{(0)}(\mathcal{F})^{S^1}_j).$$
Secondly, the vector field $\partial_\theta$ is tangent to the fiber of $u(0)^{j+v}_j:\mathcal{M}^{0}(G^{(0)}(\mathcal{F})^{S^1}_{j+v},G^{(0)}(\mathcal{F})^{S^1}_j)\to G^{(0)}(\mathcal{F})^{S^1}_{j+v}$ and therefore
$$\partial_v^0(\beta)=(u(0)^{j+v}_j)_\ast (l(0)^{j+v}_j)^\ast(\beta)=0.$$
Hence, $\partial^{(0)}_{S^1}=(\partial^{(0)}_{S^1})_0$ and $H_{S^1}^\bullet(G^{0},\partial^{(0)}_{S^1})=\bigoplus_{i}H_{S^1}^{\bullet-i}(G^{(0)}(\mathcal{F})^{S^1}_i)$. This is exactly the equivariant Morse complex associated the unit Lie 2-group $S^1\rr S^1$ acting over the unit 0-symplectic groupoid $G^{(0)}\rr G^{(0)}$.

We can use the previous data to obtain a more interesting Hamiltonian Lie groupoid. Indeed, Theorem 7.2.1 in \cite{hsz} shows how to build a Hamiltonian Lie groupoid from a Hamiltonian pre-symplectic manifold. In the specific case described above we recover the unit Lie 2-group $S^1\rr S^1$ and any source-connected foliation Lie groupoid $G^{(1)}\rr G^{(0)}$ integrating $\mathcal{F}$. That is to say, a Lie groupoid whose orbits in $G^{(0)}$ agree with the leaves of $\mathcal{F}$ and the Lie 2-group action of $S^1\rr S^1$ on $G^{(1)}\rr G^{(0)}$ is obtained by extending that of $S^1$ on $G^{(0)}$. In particular, one may choose $G^{(1)}$ to be the Holonomy groupoid associated to the foliation $\mathcal{F}$. Additionally, we get an $S^1$-invariant Morse--Bott Lie groupoid morphism $\mu_1:G^{(1)}\to \mathbb{R}$ by setting $\mu_1=s^\ast \mu$. We endow $G^{(1)}\rr G^{(0)}$ with a groupoid Riemannian metric, so that $\mathcal{F}$ becomes a Riemannian foliation. This implies that $G^{(1)}$ is proper since each leaf has a finite holonomy group (see Theorem 2.6 and page 141 in \cite{MM}). Hence, under these assumptions our constructions apply and we can similarly check that $(\partial_{S^1}^n)_v=0$ for all $n\geq 1$ and $v>0$.

\end{example}

\end{document}